\documentclass[11pt]{amsart}

\usepackage{graphicx,color,amssymb, amsmath, amsthm, mathrsfs}
\usepackage{geometry}
\usepackage[all]{xy}
\usepackage{amssymb}
\usepackage{slashed}
\usepackage{mathtools}
\usepackage[colorlinks]{hyperref}
\hypersetup{citecolor=blue}
\usepackage{todonotes}  
\usepackage{mathdots}
\usepackage{tikz}
\usetikzlibrary{matrix}

\numberwithin{equation}{section}
\newtheorem{theorem}{Theorem}[section]
\newtheorem{lemma}[theorem]{Lemma}
\newtheorem{proposition}[theorem]{Proposition}
\newtheorem{corollary}[theorem]{Corollary}
\newtheorem{definition}[theorem]{Definition}

\usepackage{calrsfs}

\usepackage[T1]{fontenc} 
\usepackage{textcomp}
\usepackage{times}
 \usepackage[scaled=0.92]{helvet} 

\renewcommand{\tilde}{\widetilde} 
\newcommand{\End}{\textnormal{End}}
\newcommand{\C}{\mathbf{C}}
\def\P{{\mathbb P}}
\def\psl{{\bf PSL}_2}
\def\pgl{{\bf PGL}_2}
\newcommand{\R}{\mathbf{R}}
\newcommand{\Q}{\mathbf{Q}}
\newcommand{\Z}{\mathbf{Z}}

\newcommand{\K}{\mathbf{K}}
\newcommand{\h}{\mathbf{H}}

\newcommand{\im}{\mathrm{im}}

\newcommand{\Hom}{\mathrm{Hom}}

\newcommand{\coker}{\mathrm{coker}}

\newcommand{\Dom}{\textnormal{Dom }}

\newcommand{\dM}{\partial\overline{M}}
\newcommand{\dd}{{\bf\textnormal{d}}}
\newcommand{\supp}{\textnormal{supp } }
\newcommand{\nv}{\mathbf{n}}
\renewcommand{\phi}{\varphi}    
\renewcommand{\epsilon}{\varepsilon}    
\renewcommand{\geq}{\geqslant }
\renewcommand{\leq}{\leqslant }

\newcommand{\lx}{\overline{{\bf x}}}
\newcommand{\Lip}{\textnormal{Lip}}
\newcommand{\grad}{\textnormal{grad}}
\newcommand{\bea}          {\begin{eqnarray}}
\newcommand{\eea}          {\end{eqnarray}}
\newcommand{\beastar}          {\begin{eqnarray*}}
\newcommand{\eeastar}          {\end{eqnarray*}}

\newcommand{\geo}{\textnormal{geo}}

\newcommand{\diag}{\textnormal{diag}}
\newcommand{\dH}{\partial\mathbf{H}}
\newcommand{\ep}{\varepsilon}
\newcommand{\Ext}{\textnormal{Ext}}
\newcommand{\Dsla}{\slashed{D}}

\newcommand{\Ad}{\textnormal{Ad}}
\newcommand{\pt}{\textnormal{pt}}
\newcommand{\Eul}{\textnormal{Eul}}
\newcommand{\intN}{\mathring{N}}
\newcommand{\alg}{\textnormal{alg}}
\newcommand{\x}{{\bf x}}
\newcommand{\e}{{\bf e}}
\newcommand{\Tsla}{T}
\usepackage{amsmath,amssymb,amsfonts,amsthm}
\usepackage{thmtools}
\declaretheoremstyle[notefont=\bfseries,notebraces={}{},%
    headpunct={},postheadspace=1em]{mystyle}
\declaretheorem[style=mystyle,numbered=no,name=Theorem]{thm-hand}

\DeclareFontFamily{OT1}{pzc}{}
\DeclareFontShape{OT1}{pzc}{m}{it}{<-> s * [1.20] pzcmi7t}{}
\DeclareMathAlphabet{\mathpzc}{OT1}{pzc}{m}{it}

\begin{document}
 
\date{\today}
\title[Hecke operators and $K$-homology of Bianchi groups]{Hecke operators in $KK$-theory and \\
 the $K$-homology of Bianchi groups}

\author[B. Mesland]{Bram Mesland}
\address{\normalfont{Institut f\"{u}r Analysis, Leibniz Universit\"{a}t Hannover, Welfengarten 1, 30167,
Hannover, Germany}}
\email{mesland@math.uni-hannover.de}

\author[M.H. \c{S}eng\"un]{Mehmet Haluk \c{S}eng\"un}
\address{\normalfont{University of Sheffield,
School of Mathematics and Statistics, Hicks Building,
Sheffield, UK}}
\email{m.sengun@sheffield.ac.uk}

\begin{abstract} 
\noindent Let $\Gamma$ be a torsion-free arithmetic group acting on its associated global symmetric space $X$. Assume that $X$ is of non-compact type and let $\Gamma$ 
act on the geodesic boundary $\partial X$ of $X$. Via general constructions in $KK$-theory, we endow the $K$-groups of the arithmetic manifold $X / \Gamma$,  of the reduced group $C^*$-algebra $C^*_r(\Gamma)$ and of the boundary crossed product algebra $C(\partial X) \rtimes\Gamma$ with Hecke operators. The $K$-theory and $K$-homology groups of these $C^{*}$-algebras are related by a Gysin six-term exact sequence. In the case when $\Gamma$ is a group of real hyperbolic isometries, we show that this Gysin sequence is Hecke equivariant. Finally, when $\Gamma$ is a Bianchi group, we assign explicit unbounded Fredholm modules (i.e. spectral triples) to (co)homology classes, inducing Hecke-equivariant isomorphisms between the integral cohomology of $\Gamma$ and each of these $K$-groups. Our methods apply to case $\Gamma \subset \psl(\Z)$ as well. As these results are achieved in the context of unbounded Fredholm modules, they shed light on noncommutative geometric aspects of the purely infinite boundary crossed product algebra. 
\end{abstract}

\maketitle

\tableofcontents 
\addtocontents{toc}{~\hfill\textbf{Page}\par}

\section*{Introduction}

Let ${\bf G}$ be a semi-simple algebraic group over $\Q$ and $\Gamma \subset {\bf G}(\Q)$ be an arithmetic group. The cohomology of $\Gamma$ can be equipped with a commuting family of endomorphisms called ``Hecke operators'' which arise from correspondences on the associated arithmetic manifold $X / \Gamma$. The cohomology of $\Gamma$ as a Hecke module plays an important role in modern number theory. It is well-known that, as a Hecke module, the complex cohomology of $\Gamma$ can be completely accounted for by various spaces of automorphic forms (see \cite{franke}). The Langlands programme predicts that classes in the complex (resp. mod $p$) cohomology of $\Gamma$ that are simultaneous eigenvectors of the Hecke operators  correspond to $p$-adic (resp. mod $p$) Galois representations. This is known in many cases (see, for example,  \cite{scholze}).

Let $X$ denote the symmetric space associated to the real Lie group ${\bf G}(\R)$. Assume that $X$ is of non-compact type and let $\partial X$ denote the geodesic boundary of $X$. We consider three $C^*$-algebras that are naturally associated to $\Gamma$, namely, the algebra $C_0(X / \Gamma)$ of functions on the arithmetic manifold $X/ \Gamma$, the reduced group $C^*$-algebra $C^*_r(\Gamma)$ and the boundary crossed product algebra $C(\partial X) \rtimes \Gamma$, which we call the \emph{arithmetic $C^*$-algebras} associated to $\Gamma$. 

Via general constructions in $KK$-theory, we endow the $K$-groups of arithmetic $C^*$-algebras with Hecke operators. As in the case of the cohomology of arithmetic groups, we expect that these Hecke modules are related to the arithmetic of automorphic forms and to Galois representations. We show that this is the case when $\Gamma$ is a Bianchi group by proving that the {\it integral} cohomology of $\Gamma$ and the $K$-homology of the arithmetic $C^*$-algebras associated to $\Gamma$ are isomorphic as Hecke modules. We realize these isomorphisms by constructing explicit unbounded Fredholm modules (i.e. spectral triples) for each of the arithmetic $C^*$-algebras. Our results also apply to the case of $\psl(\Z)$.

Bianchi groups are among the simplest generalizations of the $\psl(\Z)$, yet from the perspective of the Langlands programme, many fundamental questions are still wide open (see, for example, \cite{sengun-survey}) and moreover new phenomena (\cite{bergeron-venkatesh, calegari-venkatesh, scholze, sengun}) that are not present in the setting of $\psl(\Z)$ (or, more generally, in the setting of Shimura varieties) arise. Our results allow the possibility of using ideas and tools from $K$-theory (both topological and analytic) and noncommutative geometry in investigating these problems and phenomena. For example, one of our results implies that, up to normalization, a period of a  Bianchi (or classical) modular form can be captured as the index of a Fredholm operator.  

Our work is inspired by the works of Manin and Marcolli \cite{mm-1, manin-marcolli, mm-3} which pursue number theoretic questions around the `noncommutative modular curves' via tools of Connes' noncommutative geometry \cite{connes-book,connesspectral}. The $K$-theory of $C^{*}$-algebras associated to Kleinian and Fuchsian groups have been studied by many authors e.g. \cite{delaroche,  berkove,connestrans, EM}, with $K$-homology receiving attention only recently \cite{EN, Fuchs, rahmcompt, rahm}. The novelty of our work lies in the introduction of the Hecke module structure within $KK$-theory 
and the construction of new unbounded $K$-cycles using the unbounded Kasparov product \cite{KaLe, Mes, MR}. This is of particular interest for the purely infinite boundary crossed product $C^{*}$-algebra $C(\partial X) \rtimes \Gamma$. 

\subsection*{Description of results and plan of the paper}
We first carry out a very general treatment of Hecke operators as they are crucial for the link with arithmetic that we are aiming to establish. This is done in Section \ref{section: hecke}. A novelty of our treatment is the introduction of Hecke operators via $KK$-theory. This proves to be a powerful, and natural, way of treating Hecke operators that allows us to establish various results related to Hecke operators in a robust and efficient manner. 

Let $G$ be a group acting on a locally compact Hausdorff space $X$. Given a subgroup $\Gamma$ of $G$ which acts freely and properly on $X$, put $M=X/\Gamma$. For any element $g$ in the commensurator group of $\Gamma$ in $G$, we define a bimodule $T_g^M$ and let $[T_{g}^{M}]\in KK_{0}(C_{0}(M),C_{0}(M))$ be its class. Then we define the Hecke operator 
$$T_{g}:K^{*}(C_{0}(M))\to K^{*}(C_{0}(M)),\quad x\mapsto [T_{g}^{M}]\otimes x,$$
to be the Kasparov product with this class. Now let $B$ an arbitrary $C_{G}(\Gamma)$-$C^{*}$-algebra. Similarly we define a class 
$[T_{g}^{\Gamma}]\in KK_{0}(B\rtimes_{r}\Gamma,B\rtimes_{r}\Gamma)$ and define the Hecke operator 
$$T_{g}: K^{*}(B\rtimes_{r}\Gamma)\to K^{*}(B\rtimes_{r}\Gamma),\quad x\mapsto [T^{\Gamma}_{g}]\otimes x,$$
as the Kasparov product with this class. Of course, in both cases, we obtain Hecke operators on $K$-theory as well. These two constructions allow us to define Hecke operators 
on the $K$-groups of arithmetic $C^*$-algebras.

Next, in Section \ref{section: gysin}, we study the Hecke equivariance of a certain Gysin exact sequence, which shall play an important role in our investigation. Similar exact sequences were studied by Emerson-Meyer \cite{EM} in $K$-theory and later by Emerson-Nica \cite{EN} in $K$-homology. Our treatment is again general here. Let $\Gamma$ be a group acting freely and properly on hyperbolic $n$-space $\h_n$ via isometries such that the hyperbolic manifold $M=\h_n / \Gamma$ has finite volume. Let $C^{*}_{r}(\Gamma)$ denote the reduced group $C^*$-algebra of $\Gamma$. Starting from 
the $\Gamma$-equivariant short exact sequence of $C^*$-algebras
$$0\rightarrow C_{0}(\h_n)\rightarrow C(\h_n \cup \partial \h_n)\rightarrow C(\partial \h_n)\rightarrow 0,$$
we show the following in Theorem \ref{thm: GysinHecke} below.

\begin{thm-hand}[A.] The $K$-homology Gysin exact sequence takes the form  
$$ 
\xymatrix{ 0\ar[r] & K^{1}(C_{0}(M)) \ar[r] & K^{0}(C(\partial \h_n)\rtimes\Gamma) \ar[r] & K^{0}(C^{*}_{r}(\Gamma)) \ar[d] \\ 
0 &\ar[l] K^{1}(C^{*}_{r}(\Gamma))  & K^{1}(C(\partial \h_n)\rtimes \Gamma) \ar[l] & K^{0}(C_{0}(M))  \ar[l] }
$$
and is {\em Hecke equivariant}. 
\end{thm-hand}
  
This is done by representing the boundary extension as a Fredholm module constructed from the field of harmonic measures on the boundary. Since we work within $KK$-theory, the results of Section \ref{section: hecke} and Section \ref{section: gysin} have counterparts for $K$-theory groups as well.

At this point, we specialize to the case of Bianchi groups and study the above Hecke equivariant Gysin exact sequence in great detail. Let $K$ be an imaginary quadratic field with ring of integers $\Z_K$. 
Let $\Gamma$ be a torsion-free finite index subgroup of the Bianchi group $\psl(\Z_K)$ acting on the hyperbolic $3$-space $\h_3$ and its boundary $\partial \h_3$. In this case, the $K$-homology Gysin exact sequence splits into two exact sequences, 
\begin{equation}\tag{$\star$} \label{ses}
0 \rightarrow K^{i+1}(C_{0}(M)) \rightarrow K^i(C(\partial \h_3)\rtimes\Gamma) \rightarrow K^i(C^{*}_{r}(\Gamma)) \rightarrow 0,
\end{equation}
with $i=0,1$, revealing that $K$-homology of $C(\partial \h_3)\rtimes\Gamma$ is `made of' that of  $C_r^*(\Gamma)$ and that of $M$. We then study the $K$-groups of these two parts in Section \ref{section: algebraic} and Section \ref{section: geometric}, relating them to the ordinary 
(co)homology of $\Gamma$ and $M$. 
We do not consider $K$-theory groups anymore except in Section \ref{section: algebraic}. To describe the $K$-homology isomorphism at the unbounded level, we give a construction of unbounded Fredholm modules from group cocycles. The fact that Kasparov's $\gamma$-element is equal to the identity in $KK^{\Gamma}_0(\C,\C)$ (see \cite{ kasparov-SOn1,kasparov}) is a vital ingredient in the construction. 

Let $c:\Gamma\to \Z$ be $1$-cocycle with kernel $\Gamma_{c}$ and $(C^{*}_{r}(\Gamma), E, D_{c})$ the associated unbounded $(C^{*}_{r}(\Gamma), C^{*}_{r}(\Gamma_c))$ Kasparov module (see Proposition \ref{modularindex}). For $0<s<1$, the $\gamma$-element induces an unbounded Fredholm module $$(C^{*}_{r}(\Gamma_c), L^{2}(\wedge^{*}\h_n), \slashed{D}_{s}:=\slashed{D}_{HR}+\rho^{s}\hat{c}(\mbox{d}\rho)),$$
where $\slashed{D}_{HR}$ is the Hodge-deRham operator, $\rho(x)=d_{\h_n}(0,x)$ the distance function on the hyperbolic ball and $\hat{c}$ the Clifford multiplication. We prove the following result in Theorem \ref{algisos}.
\begin{thm-hand}[B.] The operators $D_c$ and $\slashed{D}_{s}$ assemble into an $(1-s)$ unbounded Fredholm module
\[(C^{*}_{r}(\Gamma), E\otimes_{C^{*}_{r}(\Gamma_c)}L^{2}(\wedge^{*}\h_n), D_{c}\otimes \sigma + 1\otimes_{\nabla_{g}}\slashed{D}_{s}). \]
For $n=3$, this induces Hecke equivariant isomorphisms
$$H^1(\Gamma,\Z) \simeq K^1(C^*_r(\Gamma)), \ \ c\mapsto [D_{c,s}] \ \ \ \ H_1(\Gamma,\Z) \simeq K_1(C^*_r(\Gamma)), \ \ [\gamma]\mapsto u_\gamma.$$
Moreover under these isomorphisms, the homological pairing $H^* \times H_* \rightarrow \Z$ corresponds to the index pairing 
$K^* \times K_* \rightarrow \Z$. 
\end{thm-hand}
The $K$-theory isomorphism comes from the results in \cite{bettaiebmattheyvalette, matthey}. In Section \ref{section: geometric} we exploit the equivalence between geometric and analytic $K$-homology (\cite{BHS})  of the non-compact manifold $M$ and in Theorem \ref{geometriciso} establish the following:
\begin{thm-hand} [C.] There is an explicit Hecke equivariant isomorphism
$$ H_2(\overline{M}, \partial \overline{M}, \Z) \simeq K^0(C_0(M)), \ \ \ \ (N,\partial N)\mapsto (C_{0}(M), L^{2}(\mathring{N},\mathscr{S}_{\mathring{N}}), \slashed{D}_{\mathring{N}})$$
where $\overline{M}$ denotes the Borel-Serre compactification of $M$ and $(N,\partial N)\subset (\overline{M},\partial\overline{M})$ is a properly embedded hypersurface. The interior $\mathring{N}$ inherits a complete Riemannian metric from $M$ and hence the Dirac operators  $\slashed{D}_{\mathring{N}}$ are self-adjoint.\\
\end{thm-hand}
Note that $H^1(\Gamma, \Z) \simeq H_2(\overline{M}, \partial \overline{M}, \Z)$. Our methods above also apply to the case of torsion-free finite index subgroups of $\psl(\Z)$ which we discuss in Section \ref{section: others}.  

With these results in place, we proceed to describe the unbounded Fredholm modules that exhaust the $K$-homology of the purely infinite boundary crossed product $C(\dH_3)\rtimes\Gamma$. To this end we use the maps $K^{1}(C(\dH_{3})\rtimes\Gamma)\to K^{1}(C_{r}^{*}(\Gamma))$, for which we construct an explicit section in Section \ref{section: algebraic}, and $\partial: K^{0}(C_{0}(M))\to K^{1}(C(\dH_3)\rtimes \Gamma)$ coming from (\ref{ses}). 

To compute the latter map at the unbounded level, we construct an unbounded representative for the extension class by means of a hypersingular integral operator built from the harmonic measures $\nu_x$ and associated metrics $d_x$ on $\dH_3$ based at $x\in\h$. Let $T_1\h_n=\h_n\times\partial\h_n$ be the unit tangent bundle of $\h_n$ and $L^{2}(T_1\h_n,\nu_x)$ the associated $C^{*}$-module completion. The integral operators
\[\Delta\Psi(x,\xi)=\int_{\partial\h_n}\frac{\Psi(x,\xi)-\Psi(x,\eta)}{d_x(\xi,\eta)^{n-1}}d\nu_x \eta,\ \ p\Psi(x,\xi)=\int_{\partial\h_n} \Psi(x,\eta)d\nu_x\eta, \]
are $G$-invariant and the mutliplication operator $\rho\Psi(x,\xi)=d_{\h_n}(0,x)\Psi(x,\xi)$ commutes with $G$ boundedly in $L^{2}(T_1\h,\nu_x)$. In Section \ref{section: hyper} we prove:
\begin{thm-hand}[D.] The operators $\Delta,\rho$ and $p$ 
assemble into a $G$-equivariant unbounded Kasparov module
\[(C(\dH), L^{2}(T_1\h,\nu_x), -\Delta+(2p-1)\rho),\]
representing the class of the $G$-equivariant extension $$0\rightarrow C_{0}(\h_n)\rightarrow C(\h_n \cup \partial \h_n)\rightarrow C(\partial \h_n)\rightarrow 0.$$
\end{thm-hand}
In Section \ref{section: boundary} we then compute the unbounded Kasparov product of this operator with the self-adjoint Dirac operators on embedded hypersurfaces from Section \ref{section: geometric}. The main result here can be found in Theorem \ref{product}. The results in Sections \ref{section: algebraic} and \ref{section: boundary} can be summarized to describe the structure of the $K$-homology of the purely infinite simple $C^{*}$-algebra $C(\partial \h_3)\rtimes\Gamma$ as in the following theorem.

\begin{thm-hand}[E.] The unbounded Fredholm modules in Theorem B extend to $C(\partial \h_n)\rtimes\Gamma$. A self-adjoint Dirac operator $\slashed{D}_{\mathring{N}}$ associated to an embedded hypersurface and the integral operator $S:=-\Delta+(2p-1)\rho$ of Theorem D assemble into an unbounded Kasparov product
\[(C(\partial\h_n)\rtimes \Gamma, L^{2}(T_{1}\h_n)\otimes_{C_{0}(M)}L^{2}(\mathring{N},\mathscr{S}_{\mathring{N}}), S\otimes \sigma+1\otimes_{\nabla}\slashed{D}_{\mathring{N}}).\]
Consequently, for $n=3$, there is an explicit Hecke equivariant isomorphism 
$$H^{1}(\Gamma,\Z)\oplus H_{2}(\overline{M},\partial\overline{M}) \simeq K^1(C(\partial \h_3)\rtimes\Gamma),\ \ \ \  ([c],[(N,\partial N)])\mapsto [D_{c,s}]+[S\otimes 1+1\otimes_{\nabla}\slashed{D}_{\mathring{N}}] $$
defined at the level of unbounded Fredholm modules.
\end{thm-hand}
The two pieces of $K^{1}(C(\dH_3)\rtimes \Gamma)$ give rise to very different unbounded Fredholm modules, which by virtue of the Gysin sequence pair with $K$-theory in distinct ways.

\subsection*{Questions} 
\begin{enumerate}

\item Given an arithmetic group $\Gamma$, we construct Hecke operators on the $K$-groups of arithmetic $C^*$-algebras associated to $\Gamma$. In the case of Bianchi groups, we show that our Hecke operators correspond to the classical Hecke operators on the cohomology groups of $\Gamma$. While this convinces us that our construction is natural and correct, for general $\Gamma$ such a comparison is still to be made.  A natural question to ask here is, does  the Chern character homomorphism 
$$K^i(B\Gamma)  \to \bigoplus_{n \geq 0} H^{2n+i}(B\Gamma, \Q),$$ 
where $i=0,1$ commute with the Hecke operators?

\item Torsion in the homology of arithmetic groups has gained a lot of interest in recent years. What  can we say about the torsion in the $K$-theory of arithmetic $C^*$-algebras? We observe in Section \ref{section: Kasparov} that for Bianchi groups, $H_1$ and $K_1$ hold the same torsion. However this is a coincidence of low dimensionality and in general the torsion on the two sides will not agree. Note that it is natural expect that $p$-torsion Hecke eigenclasses in the $K$-homology 
of arithmetic $C^*$-algebras have associated mod $p$ Galois representations.

\item Can the $K$-homology of the arithmetic $C^*$-algebras as Hecke modules be accounted for by automorphic forms as is the case for cohomology? If so, what are these `$K$-theoretic' automorphic forms? 
Are they the same as cohomological ones? In the case of $\psl(\Z)$ and Bianchi groups, we show that they are the same. In these cases, can we directly associate a $K$-homology class
associated to a Bianchi (or classical) modular form?

\item What can we say about the summability properties of the (un)bounded Fredholm modules that we construct? Do their spectral zeta functions relate to the arithmetic of Bianchi modular forms?

\end{enumerate}

\subsection*{Acknowledgements} We gratefully acknowledge that we benefited from conversations with Michael Atiyah, Alain Connes, Gunther Cornelissen, Robin Deeley, Nathan Dunfield, Magnus Goffeng, Paul Gunnells, Joel Hass, Matilde Marcolli, Sergey Neshveyev, Ryszard Nest, Bogdan Nica, Elmar Schrohe, Ren\'{e} Schulz, Richard Sharp, Aidan Sims and Alain Valette. We thank the Leibniz Universit\"{a}t Hannover and the MSRC of the Univerity of Sheffield for their financial support in facilitating this collaboration. We are also grateful to the University of Warwick and the Nesin Mathematics Village for providing excellent working environments. The first author was partially supported by EPSRC grant EP/J006580/2, Simons Foundation grant 346300 and the Polish Government MNiSW 2015-2019 matching fund. The second author was a Marie Curie Intra-European Fellow during the preparation of a significant part of this work.

\section{Prelude: Kasparov's spectral sequence}\label{section: Kasparov}
Let $\Gamma$ be a torsion-free cofinite discrete subgroup of $\psl(\C)$ acting on the hyperbolic $3$-space $\h_3$ and its boundary $\partial \h_3$. The limit set of $\Gamma$ is all of $\partial \h_3$ on which it 
acts with dense orbits. We can identify $\partial \h_3$ with $S^2 \simeq \P^1(\C)$ and the action of $\Gamma$ with the usual M\"obius action. Let $M$ denote the hyperbolic $3$-manifold $\h_3 / \Gamma$. In this section, we employ a $KK$-theory spectral sequence and get a description of $K$-groups of $C(\partial \h_3)\rtimes\Gamma$, $C^*_r(\Gamma)$ and $C_0(M)$ in terms of the cohomology of $\Gamma$. 

The abstract isomorphisms that will come out of the spectral sequence will motivate the main task we accomplish in the present paper: Can we equip the respective $K$-groups with a Hecke module module structure and find \emph{explicit Hecke equivariant} isomorphisms from (co)homology to $K$-groups?

We let $D$ be a $C^{*}$-algebra with a $\Gamma$-action and $D \rtimes_{r} \Gamma$ be the reduced crossed product algebra of $\Gamma$ and $D$. Let $\gamma \in KK^\Gamma(\C,\C)$ denote the Kasparov idempotent.  A spectral sequence of Kasparov (see \cite[Section 6.10.]{kasparov}, see also \cite{schochet}) calculates the $\gamma$-parts of the K-groups of $D \rtimes_{r} \Gamma$ out of those of $D$. For discrete subgroups of $\textnormal{Isom}(\h_n)$, it holds that $\gamma=1$, a fact that will be of importance in several places in the present paper (see for instance \cite[Chapter 9]{Valette} and \cite{kasparov-SOn1, kasparov}). For our groups, the  $\gamma$-part of a K-group of $D \rtimes_{r} \Gamma$ is itself. 
\begin{theorem} \label{thm: kasparov} There is a cohomological spectral sequence $(E_r,d_r)$ with differentials $d^{p,q}_r : E_r^{p,q} \rightarrow E_r^{p+r,q-r+1}$ and the term 
$E_2^{p,q} = H^p(\Gamma, K^q(D))$ converging to the K-homology groups of $D \rtimes_{r} \Gamma$. There is an analogous homological spectral sequence converging to the 
K-theory groups of $D \rtimes_{r} \Gamma$.
\end{theorem} 

By setting $D$ equal to $\C=C({\bf pt}), C_0(\h_3)$ and $C(\partial \h_3),$ we shall use the above spectral sequence to obtain information on the $K$-groups of $C(\partial \h_3) \rtimes_{r} \Gamma$, $C^{*}_r(\Gamma)$ and $C_0(M)$ respectively. We first set 
$D =C(\partial \h_3) $, and note that the action of $\Gamma$ on $\dH_{3}$ is amenable so the full and reduced crossed products coincide (see \cite[Lemma 3.8]{Lott}). The following well known lemma computes the $K$-homology groups of $D$ as $\Gamma$-modules in this case.

\begin{lemma} \label{lem: p1C}We have
$$K^0(C(\partial \h_3)) \simeq \Z^2, \ \ \ \ \ \ K^1(C(\partial \h_3))= \{ 0 \}, $$
where the action of  \hspace{.02in} $\Gamma$ on $K^0(C(\partial \h_3))$ is trivial.
\end{lemma}
\begin{proof} The $K$-homology of the two-sphere $S^{2}\simeq \dH_{3}$ is well-known. The triviality of the action of $\Gamma$ on $K^{0}(C(\partial \h_3))$ follows from the facts that it is the restriction of the action of $PSL_{2}(\C)$, which is a connected group, and $K$-homology is homotopy invariant.
\end{proof}

\begin{proposition}\label{prop: spectral-big} There is a short exact sequence
\begin{equation}\label{exseq} 
0 \rightarrow H^0(\Gamma, \Z^2) \rightarrow K^0(C(\partial \h_3) \rtimes \Gamma) \rightarrow H^2(\Gamma, \Z^2) \rightarrow 0
\end{equation}
and an isomorphism
\begin{equation}\label{iso}
H^1(\Gamma, \Z^2) \simeq K^1(C(\partial \h_3) \rtimes \Gamma)
\end{equation}
where the action of \hspace{.02in} $\Gamma$ on $\Z^2$ is trivial. 
\end{proposition}

\begin{proof} We apply Theorem \ref{thm: kasparov} with $D=C(\partial \h_3)$ and use Lemma \ref{lem: p1C}. The cohomological dimension of $\Gamma$ is two. As $\Gamma$ is torsion-free, all its integral cohomology above degree two vanishes and we see that the $E_2$ page of the spectral sequence looks like this:

\begin{center}
\begin{tikzpicture}
  \matrix (m) [matrix of math nodes,  nodes in empty cells, nodes={minimum width=3ex, minimum height=3ex,outer sep=-5pt}, column sep=1ex,row sep=1ex]
            {   &  \vdots &  \vdots  & \vdots &  \vdots &  \vdots  & \vdots & \iddots \\
                &  0 &  0  & 0 &  0 &  0  & 0 & \hdots \\
                &  H^0(\Gamma,\Z^2)  & H^1(\Gamma,\Z^2) &  H^2(\Gamma,\Z^2)  & 0 & 0 & 0 & \hdots \\
                &  0 &  0  & 0 &  0 &  0  & 0 & \hdots \\
                &  H^0(\Gamma,\Z^2)  & H^1(\Gamma,\Z^2) &  H^2(\Gamma,\Z^2)  & 0 & 0 & 0 & \hdots \\
    \quad\strut &    &    &    & \strut \\};
\draw[thick]  (-5,-2) -- (6,-2);
\draw[thick]  (-5,-2) -- (-5,3);

        \node [below=2.2cm, align=flush center,text width=8cm] 
        {
            The $E_2$ page of the spectral sequence.
        };
\end{tikzpicture}
\end{center}
\end{proof}

Note that since the action of $\Gamma$ is trivial on $\Z^2$, we have $H^{*}(\Gamma, \Z^2) \simeq H^{*}(\Gamma,\Z) \otimes \Z^2$. 
In particular, $H^1(\Gamma, \Z^2) \cong H^1(\Gamma, \Z)^{\oplus 2}$ and thus $K^1(C(\dH_3) \rtimes \Gamma)$ 
holds two copies of $H^1(\Gamma, \Z)$. Moreover, $H^0(\Gamma,\Z^2) \simeq \Z^2$. As $H^2(\Gamma,\Z)$ typically has a lot of torsion (see \cite{sengun, bergeron-venkatesh}), the sequence \ref{exseq} does not split.

Next, we apply Theorem \ref{thm: kasparov} to the case $D=\C=C(\bf{pt})$. Note that $C({\bf{pt}}) \rtimes \Gamma \simeq  C^{*}_r(\Gamma)$. 
Since $K^0(\C) \simeq \Z$, $K^1(\C)=\{ 0 \}$. and the $\Gamma$-action is trivial, we find
\begin{proposition}\label{prop: spectral-small} There is a short exact sequence
\begin{equation}
0 \rightarrow H^0(\Gamma, \Z) \rightarrow K^0(C^{*}_r(\Gamma)) \rightarrow H^2(\Gamma, \Z) \rightarrow 0
\end{equation}
and an isomorphism
\begin{equation}
H^1(\Gamma, \Z) \simeq K^1(C^{*}_r(\Gamma)),
\end{equation}
where the action of $\Gamma$ on $\Z$ is trivial. 
\end{proposition}

Lastly, we apply Theorem \ref{thm: kasparov} to the case $D=C_0(\h_3)$. Note that $C_0(\h_3) \rtimes \Gamma$ is Morita equivalent to $C_0(M)$ as the action of $\Gamma$ on $\h_3$ is free and properly discontinuous. In particular, they have the same $K$-groups. It is well known that $K^0(C_0(\h_3)) \simeq \{ 0 \}$ and that $K^1(C_0(\h_3))=\Z$ with trivial $\Gamma$-action.  

\begin{proposition}\label{prop: spectral-manifold} There is a short exact sequence
\begin{equation}
0 \rightarrow H^0(\Gamma, \Z) \rightarrow K^1(C_0(M)) \rightarrow H^2(\Gamma, \Z) \rightarrow 0
\end{equation}
and an isomorphism
\begin{equation}
H^1(\Gamma, \Z) \simeq K^0(C_0(M))
\end{equation}
where the action of $\Gamma$ on $\Z$ is trivial. 
\end{proposition}

We summarize the results of the homological spectral sequence and  omit the details as they are the same as the above.
\begin{proposition} \label{prop: spectral-homology} There are isomorpshisms
\begin{equation}
H_1(\Gamma, \Z^2) \simeq K_1(C(\dH_3) \rtimes \Gamma),
\end{equation}
and
\begin{equation}
H_1(\Gamma, \Z) \simeq K_1(C^{*}_r(\Gamma)),
\end{equation}
and 
\begin{equation}
H_1(\Gamma, \Z) \simeq K_0(C_0(M)),
\end{equation}
and short exact sequences 
\begin{equation}\label{evenKtheory_boundary}
0 \rightarrow H_0(\Gamma, \Z^2) \rightarrow K_0(C(\dH_3)) \rtimes \Gamma) \rightarrow H_2(\Gamma, \Z^2) \rightarrow 0,
\end{equation}

\begin{equation}\label{evenKtheory}
0 \rightarrow H_0(\Gamma, \Z) \rightarrow K_0(C^{*}_r(\Gamma)) \rightarrow H_2(\Gamma, \Z) \rightarrow 0,
\end{equation}

\begin{equation}\label{evenKtheory_manifold}
0 \rightarrow H_0(\Gamma, \Z) \rightarrow K_1(C_0(M)) \rightarrow H_2(\Gamma, \Z) \rightarrow 0.
\end{equation}
where the actions of $\Gamma$ on $\Z^2$ and on $\Z$ are trivial.
\end{proposition}
With these abstract isomorphisms in hand, we set ourselves two tasks. The first is to equip the $K$-groups of the above arithmetic $C^*$-algebras with natural Hecke module structures. Secondly, we would like to compare the $K$-groups and (co)homology groups as Hecke modules. For this purpose, the abstract isomorphism coming from Kasparov's spectral sequence above cannot help, so we need to construct explicit isomorphisms between the respective groups appearing in this section. In the rest of the paper, we accomplish both of our tasks. However it should be noted that while our treatment of the first task is general, 
our treatment of the second task is very specific to the case of Bianchi groups (see Question 2 in the Introduction). 

\section{Hecke operators and $KK$-theory} \label{section: hecke}
The various (co)homology groups associated with an arithmetic group $\Gamma$ come equipped with so called \emph{Hecke operators}. These arise from elements in the 
commensurator $ C_G(\Gamma)$ of $\Gamma$ in its ambient real Lie group $G$: 
\[ C_G(\Gamma):=\{g\in G: \Gamma\cap g\Gamma g^{-1}\textnormal{ has finite index in both} \ \Gamma \ \textnormal{and} \ g\Gamma g^{-1}\}.\]
We start by quickly recalling the definition of Hecke operators on the (co)homology groups that we deal with in the paper. 
Afterwards, for each element in $ C_G(\Gamma)$, we construct elements in $KK$-rings $KK_{0}(A,A)$ of the arithmetic $C^{*}$-algebras $A$ associated to $\Gamma$. 
The elements that we construct will give rise to endomorphisms which play the r\^{o}le of Hecke operators on $K$-groups of $A$.

\subsection{Homological definitions}\label{homdefsec}
Let $\Gamma \subset \psl(\C) =: G$ be a torsion-free finite-index subgroup of a Bianchi group $\psl(\Z_K)$, acting on $\h_3$ freely and proper discontinuously.
In this case, we have $ C_G(\Gamma) = \pgl(K) \subset \pgl(\C) \cong G$. For our purposes, the main distinction to be made is that between algebraically defined Hecke operators on $H^{*}(\Gamma,\Z)$ and topologically defined Hecke operators on $H_{*}(\overline{M},\partial\overline{M}, \Z)$ where $M$ is the associated hyperbolic $3$-manifold $\h_3 /\Gamma$ and $\overline{M}$ is its Borel-Serre bordification.

\subsubsection{On Group Homology.}
For a subgroup $\Delta\subset \Gamma$ of finite index $d$, any choice of coset representatives \[\gamma_{i}\in \Gamma,\quad \Gamma=\bigsqcup_{i=1}^{d}\gamma_{i}\Delta,\]
gives a map $s:\Gamma\to \Delta^{d}$ determined by $\gamma \gamma_{i}=\gamma_{\gamma(i)}s_{i}(\gamma)$, where $s_{i}(\gamma)\in \Delta$ and $\gamma(i)$ is a permutation of $1,\cdots, d$. This determines the \emph{transfer} or \emph{corestriction map} $$\textnormal{cores}:H^{1}(\Delta,\Z)\to H^{1}(\Gamma,\Z),\quad \textnormal{cores } c (\gamma)=\sum_{i=1}^{d}c(s_{i}(\gamma)),$$
which is independent of the choice of coset representatives $\gamma_{i}$. 
For $g\in  C_G(\Gamma)$, write $\Gamma_{g}:=\Gamma\cap g\Gamma g^{-1}$ and the Hecke operator on group cohomology is given by
\begin{equation}\label{Heckedef} T_{g}: H^{1}(\Gamma,\Z)\xrightarrow{\textnormal{res}} H^{1}(\Gamma_{g},\Z)\xrightarrow{\textnormal{Ad}_g} H^{1}(\Gamma_{g^{-1}},\Z)\xrightarrow{\textnormal{cores}} H^{1}(\Gamma,\Z).\end{equation}
Operators $T_{g}:H_{1}(\Gamma,\Z)\to H_{1}(\Gamma,\Z)$ are defined analogously. 
To compute the operator $T_{g}$,  one uses the disjoint union decomposition of the double coset
\begin{equation}\label{doublecoset}\Gamma g^{-1} \Gamma=\bigsqcup_{i=1}^{d} g_{i}\Gamma, \quad g_{i}=\delta_{i}g^{-1}\in G,\quad \delta_{i}\in \Gamma.\end{equation}
The elements $\delta_{i}$ form a complete set of coset representatives for $\Gamma/\Gamma_{g^{-1}}$. The group $\Gamma$ acts on the double coset $\Gamma g^{-1} \Gamma$, and thus permutes the cosets $g_{i}\Gamma$. As above there are indices $\gamma(i)$ and group elements $t_{i}(\gamma)\in \Gamma$ such that $\gamma g_{i}=g_{\gamma(i)}t_{i}(\gamma)$, determining a map $\Gamma\to \Gamma^{d}_{g^{-1}}$.  
The Hecke operators $T_{g}:H^{1}(\Gamma,\Z)\to H^{1}(\Gamma,\Z)$ and $T_{g}:H_{1}(\Gamma,\Z)\to H_{1}(\Gamma,\Z)$ are then given explicitly by
\begin{equation}\label{algHecke} (T_{g}c)(\gamma):=\sum_{i=1}^{d}c(t_{i}(\gamma)),\quad T_{g}([\gamma])=\sum_{i}^{d}[t_{i}(\gamma)]  \end{equation}
 which is independent of the choice of coset representatives $\delta_{i}$. 

\subsubsection{On Simplicial Homology.} \label{borelserre}
We start with the manifold $M_{g}:=\h_3/\Gamma_{g}$ and the associated finite covering $\pi_{g}:M_{g}\rightarrow M.$ This finite covering induces a corestriction map $\pi_{g}^{*}:H_{*}(M)\to H_{*}(M_{g})$ by mapping a simplex to the sum of its inverse images. Similarly there is a covering $\pi_{g^{-1}}:M_{g^{-1}}\rightarrow M$, and the isometry $g:\h_3 \to \h_3$ induces a homeomorphism $g_{*}:M_{g}\to M_{g^{-1}}$ because $g^{-1}\Gamma_{g} g=\Gamma_{g^{-1}}$. Thus we obtain a second covering $\tau_{g}:=\pi_{g^{-1}}\circ g_{*}:M_{g}\to M$.  For $g\in  C_G(\Gamma)$, we define Hecke operators, both denoted $T_{g}$, on homology and on cohomology as the group homomorphisms
\[T_{g}:=\tau_{g_*} \circ \pi_{g}^{*}: H_{*}(M,\Z)\to H_{*}(M,\Z),\]
\[T_{g}:=\tau_{g}^* \circ \pi_{g_*}: H^{*}(M,\Z)\to H^{*}(M,\Z).\]

We shall need Hecke operators also on the homology of the Borel-Serre compacitifications. In our low-dimensional cases, these compactifications can be described concretely as follows (see \cite{borel-serre} and also \cite[III.5.15]{borel-ji}, \cite[$\S$2.8]{berger-thesis} ). We first construct a partial compactification $\widehat{\h}_3$ of $\h_3$ by adding a copy of the complex plane $\C$ to every boundary point in $\P^1(K) \subset \P^1(\C) = \partial \h_3$, more precisely
$$\widehat{\h}_3 = \h_3 \bigsqcup_{z \in \P^1(K)} \P^1(\C) \backslash \{ z \}.$$
 The copy $\P^1(\C) \backslash \{ z \} = \C$ is the parameter space of all geodesics in $\h_3$ converging to the boundary point $z \in \P^1(K)$. The action of $\pgl(K)$, but not of $G$, on $\h_3$ extends to an action  on $\widehat{\h}_3$ by sending $\omega \in \P^1(\C) \backslash \{ z \}$ to $\omega \gamma \in \P^1(\C) \backslash \{ z \gamma \}$. One can topologize $\widehat{\h}_3$ in such a way that the action of $\pgl(\K)$ is continuous.  The action of $\Gamma$ on $\widehat{\h}_3$, unlike its action on the geodesic completion $\overline{\h}_{3}$, is free and proper. The quotient $\widehat{\h}_3 / \Gamma$ can be shown to be a compact $3$-manifold with boundary which we call the Borel-Serre compactification of $M$  and denote by $\overline{M}$. The connected components of its boundary are $2$-tori,  attached at `infinity' to each cusp of $M$. Note that $M$ is the interior of $\overline{M}$ and thus they are homotopy equivalent. 
 
Just as before, we obtain finite coverings $\overline{\pi}_{g}, \overline{\tau}_g: \overline{M}_{g} \rightarrow \overline{M}$, extending $\pi_{g}, \tau_g: M_{g} \rightarrow M$, 
and construct the Hecke operator 
\[T_{g}:=(\pi_{g^{-1}})_{*}\circ g_{*}\circ \pi_{g}^{*}: H_{*}(\overline{M}, \Z)\to H_{*}(\overline{M}, \Z).\] 
As $\overline{\pi}_{g}, \overline{\tau}_g$ restrict to finite coverings on the boundaries, we also obtain Hecke operators on the relative homology groups 
\[T_{g}: H_{*}(\overline{M}, \dM, \Z)\to H_{*}(\overline{M}, \dM, \Z).\] 

These Hecke operators are compatible with the Lefschetz duality isomorphism
\[H_{*}(\overline{M},\dM,\Z) \cong H^{n-*}(\overline{M},\Z),\]
see \cite[Lemma 1.4.3]{ash-stevens}, and the isomorphisms
\[ H^{*}(\overline{M},\Z)  \cong H^{*}(M,\Z) \cong H^{*}(\Gamma,\Z),\]
see, for example, \cite[Section 6]{Lee}.

\subsection{Hecke operators in $KK$-theory}\label{Heckedefs}

Let $X$ be a locally compact Hausdorff space and assume that $G$ acts on $X$ and that $\Gamma\subset G$ acts freely and properly on $X$. Suggestively, denote by $M:=X/\Gamma$ the quotient space which is locally compact and Hausdorff. 

The finite coverings $M\xleftarrow{\tau_{g}} M_{g}\xrightarrow{\pi_{g}}M$ form a \emph{correspondence} in the sense of \cite{conneskandalis} and define a class $[T_{g}^{M}]\in KK_{0}(C_{0}(M),C_{0}(M))$. The conditional expectation and right module structure 
\[\rho_{g}:C_{0}(M_{g})\to C_{0}(M),\quad \rho(\psi)(m)=\sum_{x\in \pi_{g}^{-1}(m)}\psi(x),\quad \psi\cdot f(x):=\psi(x)f(\pi_{g}(x))\]
give a right $C_{0}(M)$-module denoted by $T_{g}^{M}$.  Because the map $\tau_{g}: M_{g}\to M$ is proper, there is a left action by compact operators
\[C_{0}(M)\to \mathbb{K}(T_{g}^{M}),\quad f\cdot\psi(x)=f(\tau_{g}(x))\psi(x).\]
The class $[T_{g}^{M}]\in KK_{0}(C_{0}(M),C_{0}(M))$  coincides with the class of this bimodule.

\begin{definition}\label{manifoldHecke} Let $M=X/\Gamma$ as above. For any separable $C^{*}$-algebra $C$, the \emph{Hecke operators}
\[T_{g}:KK_{*}(C_{0}(M),C)\to KK_{*}(C_{0}(M),C),\quad T_{g}:KK_{*}(C, C_{0}(M))\to KK_{*}(C,C_{0}(M)) ,\]
are defined to be the Kasparov product with the class $[T_{g}^{M}]\in KK_{0}(C_{0}(M),C_{0}(M))$.
\end{definition}
For the moment, we denote by $B$ an arbitrary $C_{G}(\Gamma)$-$C^{*}$-algebra and by $g:b\mapsto g(b)$ the $C_{G}(\Gamma)$-action. Let $C_{c}(\Gamma,B)$ denote the compactly supported $B$-valued functions on $\Gamma$. The $\Gamma$-$C^{*}$-module
$$\ell^{2}(\Gamma,B):=\{\psi:\Gamma\to B:\sum_{\gamma\in\Gamma} \psi(\gamma)^{*}\psi(\gamma)<\infty\},$$  of $\ell^{2}$ functions on $\Gamma$ with values in $B$ is constructed as a completion of $C_{c}(\Gamma,B)$. The \emph{convolution product} and involution given by (see \cite{kasparovconspectus})
\begin{equation}\label{convinv} f * \psi(\gamma)=\sum_{\delta\in\Gamma}f(\delta)\delta(\psi(\delta^{-1}\gamma)),\quad f^{*}(\gamma):=\gamma f(\gamma^{-1})^{*}\end{equation}
make $C_{c}(\Gamma, B)$ into a $*$-algebra and define a $*$-representation $C_{c}(\Gamma,B)\to \End_{B}^{*}(\ell^{2}(\Gamma, B)).$ The \emph{reduced crossed product} $B\rtimes_{r}\Gamma$ is defined as the closure of $C_{c}(\Gamma, B)$ in this representation.

For a subgroup $\Delta\subset \Gamma$, restriction of functions $C_{c}(\Gamma,B)\to C_{c}(\Delta,B)\subset C_{c}(\Gamma,B)$ defines a projection $p_{\Delta}\in\End_{B}^{*}(\ell^{2}(\Gamma,B))$. This gives a contractive conditional expectation
\[\rho_{\Delta}:B\rtimes_{r}\Gamma \to B\rtimes_{r}\Delta,\quad a\mapsto p_{\Delta}ap_{\Delta},\]
extending the restriction map $C_{c}(\Gamma, B)\to C_{c}(\Delta, B)$.  Thus, for $g\in C_G(\Gamma)$ we obtain the expectation $\rho_{g^{-1}}:B\rtimes_{r}\Gamma\to B\rtimes_{r}\Gamma_{g^{-1}}$ and a $(B\rtimes_{r}\Gamma, B\rtimes_{r}\Gamma_{g^{-1}})$ bimodule $(B\rtimes_{r}\Gamma)_{\rho_{g^{-1}}}$. Using the *-homomorphism 
$$B\rtimes_{r}\Gamma_{g^{-1}}\xrightarrow{\Ad_{g}} B\rtimes_{r}\Gamma_{g}\hookrightarrow B\rtimes_{r}\Gamma,\quad \Ad_{g}(f)(\gamma):=gf(g^{-1}\gamma g),$$ 
we form the interior $C^{*}$-module tensor product
\[T^{\Gamma}_{g}:=(B\rtimes_{r}\Gamma)_{\rho_{g^{-1}}}\otimes_{\Ad_{g}}B\rtimes_{r}\Gamma ,\]
which is a $B\rtimes_{r}\Gamma$-bimodule. 

\begin{definition}\label{GammaHecke} Let $B$ be a separable $C_{G}(\Gamma)$-$C^{*}$-algebra and $C$ a seperable $C^{*}$-algebra. The \emph{Hecke operators}
\[T_{g}: KK_{*}(B\rtimes_{r}\Gamma,C)\to KK_{*}(B\rtimes_{r}\Gamma, C),\quad T_{g}: KK_{*}(C, B\rtimes_{r}\Gamma)\to KK_{*}(C, B\rtimes_{r}\Gamma).\]
are defined to be the Kasparov product with the class $[T^{\Gamma}_{g}]\in KK_{0}(B\rtimes_{r} \Gamma, B\rtimes_{r} \Gamma)$.
\end{definition}
Let $A$ be any of the $C^{*}$-algebras and $T_{g}$ any of the $KK$-theoretic Hecke operators discussed above. If $\langle x,y\rangle$ denotes the index pairing of elements $x\in K_*(A)$ and $y\in K^*(A)$, associativity of the Kasparov product gives $\langle T_{g} x,y\rangle=\langle x, T_{g} y\rangle$. That is, the Hecke action is self-adjoint with respect to the index pairing between $K$-theory and $K$-homology.
\subsection{Explicit formulae for the reduced crossed product}To describe the  $B\rtimes_{r}\Gamma$-bimodule $T^{\Gamma}_{g}$, let $\delta_{i}$ be as in \eqref{doublecoset} and $\chi_{i}\in C_{c}(\Gamma, M(B))$ be the function that is $1$ at $\delta_{i}$ and $0$ elsewhere. It is straightforward to check that $\sum_{i=1}^{d}\chi_{i}*\rho(\chi_{i}^{*} * f)=f$, and $\rho(\chi^{*}_{i}*\chi_{j})=\delta_{ij}$. This implements a unitary isomorphism of right modules
\begin{equation}\label{heckefree}u: T^{\Gamma}_{g}=(B\rtimes_{r}\Gamma)_{\rho_{g^{-1}}}\otimes_{\Ad _g}B\rtimes_{r}\Gamma\to (B\rtimes_{r}\Gamma)^{d},\quad f\otimes k\mapsto (\rho(\chi_{i}^{*}*f) * k), \end{equation} 
where $d=[\Gamma:\Gamma_{g^{-1}}]$. To describe the left $B\rtimes_{r}\Gamma$ action on $T^{\Gamma}_{g}\simeq (B\rtimes_{r}\Gamma)^{d}$ we consider the dense submodule $C_{c}(\Gamma, B^{d})$, the elements of which we view as columns $\Psi:=(\Psi_{i})_{i=1}^{d}$ of maps $\Psi_{i}: \Gamma \to B.$
First we collect some useful facts and relations for the elements $t_{i}(\gamma)$. 
\begin{lemma}\label{Heckerelations}We have the relations:\newline\newline 
1.) $t_{i}(\gamma)=g\delta_{\gamma(i)}^{-1}\gamma\delta_{i}g^{-1}=g_{\gamma(i)}^{-1}\gamma g_{i}$;\newline
2.) $t_{i}(\gamma_{1}\gamma_{2})=t_{\gamma_{2}(i)}(\gamma_{1})t_{i}(\gamma_{2})$;\newline
3.) $t_{i}(\gamma^{-1})=t_{\gamma^{-1}(i)}(\gamma)^{-1}$.
\end{lemma}
 \begin{proof}All relations are checked by direct computation using the defintions in section \ref{homdefsec} 
 \end{proof}

For $\gamma\in\Gamma$ denote by $u_{\gamma}\in C_{c}(\Gamma,M(B))$ the function which is $1$ at $\gamma$ and $0$ elsewhere. We identify $B\subset C_{c}(\Gamma, B)$ with the function that takes the value $b$ at $e\in\Gamma$ and $0$ elsewhere.
\begin{proposition}\label{explicithecke} The left $B\rtimes_{r}\Gamma$ module structure on  $T^{\Gamma}_{g}\simeq (B\rtimes_{r}\Gamma)^{d}$ is given by
\begin{equation}\label{fullheckerep}(t_{g}(f)\Psi)_{i}(\delta)=\sum_{\gamma} g_{i}^{-1}f(\gamma)t_{i}(\gamma^{-1})^{-1}\Psi_{\gamma^{-1}(i)}(t_{i}(\gamma^{-1})\delta).\end{equation}
Equivalently, we have the covariant representation 
\begin{equation}\label{heckerep}(t_{g}(b)\cdot\Psi)_{i}(\delta):= g_{i}^{-1}(b)\Psi_{i}(\delta),\quad (t_{g}(u_{\gamma})\Psi)_{i}(\delta):= t_{i}(\gamma^{-1})^{-1}(\Psi_{\gamma^{-1}(i)}(t_{i}(\gamma^{-1})\delta )). \end{equation}

Moreover, for $\gamma\in\Gamma$ we have a factorisation $t_{g}(u_{\gamma})=\tau(\gamma)\diag(u_{t_{k}(\gamma)})$, where $\tau(\gamma)\in M_{d}(\C)$ is a permutation matrix.
\end{proposition}
\begin{proof} By right $B\times_{r}\Gamma$ linearity, it suffices to prove \eqref{fullheckerep} for elements $\Psi$ with $\supp\Psi\subset \Gamma_{g}$.
Using $u$ as in \eqref{heckefree} and the relations in Lemma \ref{Heckerelations} one computes, for $h\in C_{c}(\Gamma,B)$ and $\delta\in\Gamma_{g}$:
\begin{align}\nonumber (t_{g}(h)\Psi)_{i}(\delta)&=(u (h*u^{*}\Psi))_{i}(\delta)=\Ad_{g}\rho(\chi_{i}^{*}*h*u^{*}\Psi)(\delta) \\
\nonumber&=g\rho(\chi_{i}^{*}*h*u^{*}\Psi)(g^{-1}\delta g) =g\delta_{i}^{-1}(h*u^{*}\Psi)(\delta_{i}g^{-1}\delta g) \\
&\label{support}=\sum_{j,\gamma}g\delta_{i}^{-1}(h(\gamma)\gamma \delta_{j}g^{-1} \Psi_{j}(g\delta_{j}^{-1}\gamma^{-1}\delta_{i}g^{-1}\delta ))\\
&\label{simplification}=\sum_{\gamma}g\delta_{i}^{-1}h(\gamma)g\delta_{i}^{-1}\gamma \delta_{\gamma^{-1}(i)}g^{-1} \Psi_{\gamma^{-1}(i)}(g\delta_{\gamma^{-1}(i)}^{-1}\gamma^{-1}\delta_{i}g^{-1}\delta )\\
&=\sum_{\gamma} g_{i}^{-1}h(\gamma)t_{i}(\gamma^{-1})^{-1}\Psi_{\gamma^{-1}(i)}(t_{i}(\gamma^{-1})\delta).\nonumber
\end{align}

The step from \eqref{support} to \eqref{simplification} follows since $g\delta_{j}^{-1}\gamma^{-1}\delta_{i}g^{-1}\delta\in\Gamma_{g}\Leftrightarrow j=\gamma^{-1}(i)$. Thus we have established \eqref{fullheckerep} and \eqref{heckerep} follows. Let $\tau(\gamma)\in M_{d}(\C)$ be the permutation matrix corresponding to $(\tau(\gamma)\Psi)_{i}=\Psi_{\gamma^{-1}(i)}$. To prove the last statement we compute
\begin{align*}(\tau(\gamma)\diag(u_{t_{k}(\gamma)})\Psi)_{i}(\delta)&=(\diag (u_{t_{k}(\gamma)})\Psi)_{\gamma^{-1}(i)}(\delta)=t_{\gamma^{-1}(i)}(\gamma)^{-1}(\Psi_{\gamma^{-1}(i)}(t_{\gamma^{-1}(i)}(\gamma)\delta))\\
&=t_{i}(\gamma^{-1})^{-1}(\Psi_{\gamma^{-1}(i)}(t_{i}(\gamma^{-1})\delta))=(t_{g}(u_{\gamma})\Psi)_{i}(\delta),\end{align*}
as required.
\end{proof}

\section{Gysin sequence and Hecke operators}\label{section: gysin}

The paper \cite{EM} is an extensive study of the Gysin sequence in $K$-theory arising from a group action on a space $X$ and the associated boundary action on $\partial X$, e.g. the Furstenberg or Gromov  boundary. In \cite[Section 10]{EN} , the $K$-homological version is described for hyperbolic groups $\Gamma$ with cocompact classifying space for proper actions $\mathcal{E}\Gamma$. We will describe the Gysin sequence in the setting of hyperbolic $n+1$-space $\h$, the geodesic compactification $\overline{\h}$ and its boundary sphere $\partial\h=S^{n}$.
\subsection{The $K$-homology exact sequence}
Let $G=\textnormal{Isom}(\h)$ and $\overline{\h}:=\h \cup \partial\h$ the geodesic compactification of $\h$ on which $G$ acts as well. We consider the $G$-equivariant extension
\begin{equation}\label{bdry} 0\to C_{0}(\h)\to C(\overline{\h}) \to C(\partial \h)\to 0,\end{equation}
defining a class in $KK_{1}^{G}(C(\partial \h), C_{0}(\h))$. Thus, for any subgroup $\Gamma\subset G$ we obtain a class in  $KK_{1}^{\Gamma}(C(\partial \h), C_{0}(\h))$ through restriction, and a long exact sequence in equivariant  $K$-homology:
\begin{equation}\label{equivariantexact}\cdots\rightarrow  K^{i}_{\Gamma}(C_{0}(\h))\rightarrow K^{i+1}_{\Gamma}(C(\dH))\rightarrow K^{i+1}_{\Gamma}(C(\overline{\h}))\rightarrow \cdots\end{equation}

\begin{lemma}\label{contractible} Suppose $\Gamma\subset G$ is discrete and torsion-free. Then the  inclusion $i:\C\rightarrow C(\overline{\h})$ as constant functions induces an isomorphism $i^{*}:K^{i}_{\Gamma}(C(\overline{\h}))\rightarrow K^{i}_{\Gamma}(\C)$.
\end{lemma}
\begin{proof} Because $\Gamma$ is torsion-free and $\overline{\h}$ is contractible, $H$-equivariant contractibility \cite{MeyNes} of $\overline{\h}$ for finite subgroups $H\subset \Gamma$ follows trivially. Then the  argument of \cite[Lemma 10.6]{EN} applies.
\end{proof}

The extension \eqref{bdry} induces an extension of crossed products
\begin{equation}\label{descentbdry} 0\to C_{0}(\h)\rtimes\Gamma\to C(\overline{\h})\rtimes\Gamma \to C(\partial \h)\rtimes\Gamma\to 0,\end{equation}
as the $\Gamma$-action on either of the algebras in \eqref{bdry} is amenable, and the full and reduced crossed products coincide. Let $L^{2}(\wedge^{*}\h)$ be the Hilbert space of $L^{2}$-sections of the exterior algebra bundle of $\h$, and $\Dsla_{HR}$ the Hodge-DeRham operator. The triple $[(\C, L^{2}(\wedge^{*}\h), \Dsla_{HR})]$ defines an element in $KK^{G}_{0}(\C,\C)$  and thus in $KK^{\Gamma}_{0}(\C,\C)$ for any subgroup $\Gamma\subset G$. We will refer to each of these elements as the \emph{Euler class} (cf. \cite{EM, EN}). We obtain the following Proposition.

\begin{proposition} For a discrete torsion-free subgroup $\Gamma\subset G$ there is an exact hexagon
\begin{equation}\label{Gysin} 
\xymatrix{ K^{1}(C_{0}(M)) \ar[r]^-{\partial} & K^{0}(C(\dH)\rtimes\Gamma) \ar[r]^-{i^{*}} & K^{0}(C^{*}_{r}(\Gamma)) \ar[d]^{\textnormal{Eul}_{0}}\\ 
K^{1}(C^{*}_{r}(\Gamma)) \ar[u]^{\textnormal{Eul}_{1}} & K^{1}(C(\dH)\rtimes \Gamma) \ar[l]_-{i^{*}} & K^{0}(C_{0}(M))  \ar[l]_-{\partial} ,}\end{equation}
where $i^{*}$ is induced from the inclusion  $i:\C\rightarrow C(\overline{\h})$ and the maps $\Eul_{*}$ are induced from the Kasparov product with the Euler class $[(\C,L^{2}(\wedge^{*}\h), \Dsla_{HR})]$.
\end{proposition}

\begin{proof} This follows from the arguments in \cite{EM,EN}. Since $\gamma=1\in KK^{\Gamma}_{0}(\C,\C),$ there is an isomorphism $K^{i}_{\Gamma}(\C)\cong K^{i}(C^{*}_{r}(\Gamma))$, by descent in the first variable. Thus, Lemma \ref{contractible} gives isomorphisms
\[K^{*}(C(\overline{\h})\rtimes\Gamma)\xrightarrow{\sim} K^{*}_{\Gamma}(\C)\xrightarrow{\sim} K^{*}(C^{*}_{r}(\Gamma)). \]
Because the action of $\Gamma$ on $\h$ is free and proper, we have isomorphisms
\[K^{*}_{\Gamma}(C_{0}(\h))\xrightarrow{\sim} K^{*}(C_{0}(\h)\rtimes\Gamma)\xrightarrow{\sim} K^{*}(C_{0}(M)).\]
Lastly, \cite[Lemma 3.8]{Lott} and $\gamma=1$ give
\[K^{*}_{\Gamma}(C(\dH))\xrightarrow{\sim} K^{*}(C(\dH)\rtimes\Gamma).\]
Via these isomorphisms, the sequence \eqref{equivariantexact} can be identified with the six-term exact sequence associated to the extension \eqref{descentbdry}. The identification of the maps $K^{*}(C^{*}_{r}(\Gamma))\to K^{*}(C_{0}(M))$ as induced by taking the Kasparov product with the Euler class now follows by combining the argument in \cite[Proposition 9]{EM} with \cite[Theorem 38]{EM}, yielding \eqref{Gysin}.\end{proof}
The exact sequence \eqref{Gysin} simplifies further. We denote by $[\pt]\in K^{0}(C_{0}(M))$ the class given by the homomorphism $C_{0}(M)\to \C, f\mapsto f(x)$, for some $x\in M$. Since $M$ is connected this does not depend on the choice of $x$. Furthermore we denote by
\[\chi(M):=\sum_{k=0}^{\textnormal{dim} M} (-1)^{k}\textnormal{rank }H^{k}(M,\Z),\]
the Euler characteristic of $M$. The following result uses the method of \cite[Theorem 10.7]{EN}.

\begin{theorem}For a discrete torsion-free subgroup $\Gamma\subset \textnormal{Isom}(\h)$, the homomorphism \\  $\Eul_{1}:K^{1}(C^{*}_{r}(\Gamma))\to K^{1}(C_{0}(M))$ vanishes and $\Eul_{0}$ is given by $$\Eul_{0}:K^{0}(C^{*}_{r}(\Gamma))\to K^{0}(C_{0}(M)),\quad [(C^{*}_{r}(\Gamma), \mathcal{H}, D)]\mapsto \chi(M)\textnormal{Ind}(D^{+})[\pt].$$ In particular, if $\Gamma$ is noncocompact or $\h$ has odd dimension, there are short exact sequences
\begin{equation}\label{Gysineven} 0\rightarrow K^{1}(C_{0}(M))\xrightarrow{\partial} K^{0}(C(\dH)\rtimes\Gamma)\xrightarrow{i^{*}}  K^{0}(C^{*}_{r}(\Gamma))\rightarrow 0,\end{equation}
\begin{equation}\label{Gysinodd}0\rightarrow K^{0}(C_{0}(M))\xrightarrow{\partial} K^{1}(C(\dH)\rtimes \Gamma)\xrightarrow{i^{*}}  K^{1}(C^{*}_{r}(\Gamma))\rightarrow 0.\end{equation}
\end{theorem}

\begin{proof} The exact sequences  \eqref{Gysineven} and \eqref{Gysinodd} are derived directly from Proposition \ref{Gysin}. We now prove the statements about the maps $\Eul_{*}$. Let $x\in M$, $\pi:\h\to M$ the quotient map and $\rho_{x}:C_{0}(\h)\rtimes\Gamma \to B(\ell^{2}(\pi^{-1}(x)))$ the induced representation. By \cite[Example 24]{EM} and \cite[Theorem 30]{EM} we find that
\begin{equation}\label{Eulersimple}[(C_{0}(\h), L^{2}(\wedge^{*}\h),\Dsla_{HR})]=\chi(M)[\pi_{x}]\in KK_{0}^{\Gamma}(C_{0}(\h),\C).\end{equation}
There is a factorisation $[\pi_{x}]=[\phi_{x}]\otimes [1]$ where $[1]\in KK^{\Gamma}_{0}(C_{0}(\Gamma), \C)\simeq \Z$ is the class of the map $C_{0}(\Gamma)\to B(\ell^{2}(\Gamma))$ and $\phi_{x}:C_{0}(\h)\to C_{0}(\Gamma)$ is defined through $\Gamma\to \h, \gamma\mapsto x\gamma$. This yields the explicit form of $\Eul_{0}$. Since $KK_{1}^{\Gamma}(C_{0}(\Gamma),\C)=KK_{1}(C_{0}(\Gamma)\rtimes\Gamma,\C)=KK_{1}(\mathbb{K}(\ell^{2}(\Gamma)),\C)=0$, the statement $\Eul_{1}=0$ follows.
\end{proof}

In particular the Gysin sequence simplifies for all Bianchi groups and some Fuchsian groups.

\subsection{The extension class}\label{subsec: extensionclass}  For a subgroup $H\subset G$, we denote the class defined through the exact sequence \eqref{bdry} by $[\Ext]\in KK_{1}^{H}(C(\partial\h), C_{0}(\h))$. 
We now construct an equivariant Kasparov module representing $[\Ext]$, and then employ Kasparov descent and Morita equivalence to obtain an explicit representative $[\partial]\in KK_{1}(C(\partial \h)\rtimes \Gamma, C_{0}(M))$ for torsion-free discrete subgroups $\Gamma\subset G$.

In the Poincar\'{e} ball model of hyperbolic $n+1$ space $\h$, the boundary $\partial\h$ is the unit sphere in $\R^{n+1}$.  For an element $g\in G$, write  $|g'(\xi)|=|\det J_{g}(\xi)|$, the determinant of the Jacobian of the conformal transformation $g$. Consider $T_{1}\h:=\h \times\partial \h$, which can be thought of as the unit tangent bundle of $\h$. The \emph{Poisson kernel} is the map 
\begin{equation}\label{Poissonkernel} P:T_{1}\h\to (0,\infty),\quad P(x,\xi):=\frac{1-\|x\|^{2}}{\|x-\xi\|^{2}},\end{equation}
which for $g\in G$ satisfies the transformation rule 
\begin{equation}\label{Pg} P(xg,\xi g)=|g'(\xi)|^{-1}P(x,\xi)\end{equation} (see \cite[Equation 5.1.2]{Nicholls}).  
The \emph{harmonic measure} $\nu_{x}$ on $\partial\h$ based at $x\in\h$ is defined to be unique probabiltity measure on $\partial\h$ that is invariant under the action of the stabiliser $G_{x}$ of $x$. 
Then $\nu_{0}$ is normalised Lebesgue measure on $\partial\h$ and the measures $\nu_{x}$ satisfy 
\begin{equation}\label{Poissonprops} d\nu_{x}(\xi)=P(x,\xi)^{n}d\nu_{0}(\xi),\quad d\nu_{xg}(\xi g)=d\nu_{x}(\xi).\end{equation}
We consider the $G-C^{*}$-algebra $C_{0}(\h)$ as a $G$-equivariant $C^{*}$-module over itself. 
A second $C^{*}$-module is constructed using the harmonic measures on the boundary. The harmonic measures give an expectation
\[C_{c}(T_{1}\h)\rightarrow C_{c}(\h),
\quad \rho(\Psi)(x,\xi):=\int \Psi(\xi, x)d\nu_{x}(\xi),\]
and hence a $C_{0}(\h)$-module $L^{2}(T_{1} \h,\nu_x)_{C_{0}( \h)}$. This module carries a representation of the boundary algebra $C(\partial \h)$ by pointwise multiplication. 

\begin{theorem}\label{extrep}Let $p=ww^{*}$ be the projection defined from the adjointable isometry 
$$w: C_{0}(\h)\to L^{2} (T_{1}\h,\nu_x)_{C_{0}(\h)},\quad w\Psi(x,\xi):=\Psi(x).$$
Then $p$ commutes with $G$ and has compact commutators with the $C(\dH)$-representation. The triple $(C(\dH), L^{2}(T_{1} \h)_{C_{0}(\h)}, F_{p})$, with $F_{p}:=2p-1$ is a $G$-equivariant $KK$-cycle for $(C(\dH),C_{0}(\h))$ representing the class of the boundary extension \eqref{bdry}
\end{theorem}
\begin{proof} To see that $w$ is adjointable define $w^{*}f(x):=\int_{\partial{\h}} f(\xi,x)d\nu_{x}\xi$. A quick computation shows that $w$ and $w^{*}$ are mutually adjoint:
\begin{align*}\langle w\Psi,\Phi\rangle= \int \overline{w\Psi(\xi,x)}\Phi(\xi, x)d\nu_{x}(\xi)
= \int \overline{\Psi(x)}\Phi(\xi, x)d\nu_{x}(\xi)
=\langle \Psi, w^{*}\Phi\rangle.
\end{align*}
Then computing the composition $w^{*}wf(x)=\int wf(\xi,p)d\nu_{x}(\xi)=\int f(x)d\nu_{x}(\xi)=f(x)$, that is $w^{*}w=1$ and $w$ is an isometry.
Consider 
\begin{equation}\label{difference}w^{*}fw\Psi(x)=\left(\int f(\xi)d\nu_{x}\xi\right)\Psi(x).\end{equation}
The function $x\mapsto \int f(\xi)d\nu_{x}\xi$ is hyperbolically harmonic and continuous up to the boundary with limit $f$ by \cite[page 69]{Ahlfors} (see also \cite[Theorem 5.1.5]{Nicholls}). Hence $f\to w^{*}fw$ defines a $G$-invariant completely positive linear splitting $C(\partial\h)\to C_0(\h)$.

Thus $(C(\dH), L^{2}(T_{1}\h,\nu_x)_{C_{0}( \h)}, F_{p})$ is a $G$-equivariant Kasparov module (see for example \cite[Lemma 3.7]{extensionsgoffeng}), and the usual Stinespring dilation argument shows that it represents the extension \eqref{bdry}.
\end{proof}
\subsection{Kasparov descent and Morita equivalence}
Consider the universal cover $\pi:\h\rightarrow M$, and the associated expectation
\begin{equation}\label{rhoinnprod} \rho_{M}:C_{c}(\h) \rightarrow C_{c}(M),\quad 
\rho_{M}(\Psi)(m) := \sum _{h\in \pi^{-1}(m)} \Psi(h),\end{equation}
defining a $C_{c}(M)$-valued inner product on $C_{c}(\h)$ by $\langle \Phi,\Psi\rangle:=\rho_{\Gamma}(\overline{\Phi}\Psi)$. Denote its completion by $L^{2}_{\pi}(\h)_{C_{0}(M)}$.The following result is a special case of the well known Morita equivalence for free and proper actions.

\begin{lemma}\label{Morita} The $C^{*}$-algebra $\mathbb{K}(L^{2}_{\pi}(\h)_{C_{0}(M)})$ is isomorphic to $C_{0}(\h)\rtimes \Gamma$, implementing the Morita equivalence with $C_{0}(M)$. The $C^{*}$-algebra $C(\overline{\h})\rtimes \Gamma$ acts faithfully on $L^{2}_{\pi}(\h)_{C_{0}(M)}$.
\end{lemma}
The well known descent homomorphism \cite[Theorem 6.1]{kasparovconspectus} is a map
\begin{equation}\label{descent}j_{\Gamma}:KK^{\Gamma}_{*}(A,B)\to KK_{*}(A\rtimes \Gamma, B\rtimes \Gamma),\end{equation}
which can be explicitly defined on the level of cycles. We will describe the image of the cycle $(C(\partial\h),L^{2}(T_{1} \h,\nu_{x})_{C_{0}(\h)}, F_{p})$ from Theorem \ref{extrep} under the map $j_{\Gamma}$ as well as its composition with the Morita equivalence from Lemma \ref{Morita}. This will furnish us with a representative of the mapping $K^{*}(C_{0}(M))\to K^{*+1}(C(\partial\h)\rtimes\Gamma)$ appearing the in the exact sequences \eqref{Gysineven} and \eqref{Gysinodd}. 

Following \cite[Section 6.1]{kasparovconspectus}, the underlying $C_{0}(\h)\rtimes \Gamma$ module for the class $j_{\Gamma}([\Ext])$ is given as the completion of $C_{c}(T_{1}\h\times\Gamma)$ in the $C_{c}(\h)\rtimes\Gamma$-valued inner product
\[\langle \Phi,\Psi\rangle (x,\gamma):=\sum_{\delta\in\Gamma}\int_{\partial\h}\overline{\Phi(\xi ,x,\delta)}\Psi(\xi\delta^{-1},x\delta^{-1},\delta\gamma) d\nu_{x}(\xi),\]
and left and right module structures
\[(f\cdot\Psi)(\xi,x,\gamma):=\sum_{\delta\in\Gamma}f(\xi,\delta)\Psi(\xi\delta, x\delta, \delta^{-1}\gamma),\quad (\Psi\cdot g)(\xi,x,\gamma)=\sum_{\delta\in\Gamma}\Psi(\xi,x,\delta)g(x\delta,\delta^{-1}\gamma).\]The operator $F_{p}$ is defined by viewing $\Psi_{\gamma}(\xi, x):=\Psi(\xi,x,\gamma)$ as an element of $C_{c}(T_{1}\h)$ for each $\gamma\in\Gamma$. The product with the Morita equivalence $L^{2}_{\pi}(\h)_{C_{0}(M)}$ is now easily described. The map
\[m:C_{c}(T_{1}\h \times\Gamma)\otimes_{C_{c}(\h\times \Gamma)}C_{c}(\h)\to C_{c}(T_{1}\h),\quad m(\Psi\otimes \Phi)(\xi,x)=\sum_{\delta\in\Gamma}\Psi(\xi,x,\delta)\Phi(x
\delta) ,\]
is surjective and compatible with the balancing relation. The resulting $C_{0}(M)$-valued inner product on $C_{c}(T_{1}(\h))$ is given by 
\[\langle\Psi,\Phi\rangle(m):=\sum_{x\in\pi^{-1}(m)}\int_{\partial\h}\overline{\Psi(\xi,x)}\Phi(\xi,x)d\nu_{x}\xi.\]
The left representation and the operator $F_{p}=2p-1$ are induced from tensoring with the identity operator. We summarize the above findings:
\begin{corollary}[of Theorem \ref{extrep}]\label{delext} The class $$[\partial]:=j_{\Gamma}([\textnormal{Ext}])\otimes_{C_{0}(\h)\rtimes\Gamma} [L^{2}_{\pi}(\h)_{C_{0}(M)}]\in KK_{1}(C(\partial\h)\rtimes\Gamma, C_{0}(M)),$$ is represented by the bounded Kasparov module $(C(\partial\h)\rtimes \Gamma, L^{2}_{\pi}(T_{1} \h)_{C_{0}(M)}, F_{p})$.

\end{corollary}
Consequently the boundary map $\partial:K^{0}(C_{0}(M))\to K^{1}(C(\partial\h)\rtimes \Gamma)$ is implemented by the Kasparov product with $(C(\partial\h)\rtimes \Gamma, L^{2}(T_{1} \h,\nu_x)_{C_{0}(M)}, F_{p})$. 

\subsection{Hecke equivariance}\label{Heckeeq} Our purpose is now to show that the exact hexagon \eqref{Gysin} is equivariant for the action of the Hecke operator $T_{g}$ on the various algebras appearing in \eqref{Gysin}. We first consider compatibility of the Hecke operators with Morita equivalences arising from free and proper actions. 

Let $X$ be a $G$-space such that $\Gamma$ acts freely and properly on $X$, $\pi:X\to M:=X/\Gamma$ the covering map and $L^{2}_{\pi}(X)_{C_{0}(M)}$ the associated $C_{0}(X)\rtimes\Gamma$-$C_{0}(M)$ Morita equivalence bimodule. There is a well-known unitary isomorphism
\begin{equation} \label{moritafree} T^{\Gamma}_{g}\otimes_{C_{0}(X)\rtimes\Gamma}L^{2}_{\pi}(X)_{C_{0}(M)}\to L^{2}_{\pi}(X)_{C_{0}(M)}^{d},\quad (\Phi_{i})\otimes \Psi\mapsto (\Phi_{i}\Psi),\end{equation} of right $C_{0}(M)$-modules. We show that the same is true for $L^{2}_{\pi}(X)\otimes_{C_{0}(M)} T^{M}_{g}$ and then compare the left actions. Consider the \emph{fiber product} with its natural covering maps 
\[X \times_{\tau_{g}} M_{g}:=\{(x,m)\in X\times M_{g}: \pi(x)=\tau_{g}(m)\}, \quad M \xleftarrow{\pi} X \times_{\tau_{g}} M_{g}\xrightarrow{\pi_{g}} M,\]
making $C_{c}(X \times_{\tau_{g}} M_{g})$ into a $C_{c}(M)$ inner product bimodule. There is a well-defined map
$$ w:C_{c}(X)\otimes_{C_{c}(M)} C_{c}(M_{g})\to C_{c}(X \times_{\tau_{g}} M_{g}),\quad w( \Psi\otimes \Phi)(x,m)=\Psi(x)\Phi(m),$$ 
of right $C_{c}(M)$-modules preserving the inner product.
By standard support arguments, $w$ is shown to be surjective, and $L^{2}_{\pi}(X)\otimes_{C_{0}(M)} T^{M}_{g}$ is obtained as a completion of $C_{c}(X \times_{\tau_{g}} M_{g})$.

\begin{lemma}\label{unwindfiberproduct} Let $\delta_{i}$ be a set of right coset representatives for $\Gamma_{g^{-1}}$, $\Gamma=\bigsqcup_{i=1}^{d}\delta_{i}\Gamma_{g^{-1}}$. The continuous open maps
\[\varphi_{i}:X\to X \times_{\tau_{g}} M_{g}, \quad x\mapsto (xg\delta_{i}^{-1}, [x]),\quad i=1,\cdots ,d,\]
assemble to a homeomorphism
$\phi:\bigsqcup_{i=1}^{d} X\to X \times_{\tau_{g}} M_{g}.$
\end{lemma}
\begin{proof} Since $\tau_{g}([x])=\pi(xg)=\pi(xg\delta_{i}^{-1})$, each $\varphi_{i}$ is well-defined and injective. Moreover, the $\varphi_{i}$ assemble to an injective map on the disjoint union
\[\phi:\bigsqcup_{i=1}^{d} X\to X \times_{\tau_{g}} M_{g}.\] 
This can be seen by assuming that $\phi_{i}(x)=\phi_{j}(x')$ for some $x,x'\in X$. Then $$xg\delta_{i}^{-1}=x'g\delta_{j}^{-1},\quad [x]=[x']\Leftrightarrow \exists \gamma\in \Gamma_{g} \quad x\gamma=x' ,$$
which gives $xg\delta_{i}^{-1}=x\gamma g\delta_{j}^{-1}=xgg^{-1}\gamma g\delta_{j}^{-1}$.

Now $g^{-1}\gamma g\delta_{j}^{-1}\in\Gamma$ because $\gamma\in \Gamma_{g}$ and freeness of the $\Gamma$ action on $X$ gives $g^{-1}\gamma g\delta_{j}^{-1}=\delta_{i}^{-1}$. Hence we find $\delta_{i}^{-1}\delta_{j}=g^{-1}\gamma g\in \Gamma_{g^{-1}}$, which gives $i=j$ and hence $x=x'$ as well, as desired.

It remains to show that $\varphi$ is surjective, so let $(x',[x])\in X\times_{\tau_{g}} M_{g}$. Since $\pi(x)=\pi(x'g)$, there is  $\gamma'\in\Gamma$ such that $xg=x'\gamma'=x'\delta_{i}\gamma^{-1}$ for some unique $i$ and $\gamma\in\Gamma_{g^{-1}}$. Hence, we find $x'= xg\gamma \delta_{i}^{-1}=xg\gamma g^{-1}g\delta_{i}^{-1}$. Now in $M_{g}$ it holds that
$[x]=[xg\gamma g^{-1}],$ because $g\gamma g^{-1}\in \Gamma_{g}$. Therefore
$$(x,[x'])=(xg\gamma g^{-1}g\delta_{i}^{-1},[xg\gamma g^{-1}])=\phi_{i}(xg\gamma g^{-1}),$$
proving that $\phi$ is surjective.
\end{proof}
\begin{proposition}\label{HeckeMorita} Let $X$ be a $G$-space such that $\Gamma$ acts freely and properly on $X$, $M:=X/\Gamma$ and $L^{2}_{\pi}(X)_{C_{0}(M)}$ the associated $C_{0}(X)\rtimes\Gamma$-$C_{0}(M)$ Morita equivalence bimodule. For each $g\in C_G(\Gamma)$ there is a unitary isomorphism of $C_{0}(X)\rtimes\Gamma$-$C_{0}(M)$-bimodules
\[ T^{\Gamma}_{g}\otimes_{C_{0}(X)\rtimes\Gamma} L^{2}_{\pi}(X)\xrightarrow{\sim} L^{2}_{\pi}(X)\otimes_{C_{0}(M)}T^{M}_{g}.\]
In particular we have the identity $$[T^{\Gamma}_{g}]\otimes_{C_{0}(\h)\rtimes\Gamma} [L^{2}_{\pi}(X)]=[L^{2}_{\pi}(X)]\otimes_{C_{0}(M)} [T^{M}_{g}]\in KK_{0}(C_{0}(X)\rtimes\Gamma, C_{0}(M)).$$
\end{proposition}
\begin{proof}

The homeomorphism $\varphi$ from Lemma \ref{unwindfiberproduct} induces a unitary isomorphism of right $C_{0}(M)$-modules $$L^{2}_{\pi}(X)\otimes_{C_{0}(M)} T^{M}_{g}\cong L^{2}_{\pi}(X)^{d}\cong T^{\Gamma}_{g}\otimes_{C_0(X)\rtimes\Gamma}L^{2}_{\pi}(X)_{C_{0}(M)}.$$
As before we write elements of the module $L^{2}_{\pi}(\bigsqcup_{i=1}^{d} X)_{C_{0}(M)}$ as columns $\Psi=(\Psi_{i})_{i=1}^{d}$. The action of a function $f\in C_{0}(X)$ is given by
\[\phi^{*}f\phi^{*-1}(\Psi)_{i}(x)=f\phi^{*-1}(\Psi)((x\delta_{i}g^{-1},[x]))=(f(xg\delta_{i}^{-1})\Psi_{i}(x)),\]
which equals the action coming from the identification $L^{2}_{\pi}(X)^{d}\cong T^{\Gamma}_{g}\otimes_{C_0(X)\rtimes\Gamma}L^{2}_{\pi}(X)$.
The action of a group element $\gamma$ is given by
\[\phi^{*}u_{\gamma}\phi^{*-1}(\Psi)_{i}(x)=u_{\gamma}\phi^{*-1}(\Psi)(xg\delta_{i}^{-1},[x])=\phi^{*-1}(\Psi)(xg\delta_{i}^{-1}\gamma,[x]).\]
We compute further by using Lemma \ref{Heckerelations} and observing that
$$xg\delta_{i}^{-1}\gamma=xt_{i}(\gamma^{-1})^{-1}g\delta_{\gamma^{-1}(i)}^{-1},$$
and since $t_{i}(\gamma^{-1})^{-1}\in \Gamma_{g}$ we find $[x]=[xt_{i}(\gamma^{-1})^{-1}]$ so
\begin{align*}\phi^{*-1}(\Psi)(xg\delta_{i}^{-1}\gamma,[x])&=\phi^{*-1}(\Psi)(xt_{i}(\gamma^{-1})^{-1}g\delta_{\gamma^{-1}(i)}^{-1},[xt_{i}(\gamma^{-1})^{-1}])\\&=\Psi_{\gamma^{-1}(i)}(xt_{i}(\gamma^{-1})^{-1})=\Psi_{\gamma^{-1}(i)}(xt_{\gamma^{-1}(i)}(\gamma)).\end{align*}
As above, this equals the action coming from the identification $$L^{2}_{\pi}(X)^{d}\cong T^{\Gamma}_{g}\otimes_{C_0(X)\rtimes\Gamma}L^{2}_{\pi}(X).$$ This completes the proof. \end{proof}
As a right $C(\partial\h)\rtimes\Gamma$-module, $T^{\Gamma}_{g}$ is free of rank $d$. Therefore, as in \eqref{moritafree} the map 
\begin{equation}\label{unisofree}u:T^{\Gamma}_{g}\otimes_{C(\partial\h)\rtimes\Gamma} L^{2}(T_{1}\h\rtimes\Gamma,\nu_x)\to L^{2}(T_{1}\h\rtimes\Gamma,\nu_x)^{d},\quad  (\Psi_{i})\otimes f \mapsto (\Psi_{i}  f),\end{equation}
defined through coordinatewise product, is a unitary right module map. Using this map we can define the operator $u^{*}\textnormal{diag}(p)u$ on $T^{\Gamma}_{g}\otimes_{C(\partial\h)\rtimes\Gamma} L^{2}(T_{1}\h\rtimes\Gamma,\nu_x)$. By a slight abuse of notation, we denote this operator by $1\otimes p$.

\begin{theorem}\label{bdryeq} There is a unitary isomorphism of $(C(\partial\h)\rtimes\Gamma,C_{0}(\h)\rtimes\Gamma)$-bimodules
\[L^{2}(T_{1}\h\rtimes\Gamma,\nu_{x})\otimes_{C_{0}(\h)\rtimes\Gamma}T^{\Gamma}_{g}\xrightarrow{\sim} T^{\Gamma}_{g}\otimes_{C(\partial\h)\rtimes\Gamma}L^{2}(T_{1}\h\rtimes\Gamma,\nu_x),\]
intertwining the operators $p\otimes 1$ and $1\otimes p$. 
We have the identity
\[[T_{g}^{\Gamma}]\otimes [\partial] = [\partial]\otimes [T^{M}_{g}]\in KK_{1}(C(\partial\h)\rtimes\Gamma, C_{0}(M)).\]
In particular, the boundary map $\partial:K^{0}(C_{0}(M))\to K^{1}(C(\partial\h)\rtimes\Gamma)$ is Hecke equivariant.
\end{theorem}
\begin{proof} By Corollary \ref{delext} and Lemma \ref{HeckeMorita}, the second statement follows from the first because it implies that $$j_{\Gamma}([\Ext])\otimes [T^{\Gamma}_{g}]=[T^{\Gamma}_{g}]\otimes j_{\Gamma}([\Ext]),$$
and the class $j_{\Gamma}([\Ext])$ is represented by the Kasparov module $(L^{2}(T_{1}\h \rtimes\Gamma,\nu_x)_{C_{0}(\h)\rtimes\Gamma}, F_p)$. First we compare the bimodules $L^{2}(T_{1}\h\rtimes\Gamma,\nu_x)\otimes_{C_{0}(\h)\rtimes\Gamma} T^{\Gamma}_{g}$ and $T^{\Gamma}_{g}\otimes_{C(\partial\h)\rtimes\Gamma} L^{2}(T_{1}\h\rtimes\Gamma,\nu_x)$. 
The right $C_{0}(\h)\rtimes \Gamma$-module $L^{2}(T_{1}\h\rtimes\Gamma,\nu_x)_{C_{0}(\h)\rtimes\Gamma}^{d}$ is a left $B\rtimes_{r}\Gamma$ module for either of the $\Gamma$-$C^{*}$-algebras $B=C(\partial\h),C_{0}(\h),C_{0}(T_{1}\h)$ via Equation \eqref{heckerep}. The map $u$ in  Equation \eqref{unisofree} is readily seen to be a left $C(\partial\h)\rtimes\Gamma$ module map.

We now construct a unitary isomorphism of $(C(\partial\h)\rtimes\Gamma, C_{0}(\h)\rtimes\Gamma)$-bimodules
$$\alpha:L^{2}(T_{1}\h\rtimes\Gamma,\nu_x)\otimes_{C_{0}(\h)\rtimes\Gamma} T^{\Gamma}_{g}\xrightarrow{\sim} L^{2}(T_{1}\h\rtimes\Gamma,\nu_x)_{C_{0}(\h)\rtimes\Gamma}^{d}.$$
To achieve this we view $L^{2}(T_{1}\h\rtimes\Gamma,\nu_x)$ as a completion of $C_{c}(T_{1}\h\times\Gamma)$ and we consider the embedding of right $C_{0}(\h)\rtimes\Gamma$ bimodules $$\beta :T^{\Gamma}_{g}=(C_{0}(\h)\rtimes\Gamma)^{d}\to L^{2}(T_{1}\h\rtimes\Gamma,\nu_x)^{d},\quad \beta(\Psi_{i})(x,\xi,\gamma):=(\Psi_{i}(x,\gamma)).$$ 
We view $L^{2}(T_{1}\h\rtimes\Gamma,\nu_x)^{d}$ as a left $C_{c}(X\rtimes\Gamma)$-module, where $X=T_{1}\h$ or $X=\partial\h$, via 
$$C_{c}(X\rtimes\Gamma)\times L^{2}(T_{1}\h\rtimes\Gamma,\nu_x)^{d}\to L^{2}(T_{1}\h\rtimes\Gamma,\nu_x)^{d},\quad  f\cdot (\Psi_{i}):= t_{g}(f)(\Psi_{i}),$$
using Equation \eqref{heckerep} in Proposition \ref{explicithecke}. Now define a map
\[\alpha:C_{c}(T_{1}\h\times\Gamma)\otimes_{C_{c}(\h\rtimes\Gamma)} T^{\Gamma}_{g}\to L^2(T_{1}\h\times\Gamma,\nu_x)^{d},\quad\alpha(f\otimes (\Psi_{i})):=f\cdot\beta(\Psi_{i}), \]
which respects the  $(C_{c}(\partial\h\times\Gamma),C_{c}(\h\times\Gamma))$ bimodule structures because
\[\alpha(f * h \otimes (\Psi_{i}))=t_{g}(f * h)(\beta(\Psi_{i}))=t_{g}(f)t_{g}(h)\beta(\Psi_{i})=t_{g}(f)\alpha(h\otimes(\Psi_{i})),\]
and the right module structure is respected because $\beta$ is a right module map. We find
\begin{align*}\alpha(pf\otimes (\Psi_{i}))(\delta)&=t_{g}(pf)\beta(\Psi_{i})(\delta)=\sum_{\gamma} g_{i}^{-1}pf(\gamma)t_{i}(\gamma^{-1})^{-1}\beta\Psi_{\gamma^{-1}(i)}(t_{i}(\gamma^{-1})\delta)\\
&=p\left(\sum_{\gamma} g_{i}^{-1}f(\gamma)t_{i}(\gamma^{-1})^{-1}\beta\Psi_{\gamma^{-1}(i)}(t_{i}(\gamma^{-1})\delta)\right)=p(t_{g}(f)\beta(\Psi_{i})),\end{align*}
by Proposition \ref{explicithecke} using that  $p(\Phi \beta(\Psi_{i}))=(p\Phi)(\beta(\Psi))_{i}$ and $p$ is $G$-invariant. Thus,
\[ \textnormal{diag}(p)\alpha(f\otimes(\Psi_{i}))=\alpha(pf\otimes (\Psi_{i})),\quad (\textnormal{diag}(p)) \circ \alpha=\alpha\circ (p\otimes 1).\]
The fact that $\alpha$ is unitary follows from a lengthy but straightforward calculation which we omit.
\end{proof} 
We arrive at the main general result of this section, expressing the compatibility of the of the various Hecke operators we construct.
\begin{theorem}\label{thm: GysinHecke} The Gysin-sequences in $K$-homology 
\begin{equation}\nonumber
\xymatrix{ 0\ar[r] &K^{1}(C_{0}(M)) \ar[r]^{\partial \ \ \ } & K^{0}(C(\partial\h)\rtimes\Gamma) \ar[r] & K^{0}(C^{*}_{r}(\Gamma)) \ar[d] \\ 
0&\ar[l] K^{1}(C^{*}_{r}(\Gamma))  & K^{1}(C(\partial\h)\rtimes \Gamma) \ar[l] & K^{0}(C_{0}(M))  \ar[l]_{\ \ \ \ \partial} }
\end{equation}
and $K$-theory
\begin{equation}\nonumber
\xymatrix{ 0\ar[r] & K_1(C^{*}_{r}(\Gamma)) \ar[r] & K_1(C(\partial\h)\rtimes\Gamma) \ar[r]^{\ \ \ \partial} & K_{0}(C_{0}(M)) \ar[d] \\ 
0&\ar[l] K_1(C_{0}(M))  & K_0(C(\partial\h)\rtimes \Gamma) \ar[l]_{\partial \ \ \ }  & K_0(C^{*}_{r}(\Gamma))  \ar[l]}
\end{equation}

are Hecke-equivariant.
\end{theorem}
\begin{proof} This follows by combining Proposition \ref{HeckeMorita}, \ref{bdryeq} and the observation that the inclusion $K_{*}(C^{*}_{r}(\Gamma))\to K_{*}(C(\partial\h)\rtimes\Gamma))$ and restriction $K^{*}(C(\partial\h)\rtimes\Gamma)\to K^{*}(C^{*}_{r}(\Gamma))$ are Hecke equivariant by construction. Hecke equivariance of the Euler maps $K^{*}(C_{0}(M))\to K^{*}(C^{*}_{r}(\Gamma))$ follows from the commutative diagram
\begin{equation}\nonumber
\xymatrix{ K^{*}(C_{r}^{*}(\Gamma))\ar[d]_{\mbox{Eul}} \ar[r]^{(\iota^{*})^{-1}}   &  K^{*}(C(\overline{\h})\rtimes\Gamma)\ar[d]\\
 K^{*}(C_{0}(M))& \ar[l] K^{*}(C_{0}(\h)\rtimes\Gamma) }
\end{equation}
and since $T_{g}\iota^{*}=\iota^{*}T_{g}$ implies $(\iota^{*})^{-1}T_{g}=T_{g}(\iota_{*})^{-1}$, the map $\mbox{Eul}$ is a composition of Hecke equivariant maps, whence Hecke equivariant. The argument for the $K$-theory sequence is identical.
\end{proof}


\section{$K$-cycles for the reduced Bianchi group $C^{*}$-algebra}\label{section: algebraic}
We will describe a naturally defined map $s: H^{1}(\Gamma,\Z)\to K^{1}(C^{*}_{r}(\Gamma))$ for a discrete group $\Gamma$ of hyperbolic isometries and show that in the special case when $\Gamma$ is a torsion-free finite index subgroup of a Bianchi group, our explicit map 
$s$ is a Hecke equivariant isomorphism. 

It is well known that for a general discrete group $\Gamma$, there is a homomorphism $t: H_{1}(\Gamma,\Z)\to K_{1}(C^{*}_{r}(\Gamma))$. This homorphsim has been studied for instance by Matthey \cite{matthey} in the context of the Baum-Connes conjecture. We show that when $\Gamma$ is a torsion-free finite index subgroup of a Bianchi group, $t$ is an isomorphism and that the homological pairing 
$H^1 \times H_1 \rightarrow \Z$ and the index pairing $K^1 \times K_1 \rightarrow \Z$ commute with the isomorphisms $s$ and $t$.

\subsection{Group cocycles and index theory} In this subsection, $\Gamma$ is an arbitrary countable discrete group. Let $c:\Gamma\rightarrow \Z$ be an integral group cocycle, simply a group homomorphism, and denote by $\Gamma_{c}$ its kernel. The multiplication operator
\[D_{c}:C_{c}(\Gamma)\rightarrow C_{c}(\Gamma), \]
defined through $(D_{c}f)(\gamma):=-c(\gamma)f(\gamma)$ extends to a selfadjoint regular operator in the $C^{*}$-completion $E_{c}$ of $C_{c}(\Gamma)$ over $\Gamma_{c}$. This gives an unbounded Kasparov module $(E_{c},D_{c})$ and an element in the group $KK_{1}(C^{*}_{r}(\Gamma), C^{*}_{r}(\Gamma_{c}))$, as a special case of the construction in \cite[Theorem 3.2.2, Lemma 3.4.1]{Mescocycle},

To describe the pairing of this cycle with $K_{1}(C^{*}_{r}(\Gamma))$, we need a concrete description of the latter group. 
 First note that here we use the surjective (Hurewicz) map $\Gamma \rightarrow \Gamma^{ab} \simeq H_1(\Gamma,\Z)$, we can represent homology classes by elements $\delta \in \Gamma$. A group element $\delta\in \Gamma$ defines a unitary $u_{\delta}$ in the reduced $C^{*}$-algebra $C^{*}_{r}(\Gamma)$, and thus a class $[u_{\delta}]\in K_{1}(C^{*}_{r}(\Gamma))$ via the standard picture of $K_{1}$. This gives us a homomorphism $H_{1}(\Gamma,\Z)\to K_{1}(C^{*}_{r}(\Gamma))$, for any discrete group $\Gamma$.

\begin{definition}We define the \emph{norm} of a cocycle $c:\Gamma\rightarrow \Z$ to be the nonnegative integer
\[|c|:=\min\{|c(\gamma)|:\gamma\in\Gamma, \gamma\notin \Gamma_{c}\}.\]
A cocycle $c$ is  is \emph{normalised} if $1\in c(\Gamma)\subset \Z$. The norm of the $0$-cocycle is defined to be $\infty$.
\end{definition}
Since $c(\Gamma)=|c|\Z$, the statement that $c$ is normalised is equivalent to saying that $c(\Gamma)=\Z$. Any cocycle is an integral multiple of a normalised cocycle, and thus $H^{1}(\Gamma,\Z)$ is generated by normalised cocycles. If $c$ is normalised, 
the short exact sequence of groups
\[0\rightarrow \Gamma_{c}\rightarrow \Gamma\rightarrow \Z \rightarrow 0,\]
admits a non canonical splitting by choosing $g\in c^{-1}(1)$ and define
$s:\Z \rightarrow \Gamma,\quad 
n  \mapsto g^{n}.$
Any such splitting determines a group is isomorphism $\Gamma \cong \Gamma_{c}\rtimes_{s}\Z$ (semidirect product) and a $C^{*}$-algebra isomorphism $C^{*}_{r}(\Gamma)\cong C^{*}_{r}(\Gamma_{c})\rtimes \Z$.
\begin{proposition} \label{modularindex} Let $c:\Gamma\rightarrow \Z$ be a normalised cocycle. The Kasparov product
\[K_{1}(C^{*}_{r}(\Gamma))\times KK_{1}(C^{*}_{r}(\Gamma), C^{*}_{r}(\Gamma_{c}))\rightarrow K_{0}(C^{*}_{r}(\Gamma_{c}))
,\]
maps the pair $([u_{\delta}], [D_{c}])$ to the class 
\[\textnormal{sgn}(c(\delta))[ C^{*}_{r}(\Gamma_{c})^{|c(\delta)|}]=c(\delta)[1_{C^{*}_{r}(\Gamma_{c})}].\]

\end{proposition}
\begin{proof} Choosing $g\in c^{-1}(1)$ gives a generating set $\{e_{g^{n}}\}_{n\in\Z}$ for the module $E$ and a decomposition 
\begin{equation}\label{freemod} E\simeq \bigoplus_{n\in\Z} C^{*}_{r}(\Gamma_{c}), \quad e_{\gamma}\mapsto e_{g^{c(\gamma)}}u_{g^{-c(\gamma)}\gamma},\end{equation}
under which the operator $D_{c}$ becomes multiplication by $-n\in\Z$. Denote by $p_{c}:E_{c}\rightarrow E_{c}$ the projection onto the positive spectrum of $D_{c}$, which is adjointable by the decompisition \eqref{freemod}. The Fredholm operator given by $F_{c}:=D_{c}(1+D_{c}^{2})^{-\frac{1}{2}}$ is a compact perturbation of the adjointable  operator $S_{c}=2p_{c}-1$ and defines the same class as $[D_{c}]=[F_{c}]\in KK_{1}(C_{r}^{*}(\Gamma), C^{*}_{r}(\Gamma_{c}))$. For  $\delta\in\Gamma$, since  $u_{\delta}e_{g^{n}}=e_{g^{n+c(\delta)}}u_{g^{-n-c(\delta)}\delta g^{n}}$ it follows that 
\[\im p_{c} u_{\delta} p_{c} =\textnormal{span}\{e_{g^{n}}: n\leq \min\{0,-c(\delta)\} \},\]
which is a complemented submodule, and hence so is $\im p_{c}u_{\delta}^{*} p_{c}$. Thus by \cite[Ch. 3]{Lancebook} the operator $p_{c}u_{\delta}p_{c} +1-p_{c}$ admits a polar decomposition and by  \cite[Theorem 7.8]{KNR} and the argument in \cite[Lemma 2.1]{KNR}, the Kasparov product maybe computed as a higher index, that is
\[[u_{\delta}]\otimes [(E_{c}, D_{c})]=[\textnormal{ker}p_{c}u_{\delta}p_{c}]-[\coker p_{c}u_{\delta}p_{c}] \in K_{0}(C_{r}^{*}(\Gamma_{c})).\]
As above we have  
\[\ker p_{c}u_{\delta}p_{c}=\textnormal{span}\{e_{g^{n}}: -c(\delta)<n\leq 0\},\]
and since $c$ is normalised, this module is isomorphic to the free module of rank $|c(\delta)|$ over $C^{*}_{r}(\Gamma_{c})$ if $c(\delta)>0$ and $0$ otherwise. Since 
\[\coker p_{c}u_{\delta}p_{c}=\ker p_{c}u_{\delta^{-1}}p_{c},\]
the statement follows.\end{proof}

In order to obtain a genuine unbounded Fredholm module from a cocycle, we need to get rid of the algebra $C^{*}_{r}(\Gamma_{c})$ in Proposition \ref{modularindex}. It is not clear how to do this without making more assumptions on $\Gamma$. In the next subsection, we 
achieve this when $\Gamma$ is a discrete group of hyperbolic isometries.

\subsection{The unbounded $\gamma$-element}\label{geltsec}
 In case $\Gamma$ is a group of isometries of a simply connected, complete Riemannian manifold $X$ with nonpositive sectional curvature, Kasparov's Dirac/dual-Dirac construction \cite{kasparov} gives a canonical element $[ \gamma_{ X}] \in K^{0}(C^{*}_{r}(\Gamma_{c}))$. In this section we work with the real hyperbolic $n+1$-space $X=\h$, but this is not necessary. Let $\rho$ denote the function $\rho (x):=d_{\h}(0,x),$ $L^{2}(\wedge^{*}\h)$ the Hilbert space of $L^{2}$-sections of the exterior algebra bundle of $\h$, $\Dsla_{HR}$ the Hodge-DeRham operator, $\hat{c}$ the Clifford multiplication and $\dd$ the exterior derivative. 
 
The following lemma is well-known in the case $s=1$ and we state it for convenience.
\begin{lemma}[\cite{kasparov}]\label{unbddgamma} For $0<s\leq 1$ the triple $ \left(\C,L^{2}(\wedge^{*}\h),  \Dsla_{\rho,s}:=\Dsla_{HR} +\rho^{s}\hat{c}(\textnormal{d}(\rho)) \right)$ 
is a $G$-equivariant unbounded Fredholm module representing the class $[\gamma_{\h}]=1\in KK^{G}(\C,\C)$. In particular, for any discrete subgroup $\Gamma\subset G$ the triple
\[\left(C^{*}_{r}(\Gamma), L^{2}(\wedge^{*}\h),  \Dsla_{\rho,s} \right),\]
is an unbounded Fredholm module and 
$\textnormal{Ind} (\Dsla_{\rho,s}^{+})=1$.
\end{lemma}
\begin{proof} The statement follows from the observation that $(1+\rho)^{s}\hat{c}(\dd\rho)$ a bounded perturbation of $\rho^{s}\hat{c}(\dd\rho)$ and is a representative for Kasparov's dual Dirac element. The graded commutator
\[[\Dsla_{HR}, (1+\rho)^{s}\hat{c}(\dd\rho)]=[\Dsla_{HR},(1+\rho)^{s}]\hat{c}(\dd\rho)+(1+\rho)^{s}[\Dsla_{HR},\hat{c}(\dd\rho)],\]
is relatively bounded to $(1+\rho)^{s}$, so $\Dsla_{\rho,s}$ represents the Kasparov product of the Dirac and dual-Dirac element and is hence in the class of $[\gamma_{\h}]$. For a discrete subgroup $\Gamma\subset G$ its action on $\h$ defines a representation of $C^{*}_{r}(\Gamma)$ on $L^{2}(\wedge^{*}\h)$. The statement about the index is immediate.
\end{proof}
\begin{lemma}\label{awesomeestimate} For  $0<s\leq 1$ and   any element $g\in G$ we have the estimate
\[\|\rho^{s}[\hat{c}(\dd\rho),u_{g}]\|\leq 2d_{\h}^{s}(0,0g).\]
\end{lemma}
\begin{proof}By the proof of \cite[Lemma 5.3]{kasparov}, for $x\neq 0$ it holds that
\[\|[\hat{c}(\dd_{x}\rho),u_{g} ]\|\leq 2 d_{\h}(0,0g)(d_{\h}(0,x)+d_{\h}(0,xg))^{-1},\]
so we assume $0g\neq 0$ as well. This yields the estimate
\begin{align*}\|d_{\h}^{s}(0,x)[\hat{c}(\dd_{x}\rho),u_{g}]\|&\leq 2 d_{\h}(0,0g)d_{\h}^{s}(0,x)(d_{\h}(0,x)+d_{\h}(0,xg))^{-1}\\
&\leq2 d_{\h}(0,0g)(d_{\h}(0,x)+d_{\h}(0,xg))^{s-1}\\
&\leq 2d_{\h}(0,0g)d_{\h}(0,0g)^{s-1}=2d_{\h}(0,0g)^{s},
\end{align*}
which produces the claimed norm estimate.
\end{proof}

\begin{lemma} \label{peel}Let $0<s\leq 1$, $\gamma_{1},\cdots,\gamma_{k}\in \Gamma$ and $x\in \h$. Then
\[d_{\h}^{s}(x, x\gamma_{k}\cdots \gamma_{1})\leq \left(\sum_{i=1}^{k} d_{\h}(x,x\gamma_{i})\right)^{s}\leq \sum_{i=1}^{k} d_{\h}^{s}(x,x\gamma_{i}).\]
\end{lemma}
\begin{proof} This is a straightforward induction. For $k=1$ there is nothing to prove. Then for $k>1$ we write

\[\begin{split}d_{\h}^{s}(x,&x\gamma_{k}\cdots\gamma_{k} )
\leq \left(d_{ \h } (x, x\gamma_{k-1}\cdots \gamma_{1})
+d_{\h}(x\gamma_{k-1}\cdots \gamma_{1} , x\gamma_{k}\cdots \gamma_{1} )\right)^{s} \\
& =\left(d_{\h}(x, x\gamma_{k-1}\cdots \gamma_{1})+d_{\h}(x, x\gamma_{k} )\right)^{s}\leq\left(\sum_{i=1}^{k}d_{\h}(x,x\gamma_{i})\right)^{s}
 \leq \sum_{i=1}^{k}d_{\h}^{s}(x, x\gamma_{i} ) ,
\end{split}\]
which are the desired inequalities.
\end{proof}
We wish to construct the Kasparov product of the element $[D_{c}]\in KK_{1}(C^{*}_{r}(\Gamma), C^{*}_{r}(\Gamma_{c}))$ and $[\gamma_{\h}]\in K^{1}(C^{*}_{r}(\Gamma_{c}))$ in order to obtain an unbounded Fredholm module and a class in $K^{1}(C^{*}_{r}(\Gamma))$. In order to do this we define, for $g\in c^{-1}(|c|)$
\begin{equation}\label{conndef}1\otimes_{\nabla_{g}}\Dsla_{\rho,s}(e_{\gamma}\otimes \psi):=e_{g^{c(\gamma)}}\otimes \Dsla_{\rho,s}(u_{g^{-c(\gamma)}\gamma}\psi)\end{equation}
which is a densely defined symmetric operator with initial domain $C_{c}(\Gamma)\otimes_{C_{c}(\Gamma_{c})}\Dom \Dsla_{\rho,s}.$ We then consider the densely defined symmetric operator 
\begin{equation}\label{Dcdef}D_{c}\otimes \sigma +1\otimes_{\nabla}\Dsla_{\rho,s}=\begin{pmatrix} -c & 1\otimes_{\nabla_{g}}\Dsla^{+}_{\rho,s} \\ 1\otimes_{\nabla_{g}}\Dsla^{-}_{\rho,s} & c\end{pmatrix},\end{equation}
on the Hilbert space $E\otimes_{C^{*}_{r}(\Gamma_{c})}L^{2}(\wedge^{*}\h)$ with grading operator $\sigma$, decomposed according to even and odd forms  $L^{2}(\wedge ^{*}\h)=L^{2}(\wedge^{+}\h)\oplus L^{2}(\wedge^{-}\h)$. 

We recall from the appendix to \cite{GoffengMesland}, that the notion of unbounded Fredholm module can be loosened.
\begin{definition}\label{epdef} An \emph{unbounded Fredholm module} is a triple $(\mathcal{A},H,D)$, where
\begin{enumerate}
\item $\mathcal{A}$ is a $*$-algebra represented on the $\Z/2$-graded Hilbert space $H$;
\item $D$ is a self-adjoint operator such that $a(D\pm i)^{-1}\in\K(H)$;
\item for all $a\in\mathcal{A}$, $a\Dom D\subset \Dom D$ and  there exists $\epsilon>0$ such that $[D,a](1+D^{2})^{-\frac{1-\epsilon}{2}}$ and $(1+D^{2})^{-\frac{1-\epsilon}{2}}[D,a]$ extend to bounded operators.
\end{enumerate}
If $\epsilon$ can be chosen independent of $a\in\mathcal{A}$ then $(\mathcal{A},H,D)$ is called an $\epsilon$-unbounded Fredholm module.
\end{definition}
\begin{theorem}\label{spectral-triple} Let $\Gamma\subset\textnormal{Isom}(\h)$ be a discrete group,  $0<s<1$ and $c:\Gamma\rightarrow \Z$ be a normalised cocycle. The Kasparov product of the classes $[(E_{c},D_{c})]$ and $[\gamma_{\h} ]$ is represented by the $(1-s)$ unbounded Fredholm module
\[\left( C^{*}_{r}(\Gamma), E\otimes_{C^{*}_{r}(\Gamma_{c})} L^{2}(\wedge^{*}\h), \Dsla_{c,s}:=D_{c}\otimes \sigma + 1\otimes_{\nabla_{g}}\Dsla_{\rho,s}\right),\]
and in particular is independent of the choice of $g\in c^{-1}(1)$. 
\end{theorem}

\begin{proof} Essential self-adjointness and compact resolvent of the operator $\Dsla_{c,s}$ in \eqref{Dcdef} follows from  general considerations in \cite{MR}. It remains to show condition 3 of Definition \ref{epdef} is satisfied for the unitaries $u_{\gamma}$ generating $C^{*}_{r}(\Gamma)$ and $\epsilon=1-s$. For the operator $D_{c}\otimes \sigma$ this follows from the fact that $[D_{c},u_{\gamma}]$ defines an adjointbale operator on $E_{c}$. 

For $1\otimes_{\nabla_{g}}\slashed{D}_{\rho,s}$, Equation \eqref{conndef} shows that the commutator can be expressed as
\[[1\otimes_{\nabla_{g}}\slashed{D}_{\rho,s}, u_{\gamma} ] (e_{g^{n}}\otimes \psi)=e_{g^{n+c(\gamma)}}\otimes [\slashed{D}_{\rho,s},u_{g^{-n-c(\gamma)}\gamma g^{n}}]\psi.\]
The Hilbert space $E\otimes L^{2}(\wedge^{*}\h)$ decomposes as a direct sum
\[E\otimes_{C^{*}_{r}(\Gamma_{c})} L^{2}(\wedge^{*}\h)\cong \bigoplus_{n\in\Z}e_{g^{n}}\otimes L^{2}(\wedge^{*}\h),\]
and it suffices to control the supremum of the norms $\|\cdot \|_{n}$ of the operators
\[[1\otimes_{\nabla_{g}}\Dsla_{\rho,s}, u_{\gamma}](1+D^{2}_{c} + (1\otimes_{\nabla}\Dsla_{\rho,s})^{2})^{-\frac{s}{2}}: e_{g^{n}}\otimes L^{2}(\wedge^{*}\h)\to e_{g^{n+c(\gamma)}}\otimes L^{2}(\wedge^{*}\h).\]
To compute the commutator $[\slashed{D}_{\rho,s},u_{g^{-n-c(\gamma)}\gamma g^{n}}]$, we observe that $\slashed{D}_{\rho,s}=\Dsla_{HR}+\rho^{s}\hat{c}(\dd\rho)$, and $\Dsla_{HR}$ commutes with $u_{\delta}$ for all $\delta$. So we need only concern ourselves with the dual Dirac part. To this end we expand
\begin{align*}[\rho^{s}\hat{c}(\dd\rho),u_{g^{-n-c(\gamma)}\gamma g^{n}}]=[\rho^{s}, u_{g^{-n-c(\gamma)}\gamma g^{n}}]\hat{c}(\dd\rho)+\rho^{s}[\hat{c}(\dd\rho),u_{g^{-n-c(\gamma)}\gamma g^{n}}].\end{align*}
Since$\|\hat{c}(\dd\rho)\|=1$, the norm of the first term is controlled by
\[\sup_{x\in\h}|(d^{s}_{\h}(0,x)-d^{s}_{\h}(0, xg^{-n-c(\gamma)}\gamma g^{n})) |\leq d^{s}_{\h}(0, 0g^{-n-c(\gamma)}\gamma g^{n}),\] whereas Lemma \ref{awesomeestimate} takes care of the second term with
 the estimate
\begin{align*}\|\rho^{s}[\hat{c}(\dd\rho),u_{g^{-n-c(\gamma)}\gamma g^{n}}]\|\leq 2d_{\h}^{s}(0,0g^{-n-c(\gamma)}\gamma g^{n}).\end{align*}

Thus the size of the commutator is determined by the distance $d^{s}_{\h}(0, 0g^{-n-c(\gamma)}\gamma g^{n})$. Using lemma \ref{peel} we can estimate
\[\begin{split} d^{s}_{\h}(0, 0g^{-n-c(\gamma)}\gamma g^{n})& \leq 2d^{s}_{\h}(0, 0g^{n})+d^{s}_{\h}(0,0\gamma )+ d^{s}_{\h}(0,0g^{c(\gamma)})\\ & \leq 2n^{s} d^{s}_{\h}(0, 0g) + d^{s}_{\h}(0,0\gamma )+ d_{\h}^{s}(0,0g^{c(\gamma)}).\end{split}\]
 Thus, the norm  of the operator 
\[[1\otimes_{\nabla_{g}}\Dsla_{\rho,s},u_{\gamma}]:e_{g^{n}}\otimes L^{2}(\wedge^{*}\h)\to e_{g^{n+c(\gamma)}}\otimes L^{2}(\wedge^{*}\h), \]
satisfies $\| [1\otimes_{\nabla_{g}}\Dsla_{\rho,s},u_{\gamma}]\|_{n}\leq C_{\gamma}+2n^{s}$. Since we also have the estimate $$\|(1+D_{c}^{2}\otimes 1 +(1\otimes_{\nabla}\Dsla_{\rho,s})^{2})^{-\frac{s}{2}}\|_{n}\leq (1+n^{2})^{-\frac{s}{2}},$$
we find that
\[\sup_{n}\|[1\otimes_{\nabla_{g}}\Dsla_{\rho,s}, u_{\gamma}](1+D^{2}_{c} + (1\otimes_{\nabla}\Dsla_{\rho,s})^{2})^{-\frac{s}{2}}\|_{n}\leq 2d_{\h}^{s}(0,0g)+ C_{\gamma}.\]
The operator  $(1+D^{2}_{c} + (1\otimes_{\nabla}\Dsla_{\rho,s})^{2})^{-\frac{s}{2}}[1\otimes_{\nabla_{g}}\Dsla_{\rho,s}, u_{\gamma}]$ is  shown to be bounded by noting that $u_{\gamma}=u_{\gamma^{-1}}^{*}$.
\end{proof}

Our next result says that the index pairing between $K^{1}$ and $K_{1}$ when applied to the $K$-cycles constructed from group cocycles in Theorem \ref{spectral-triple} and unitaries $[u_{\delta}]\in K_{1}(C^{*}_{r}(\Gamma))$  recovers the pairing between $H_{1}$ and $H^{1}$. This result will play an important r\^{o}le in what follows.

\begin{proposition}\label{index-pairing} Let $\Gamma\subset\textnormal{Isom}(\h)$ be a discrete group and $c:\Gamma\to\Z$ a normalised cocycle. The index pairing 
\[K_{1}(C^{*}_{r}(\Gamma))\times K^{1}(C^{*}_{r}(\Gamma))\rightarrow\Z \]
maps the pair $([u_{\delta}],[\Dsla_{c,s}])$ to the integer $ c(\delta)$, and thus
recovers the (co)homology pairing
\[H_{1}(\Gamma,\Z)\times H^{1}(\Gamma,\Z)\rightarrow \Z.\]
\end{proposition}
\begin{proof}We use that the Kasparov product is associative:
\[ [u_{\delta}]\otimes[\Dsla_{c,s}]=[u_{\delta}]\otimes [\Dsla_{c,s}]\otimes [\gamma_{\h}],\]
and apply Proposition \ref{modularindex} and Lemma \ref{unbddgamma} to obtain that this equals
\[ c(\delta)[1_{C^{*}_{r}(\Gamma_{c})}]\otimes [\gamma_{\h}]=c(\delta)\textnormal{Ind}(\Dsla_{\rho,s}^{+})=c(\delta).\]
This proves the proposition.
\end{proof}

\subsection{A Hecke equivariant isomorphism}
We return to the specific setting of Bianchi groups in dimension 3. We saw in Propositions \ref{prop: spectral-small} and  \ref{prop: spectral-homology} that for a Bianchi group $\Gamma$ there are isomorphisms
$$K^{1}(C^{*}_{r}(\Gamma)) \simeq  H^1(\Gamma,\Z),\quad K_1(C^{*}_{r}(\Gamma)) \simeq H_1(\Gamma,\Z). $$ 
However for our purposes, these isomorphisms are not useful as they are only given abstractly. In this subsection we set out to show that the construction of the previous section gives explicit isomorphisms between the above. We prove 
their the Hecke equivariance and construct a section for the restriction map $K^{1}(C(\partial\h_{3})\rtimes\Gamma)\to K^{1}(C^{*}_{r}(\Gamma))$ in the Gysin sequence.

\begin{proposition}\label{Kthiso} Let $\Gamma\subset \psl(\C)$ be a noncocompact torsion-free discrete subgroup. The map $$H_{1}(\Gamma,\Z)\to K_{1}(C^{*}_{r}(\Gamma)),\quad [\delta]\mapsto [u_{\delta}],$$ is an isomorphism.
\end{proposition}
\begin{proof} The quotient manifold $M=\h_3/\Gamma$ is a model for $B\Gamma$. Since $H_{3}(\Gamma,\Z) =0$, by \cite[Propositon 2.1.ii)]{matthey} there is an isomorphism $\beta_{1}^{M}: H_{1}(\Gamma,\Z)\to K_{1}^{\geo}(M)$, which we compose with the Novikov assembly map $\nu^{\Gamma}_{1}:K_{1}^{\geo}(M)\to K_{1}(C^{*}_{r}(\Gamma))$. Since the Baum-Connes conjecture holds for $\Gamma$, $\nu^{\Gamma}_{1}$ is an isomorphism. The composition $\nu_{1}^{\Gamma}\circ\beta_{1}^{M}$ is shown to coincide with the map $[\delta]\mapsto [u_{\delta}]$ in \cite[Theorem 10.4]{bettaiebmattheyvalette}. \end{proof}

\begin{theorem}\label{algisos} Let $\Gamma\subset\psl(\C)$ be a noncocompact torsion-free discrete subgroup. The maps
\begin{center}
\begin{tabular}{rclcccrcl} 
$H_{1}(\Gamma,\Z)$ & $\longrightarrow$ & $K_{1}(C^{*}_{r}(\Gamma))$ && \text{and}  && $H^{1}(\Gamma,\Z)$ & $\longrightarrow$ & $K^{1}(C^{*}_{r}(\Gamma))$ \\
$[\delta]$ & $\mapsto$ & $[u_{\delta}]$ & &&& $[c]$ &  $\mapsto$ & $|c|\cdot [\Dsla_{c,s}]$ 
\end{tabular}
\end{center}
are isomorphisms compatible with the pairings of the respective groups.
\end{theorem}
\begin{proof}  Proposition \ref{Kthiso} gives the $K$-theory isomorphism. To show that the $K$-homology map is a homomorphism, we use that $C^{*}_{r}(\Gamma)$ is $KK$-equivalent to $C(\overline{\h}_3)\rtimes\Gamma$ which is in the bootstrap class. By the Universal Coefficient Theorem (UCT) \cite[Theorem 1.17, Corollary 1.18]{RosSch} there is a short exact sequence
\[0\to \textnormal{Ext}^{1}_{\Z}(K_{0}(C^{*}_{r}(\Gamma)),\Z)\to K^{1}(C^{*}_{r}(\Gamma))\xrightarrow{\otimes} \textnormal{Hom}(K_{1}(C_{r}^{*}(\Gamma)),\Z)\to 0, \]
where $\otimes$ denotes the map induced by the Kasparov product. By \eqref{evenKtheory} $K_{0}(C^{*}_{r}(\Gamma))$ is finitely generated and torsion-free, so the $\textnormal{Ext}$ group vanishes and $K^{1}(C^{*}_{r}(\Gamma))\cong \Hom(C_{r}^{*}(\Gamma),\Z)$. That is classes in the $K$-homology $K^{1}(C^{*}_{r}(\Gamma))$ are determined by the index pairing. For an arbitrary cocycle $c:\Gamma\to \Z$, $\frac{c}{|c|}$ is normalised and $D_{\frac{c}{|c|}}=|c|D_{c}$ is a scalar mutliple of $D_{c}$. Thus
\[[D_{c}]=[D_{\frac{c}{|c|}}]\in KK_{1}(C^{*}_{r}(\Gamma),C_{r}^{*}(\Gamma_{c})).\]

Theorem \ref{index-pairing} and the $K$-theory isomorphism show that the classes $|c|[D_{c}]+|c'|[D_{c'}]$ and $|c+c'|[D_{c+c'}]$ have the same index pairing and hence are equal, proving that the map $c\mapsto |c|[\Dsla_{c,s}]$ is a homomorphism. Injectivity follows in the same way. For surjectivity, let $(H,F)$ be an odd Fredholm module and $p_{+}$ the positive spectral projection of $F$. Then $c:\gamma\mapsto \textnormal{Ind}p_{+} u_{\gamma} p_{+}$ is a $1$-cocycle on $\Gamma$, and $|c|[\Dsla_{c,s}]$ is an unbounded Fredholm module whose index pairing coincides with $F$. Therefore $[(H,F)]=|c|[\Dsla_{c,s}]$ proving surjectivity.
\end{proof}

We now show that the explicit isomorphism of abelian groups  $H^{1}(\Gamma,\Z))\xrightarrow{\sim} K^{1}(C^{*}_{r}(\Gamma))$ is Hecke equivariant and construct an explicit section for the restriction map $K^i(C(\partial \h_3)\rtimes\Gamma) \rightarrow K^i(C^*_r(\Gamma))$ in the Gysin sequence.

\begin{proposition} Let $[T_{g}^{\Gamma}]\in KK_{0}(C^{*}_{r}(\Gamma),C^{*}_{r}(\Gamma))$ be the Hecke class from Definition \ref{GammaHecke},  $c:\Gamma\to \Z$ a cocycle and $\delta\in\Gamma$. We have the identities
\[[u_{\delta}]\otimes [T_{g}^{\Gamma}]=[u_{T_{g}(\delta)}]\in K_{1}(C^{*}_{r}(\Gamma)),\quad [T^{\Gamma}_{g}]\otimes |c|[\Dsla_{c,s}]=|T_{g}(c)|[\Dsla_{T_{g}(c),s}]\in K^{1}(C^{*}_{r}(\Gamma)).\]
In particular the isomorphisms $H_{1}(\Gamma,\Z)\to K_{1}(C^{*}_{r}(\Gamma))$ and  $H^{1}(\Gamma,\Z)\to K^{1}(C^{*}_{r}(\Gamma))$ are Hecke equivariant.
\end{proposition}
\begin{proof} 
By Proposition \ref{explicithecke} we have $t_{g}(u_{\delta})=\tau(\delta)\diag(u_{\chi_{k}(\delta)})$ and since  $\tau(\delta)\in M_{n}(\C)$ we have  $[\tau(\gamma)]=0\in K_{1}(C^{*}_{r}(\Gamma))$. So together with \eqref{algHecke} we find $$[T_{g}(u_{\delta})]=[t_{g}(u_{\delta})]=[\tau(\gamma)\diag(\chi_{i}(\gamma))]=[\diag(u_{\chi_{i}(\gamma)})]=\sum_{i=1}^{d} [u_{\chi_{i}(\gamma)}]=[u_{T_{g}([\delta])}].$$
Thus Hecke equivariance of the map $[\delta]\mapsto [u_{\delta}]$ is proved. For $K$-homology, by the UCT, it suffices to show that for all $\gamma\in \Gamma$ it holds that
\[ (T^{\Gamma}_{g}\otimes |c|[D_{c}], u_{\gamma})=(|T_{g}(c)|[D_{T_{g}(c)}], u_{\gamma}).\]
By Theorem \ref{index-pairing}, we can compute the right handside to equal $T_{g}(c)(\gamma)$. For the left handside, observe that the class $T^{\Gamma}_{g}\otimes [\Dsla_{c,s}]$ is represented by $(\bigoplus_{i=1}^{d} E\otimes_{C^{*}_{r}(\Gamma_{c})}L^{2}(\wedge^{*}\h_3) ,\textnormal{diag}(\Dsla_{c,s}))$. The representation of a unitary $u_{\gamma}$ is given by $\alpha_{g}(u_{\gamma})(h_{i})=(\chi_{i}(\gamma)h_{\gamma(i)})$. The positive spectral projection of $\textnormal{diag}(\Dsla_{c,s})$ is $\tilde{p_{+}}=\textnormal{diag}(p_{+})$ and thus the index pairing becomes
\[(T^{\Gamma}_{g}\otimes |c|[D_{c}], u_{\gamma})=|c|\textnormal{Ind} \tilde{p_{+}}\alpha_{g}(u_{\gamma})\tilde{p_{+}}=|c|\sum_{i=1}^{d}\textnormal{Ind}p_{+}u_{\chi_{i}(\gamma)}p_{+}=\sum_{i=1}^{d} c(\chi_{i}(\gamma))=T_{g}(c)(\gamma),\]
as required.\end{proof}
Unlike the previous results in this section, the following theorem is valid for discrete subsgroups $\Gamma\subset\textnormal{Isom}\h$ in any dimension.
\begin{theorem}  
The $(1-s)$-unbounded Fredholm modules in Theorem \ref{spectral-triple} extend to $(1-s)$-unbounded Fredholm modules for $C(\partial\h)\rtimes \Gamma$ such that $C(\partial\h)$ commutes with $\Dsla_{c,s}$. The extension is compatible with the restriction map $K^{1}(C(\partial\h)\rtimes \Gamma)\to K^{1}(C^{*}_{r}(\Gamma))$.
\end{theorem}
\begin{proof} Let $X=X_{\Gamma}\subset \h$ be an open connected fundamental domain for $\Gamma$. The disjoint union $\bigcup_{\gamma\in\Gamma} X\gamma$ is dense in $\h$ and  $$\tau:\h\to \Gamma,\quad \tau(x)=\tau_{X}(x)=g \Leftrightarrow xg^{-1}\in X,$$ is an almost everywhere defined equivariant measurable map. 

The tensor product $E\otimes L^{2}(\wedge^{*}\h)$ can be identified with the Hilbert space 
$\bigoplus_{n\in \Z}L^{2}(\wedge^{*}\h)$ by choosing $g\in g^{-1}(|c|)$ and using Equation \eqref{freemod}. By choosing a point $\xi\in\partial\h$, representations of  $C^{*}_{r}(\Gamma)$ and $C(\partial\h)$ are defined, for $\psi=(\psi_{n})_{n\in\Z}$, by
\[u_{\delta}(\psi)_{n}(h)=\psi_{n-c(\delta)}(hg^{-n}\delta g^{n-c(\delta)}),\quad (\pi_{X,\xi}(f)\psi)_{n}(h):=f(\xi\tau(h)g^{-n})\psi_{n}(h),\]
and form a covariant pair. Thus we obtain a representation of $C(\partial\h)\rtimes\Gamma$ on $E\otimes_{C^{*}_{r}(\Gamma_{c})}L^{2}(\wedge^{*}\h)$.

 The representation $\pi_{X,\xi}$ clearly commutes with the multiplication operators $\rho$ and $c$. Because $\pi_{X,\xi}(f)$ is constant on each $X\gamma$, it
 also commutes with the Dirac operator $\Dsla$. Therefore the $(1-s)$ spectral triples from Theorem\ \ref{spectral-triple} extend to $C(\partial \h)\rtimes\Gamma$. Since $\partial\h$ is connected, the choice of $\xi\in\partial\h$ does not affect the homotopy class of the spectral triple. If $Y$ is another
  open connected fundamental domain for $\Gamma$, there exists $\delta\in\textnormal{Isom}(\h)$ such that $X=Y\delta$ and thus $\tau_{X}(x)=\delta\tau_{Y}(x)$. This implies that $\pi_{Y,\xi}=\pi_{X,\xi\delta}$. Therefore the representations $\pi_{X,\xi}$ and $\pi_{Y,\xi}$ are homotopic as well.
\end{proof}

\section{$K$-cycles for Bianchi manifolds} \label{section: geometric}
Let $\Gamma$ be a torsion-free finite-index subgroup of a Bianchi group and $M$ be the associated hyperbolic $3$-manifold. 
We already know from Propositions \ref{prop: spectral-manifold} and  \ref{prop: spectral-homology} that there is an abstract isomorphism 
$$K^0(C_0(M)) \simeq H^1(\Gamma,\Z).$$
In this section, we shall construct an explicit Hecke equivariant isomorphism
$$K^0(C_0(M)) \simeq H_2(\overline{M},\partial \overline{M}, \Z),$$
where $\overline{M}$ is the Borel-Serre compactification of $M$ (see Section \ref{borelserre}).
Recall that 
$$H_{2}(\overline{M},\partial\overline{M},\Z) \cong H^1(\overline{M},\Z) \cong H^1(\Gamma,\Z)$$
and these isomorphisms are Hecke equivariant. Our approach uses geometric $K$-homology and employ work of Matthey \cite{matthey} on geometric $K$-homology of low-dimensional $CW$-complexes.

\subsection{Complex spin structures}
Spin structures on $M$ are in bijection with lifts of the holonomy representation $\Gamma \hookrightarrow \psl(\C)$ to $\textrm{SL}_2(\C)
\simeq \textrm{Spin}(3,1)$ (see, e.g. \cite[Section 2.7]{Pfaff}). It is known that such lifts exist and thus $M$ admits a spin structure. Let us fix a lift of the holonomy map of $M$ and denote the corresponding spin structure on $M$ by $\sigma$. It is well known that any compact oriented $3$-manifold admits a spin structure (see \cite[Section IV]{kirby}), in particular, the Borel-Serre compactification $\overline{M}$ of $M$ admits a spin structure. It turns out that, see \cite[Proposition 1, Section IV]{kirby}, we can choose a spin structure on $\overline{M}$ so that the induced spin structure on $M$ agrees with our fixed $\sigma$. We fix such a spin structure $\delta$ on $\overline{M}$.

A spin structure induces a complex spin (or spin$^{c}$) structure, in a canonical way. We denote the corresponding spin$^c$-structures on $M$ and $\overline{M}$ with the same symbols $\sigma$ and $\delta$ respectively. This will not cause confusion as we shall only consider spin$^c$ structures. In the rest of the paper, we will endow all codimension $0$ and codimension $1$ submanifolds of $\overline{M}$ with the canonical spin$^c$ structure arising from $\delta$. 

\subsection{Geometric $K$-homology}Let us describe $K^{0}(C_{0}(M))$ as a relative group in the Baum-Douglas model for $K$-homology of manifolds \cite{BaumDouglas}. For a CW-pair $(X,Y)$, a \emph{geometric cycle} is a triple $(N, E, \phi)$ consisting of a compact spin$^{c}$ manifold $N$ with boundary $\partial N$, a vector bundle $E\to N$ and a continuous map $\phi:N\to X$ such that $\phi(\partial N)\subset Y$. The parity $*=0,1$ corresponds to the dimension of $N$ being even or odd. Modulo a suitable equivalence relation, such cycles generate the \emph{geometric $K$-homology} $K^{\geo}_{*}(X,Y)$ of the pair $(X,Y)$. By taking $Y=\emptyset$, we obtain the geometric $K$-homology group $K^{\geo}_{*}(X):=K_{*}^{\geo}(X,\emptyset)$. For details see \cite{BHS, JakobHom}.

The paper \cite{matthey} describes explicit relationships between ordinary homology and geometric $K$-homology of low dimensional $CW$-complexes. Recall the Hurewicz homomorphism $h:\pi_{1}(M)\to H_{1}(M,\Z)$, which sends the class of a map $\phi:S^{1}\to M$ to $\phi_{*}([S^{1}])$, where $[S^{1}]\in H_{1}(S^{1},\Z)\simeq \Z$ is the fundamental class. By a slight abuse of notation, we denote $h([\phi])\in H_{1}(M,\Z)$ by $[\phi]$. By surjectivity of $h$, the group $H_{1}(M,\Z)$ is exhausted by the elements $[\phi]$. Similarly, any nontrivial class $z \in H_{2}(M,\Z)$ can be represented by an embedded surface, that is, there is a compact oriented surface $N$ and an embedding $\varphi:N\to M$ such that 
$\varphi_{*}([N]) =z$ where $[N] \in H^2(N,\Z) \simeq \Z$ is the fundamental class, \cite[Corollaire III.7.]{thom}. We shall denote $\varphi_*([N])$ by $[N,\varphi]$.

\begin{proposition}[\cite{matthey}]\label{betas} Let $X$ be a connected $CW$-complex such that $H_{k}(X,\Z)=0$ for all $k\geq 3$. There are explicit natural isomorphisms
\begin{align*}\beta_{\textnormal{odd}}: H_{1}(X,\Z)\to K_{1}^{\geo}(X),\quad & [\varphi]\mapsto [S^{1}, 1_{S^{1}}, \varphi ]\\
 \beta_{\textnormal{ev}}: H_{0}(X,\Z)\oplus H_{2}(X,\Z)\to K_{0}^{\geo}(X),\quad & ([\pt],[N,\varphi])\mapsto [\pt,1_{\pt},i]+[N, 1_{N},\varphi],
\end{align*}
where $[\phi]\in H_{1}(X,\Z)$ and $[N,\varphi]\in H_{2}(X,\Z)$ are as above, and $i:\pt\to X$ is any choice of inclusion.
\end{proposition}
\begin{proof} This result follows by Theorem 2.1 and Propositions 3.2, 3.3, 3.4 and 3.6 in \cite{matthey}.
\end{proof}

Given a geometric cycle $(N, E, \phi)$ for a CW-pair $(X,Y)$, let $S_{N}\to N$ be the spinor bundle and $D_{E}$ the associated symmetric Dirac operator on the bundle $E\otimes S$. The restriction of $\phi$ to $N\setminus\phi^{-1}(Y)$ gives a continuous map $\phi:N\setminus\phi^{-1}(Y)\to X\setminus Y$, which by the Tietze extension theorem gives a *-homomorphism $C_{0}(X\setminus Y)\to C_{0}(\intN)$. Here $\mathring{N}=N\setminus\partial N\subset N$ denotes the interior of $N$. We so obtain a representation $C_{0}(X\setminus Y)\to B(L^{2}(\mathring{N},S))$. The symmetric operator $D_{E}$ then defines a $K$-homology class by \cite[Theorem 3.2]{Hilsum}. The relation between geometric and analytic $K$-homology is given by the following result.
\begin{lemma}\label{BHSrel} Let $\overline{M}$ denote a topological compactification of $M$ and $\partial{\overline{M}}:=\overline{M}\setminus M$. If $(\overline{M},\partial{\overline{M}})$ is a $CW$-pair, then the map $$K_{0}^{\geo}(\overline{M},\partial{\overline{M}})\xrightarrow{\sim} K^{0}(C_{0}(M))\quad (N,E,\phi)\mapsto (C_{0}(M), L^{2}(\mathring{N}, E\otimes S), D_{E})$$
 is a natural isomorphism.
\end{lemma}
\begin{proof} This is the statement of \cite[Theorem 6.2]{BHS}.
\end{proof}

In view of the last lemma, we consider the Borel-Serre compactification  $\overline{M}$ of $M$, see Section \ref{borelserre}. The pair $(\overline{M},\partial \overline{M})$ form a $CW$-pair. In view of Proposition \ref{betas} and Lemma \ref{BHSrel}, we aim to construct a relative version of the map $\beta_{\textnormal{ev}}$. We begin with a relative version of Steenrod representability for $H_{2}$, which can be found in \cite[Proposition 1.7.16]{Martelli} (see also \cite[Lemma 2.9]{kapovich} and the remark after its proof). 

\begin{lemma}\label{Steenrodrepresentable} Any nontrivial class $z \in H_{2}(\overline{M},\partial\overline{M},\Z)$ can be represented by a properly embedded surface, that is, 
there is a compact oriented surface $N$ and an embedding $\varphi: N \rightarrow \overline{M}$ such that 
$\varphi( \partial N ) = \varphi(N) \cap \partial \overline{M}$ 
and $\varphi_*([N]) =z$ where $[N] \in H_2(N, \partial N,\Z) \simeq \Z$ is the fundamental class. Moreover $N$ can be chosen so that all its components have negative Euler characteristic.
\end{lemma}
As before, we denote $\varphi_*([N])$ by $[N, \varphi]$. For convenience we will write $(N,\partial N) \subset (\overline{M},\partial{\overline{M}})$ to mean that $N$ is a compact surface with boundary 
that is properly embedded into $\overline{M}$ as in Lemma \ref{Steenrodrepresentable}.

As $N\subset\overline{M}$ is an embedded hypersurface, the spin structure on $\overline{M}$ descends to $N$ and $(N,1_{N},\phi)$ is a geometric $K$-cycle for $(\overline{M},\partial\overline{M})$. We now show that these cycles exhaust the group $K^{\geo}_0(\overline{M},\partial{M})$.
\begin{proposition}\label{non-compact}There is a  natural isomorphism
\begin{equation}\label{relativematthey} \beta_{2}^{\textnormal{rel}}:H_{2}(\overline{M},\partial\overline{M}, \Z)\xrightarrow{\sim}K_{0}^{\geo}(\overline{M},\partial\overline{M}),\quad [N, \varphi]\mapsto [N, 1_{N}, \varphi],\end{equation}
where $\overline{M}$ is the Borel-Serre compactification of $M$ and $\varphi:(N, \partial N) \rightarrow (\overline{M}, \partial \overline{M})$ is an embedding.

\end{proposition}
\begin{proof} The notation $\beta_{2}^{\textnormal{rel}}$ is in accordance with \cite{matthey} and Proposition \ref{betas}, as the maps $\beta_{\textnormal{ev}}=\beta_{0}\oplus \beta_{2}$ and $\beta_{\textnormal{odd}}=\beta_{1}$.

To show that $\beta_{2}^{\textnormal{rel}}$ is well-defined, let $[N_1, \varphi_1], [N_2, \varphi_2]$ represent the same homology class. Consider the oriented bordism group $\Omega^{SO}_2(\overline{M}, \partial \overline{M})$ (see \cite[Section 4]{conner-floyd}). 
Noting that $\Omega^{SO}_0 \simeq \Z$ and $\Omega^{SO}_1 = \Omega^{SO}_2 = 0$, and that $H_{2}(\overline{M},\partial\overline{M}, \Z)$ is finitely 
generated and is torsion-free (as it is isomorphic to $H^1(\Gamma,\Z) = {\rm Hom}(\Gamma,\Z)$), we conclude by Theorem 15.2 of \cite{conner-floyd} 
that the representation map $\mu : \Omega_2(\overline{M}, \partial \overline{M}) \rightarrow H_{2}(\overline{M},\partial\overline{M}, \Z)$ is an isomorphism. 
This implies that $[N_1, \varphi_1]$ and  $[N_2, \varphi_2]$ are bordant in $(\overline{M}, \partial \overline{M})$. As we consider codimension $0$ and codimension $1$ submanifolds of $\overline{M}$ with the spin$^c$ structure inherited from that of $\overline{M}$, it 
follows immediately that the cycles $[N_1, 1_{N_1}, \varphi_1]$ are  $[N_2, 1_{N_2}, \varphi_2]$ 
are spin$^c$-bordant and thus represent the same geometric $K$-homology class. As the addition operation on both groups is given by disjoint union, it is now clear that we have a homomorphism $H_{2}(\overline{M},\partial\overline{M}, \Z)\rightarrow K_{0}^{\geo}(\overline{M},\partial\overline{M})$.

As group operations on both sides amount to taking disjoint unions of manifolds representing classes, it is clear that $\beta_2$ is a homomoprhism.

To show that the map $\beta_{2}^{\textnormal{rel}}$ is an isomorphism, recall the long exact sequence in homology associated to the pair $(\overline{M}, \partial \overline{M})$ which takes the form 
$$\xymatrix{0 & H_0(\overline{M}) \ar[l]& H_0(\partial \overline{M}) \ar[l] & H_1(\overline{M}, \partial \overline{M}) \ar[l]  & H_1(\overline{M}) \ar[l] & H_1(\partial \overline{M}) \ar[l]\\
& 0 \ar[r] & H_3(\overline{M}, \partial \overline{M}) \ar[r] & H_2(\partial \overline{M}) \ar[r]&H_2(\overline{M}) \ar[r]&H_2(\overline{M}, \partial \overline{M}) \ar[u] }$$ 
due to the facts that $H^0(\overline{M},\partial \overline{M}) \simeq 0 \simeq H_3(\partial \overline{M})$. Next, consider the six-term exact sequence of geometric $K$-homology groups (see for instance \cite{BHS, JakobHom}):
$$\xymatrix{K^{\geo}_0(\partial \overline{M}) \ar[r] & K^{\geo}_0(\overline{M}) \ar[r] & K^{\geo}_0(\overline{M}, \partial \overline{M}) \ar[d]^{\partial} \\ 
K^{\geo}_1(\overline{M}, \partial \overline{M}) \ar[u]_{\partial} & K^{\geo}_1(\overline{M}) \ar[l] & K^{\geo}_1(\partial \overline{M})\ar[l]}$$ 

Writing $H_{\textnormal{ev}}(X)=H_{0}(X)\oplus H_{2}(X)$, and $\iota: (\overline{M},\emptyset)\to (\overline{M},\partial\overline{M})$ for the inclusion of CW-pairs, 
Proposition \ref{betas} yields a diagram with exact rows 
$$\xymatrix{H_{\textnormal{ev}}(\partial \overline{M}) \ar[r] \ar[d]_{\beta_{\textnormal{ev}}}& H_{\textnormal{ev}}(\overline{M}) \ar[r]^{\iota_{*}}\ar[d]_{\beta_{\textnormal{ev}}} & H_{2}(\overline{M}, \partial \overline{M})\ar[r]^{\partial}\ar[d]_{\beta_{2}^{\textnormal{rel}}}& H_{1}(\partial\overline{M})\ar[r]\ar[d]_{\beta_{\textnormal{odd}}}& H_{1}(\overline{M})\ar[d]_{\beta_{\textnormal{odd}}}\\
K^{\geo}_0(\partial \overline{M}) \ar[r] & K^{\geo}_0(\overline{M}) \ar[r]^{\iota_{*}} & K^{\geo}_0(\overline{M}, \partial \overline{M})\ar[r]^{\partial} & K^{\geo}_1(\partial\overline{M}) \ar[r]& K^{\geo}_1(\overline{M}), }
$$
whose outer squares commute. If we show that the inner squares commute as well, then the five Lemma and the fact that $\beta_{\textnormal{ev}}, \beta_{\textnormal{odd}}$ are isomorphisms, implies that $\beta_{2}^{\textnormal{rel}}$ is an isomorphism as well. 

To show that $\beta_{2}^{\textnormal{rel}}\circ\iota_{*}=\iota_{*}\circ \beta_{\textnormal{ev}}$, observe that the $H_{0}(\overline{M})$ summand of $H_{\textnormal{ev}}(\overline{M})$ is annihilated by $\iota_{*}$, as is the class of a point in $K_{0}^{\geo}(\overline{M})$. For a surface class $[(N,\varphi)]\in H_{2}(\overline{M})$ we find that
\[\beta_{2}^{\textnormal{rel}}\circ\iota_{*}[(N,\varphi)]=\beta_{2}^{\textnormal{rel}}[(N,\varphi)]=[(N,1_{N},\varphi)]=\iota_{*}[(N,1_{N},\varphi)]=\iota_{*}\circ \beta_{\textnormal{ev}}[(N,\varphi)],\]
as desired. We now prove that $\beta_{\textnormal{odd}} \circ \partial = \partial\circ \beta^{\textnormal{rel}}_2$. 
 By Lemma \ref{Steenrodrepresentable} all classes in $H_{2}(\overline{M},\partial\overline{M})$ are of the form $[(N,\varphi)]$. The boundary $\partial N$ is a compact 1-dimensional manifold, and therefore decomposes as a disjoint union $\partial N=\bigsqcup_{i=1}^{k} S^{1}$ of circles $S^{1}$. Denote by $\varphi_{i}$ the restriction of $\varphi$ to the $i$-th circle in this decomposition.  We compute the composition
\begin{align*}\beta_{\textnormal{odd}}\circ\partial[(N,\varphi)]&=\beta_{\textnormal{odd}}[(\partial N, \varphi|_{\partial N})]=\sum_{i=1}^{k}\beta_{\textnormal{odd}}[(S^{1},\varphi_{i})] \\ &=\sum_{i=1}^{k}[(S^{1},1_{S^{1}},\varphi_{i})]=[(\partial N, 1_{N}, \varphi)]=\partial[(N,1_{N},\varphi)]=\partial\circ\beta_{2}^{\textnormal{rel}}[(N,\varphi)].\end{align*}
This completes the proof that $\beta_{2}^{\textnormal{rel}}$ is an isomorphism.\end{proof}
Note that Lemma \ref{Steenrodrepresentable} implies that $\mathring{N}=N\cap M\subset M$ is a closed embedded hypersurface. We equip $\intN$ with the metric inherited from the hyperbolic metric on $M$ as well as with the inherited spin$^{c}$ structure. The Riemannian distances $d_{\intN},d_{M}$ satisfy $d_{M}(x,y)\leq d_{\intN}(x,y)$ for $x,y\in \intN$. Since $\intN$ carries the relative topology as a subset of $M$ and $M$ is complete, it follows that $\intN$ is complete. The spinor bundle $\mathcal{S}_{\intN}\to \intN$ is the restriction of the spinor bundle $\mathcal{S}_M\to M$ to $\intN$ (see \cite{Baer, Diracembedded}). Thus $\intN$ is a complete Riemannian spin$^{c}$ manifold and we denote by $\Dsla_{\intN}$ its Dirac operator, which is essentially self-adjoint on $C^{1}_{c}(\intN,\mathcal{S}_{\intN})$, the compactly supported $C^{1}$-sections.

\begin{theorem}\label{geometriciso} Let $\Gamma\subset \psl(\C)$ be a noncocompact torsion-free discrete subgroup. There is a natural isomorphism \begin{equation}\label{relativeiso} H_{2}(\overline{M},\partial\overline{M})\xrightarrow{\sim}K^{0}(C_{0}(M)),\quad [(N,\phi)]\mapsto (C_{0}(M), _{\phi}L^{2}(\mathring{N}, \mathcal{S}_{\mathring{N}}),\Dsla_{\mathring{N}}),\end{equation}
where $\mathring{N}$ is viewed as a spin$^c$ surface with associated Dirac operator $\Dsla_{\mathring{N}}$.
\end{theorem}
\begin{proof} By Lemma \ref{BHSrel}, we obtain a map $(N,1_{N},\phi)\to (C_{0}(M), L^{2}(\intN,S), D_{\intN})$ where $D_{\intN}$ is the symmetric operator obtained from the manifold with boundary $N$. Since we have chosen the spin structure on $\overline{M}$ to be compatible with that on $M$, the spin structure that $N$ inherits from $\overline{M}$ is compatible with the spin structure that $\intN$ inherits from $M$. By \cite[Proposition 11.27]{HR} it follows that
\[[(C_{0}(M), L^{2}(\intN,S), D_{\intN})]=[(C_{0}(M),L^{2}(\intN,\mathcal{S}|_{\intN}), \Dsla_{\intN})]\in K^{0}(C_{0}(M)).\]
Combining 
Lemmas \ref{BHSrel}, \ref{Steenrodrepresentable} and Proposition \ref{non-compact}, it thus follows that the map \eqref{relativeiso} is an isomorphism.
\end{proof}

\subsection{Hecke equivariance} 
Given a class $[N,\varphi] \in H_2(\overline{M}, \partial \overline{M},\Z)$ and a Hecke operator $T_g$, it can be seen that the class $T_g([N,\varphi])$ is represented by 
\begin{equation}\label{GeoHecke} T_g([N,\varphi])=[(\pi^{-1}_{g}(N),\tau_{g})]=[(N_{g},\tau_{g})] \end{equation} 
where $\pi_g, \tau_g$ are as in Section \ref{borelserre} and $N_g$ is a compact surface with boundary $\partial N_{g}\subset \partial\overline{M}_{g}$ given by the fiber product 
\begin{equation}\label{relativefibre}N_{g}:=(\overline{M}_{g})_{\pi_{g}}\times_{\varphi}N\simeq \pi^{-1}_{g}(N)\subset \overline{M}_{g}.\end{equation} 
The reader should compare this with the discussion in \cite[Section 3]{DunRam}.  

\begin{proposition}\label{geoequiv} The isomorphism $H_{2}(\overline{M},\partial{\overline{M}},\Z) \to K^{0}(C_{0}(M))$ (cf.  \eqref{relativeiso}) is Hecke equivariant.
\end{proposition}
\begin{proof} Take a class $[N,\varphi] \in H_{2}(\overline{M},\partial{\overline{M}},\Z)$. Comparing \ref{GeoHecke} and the isomorphism \eqref{relativeiso}, we see that we need to show that
\[[(C_{0}(M), _{\tau_{g}}L^{2}(\intN_{g},\mathcal{S}_{\intN_{g}}), \Dsla_{N_{g}})]=[T_{g}^{M}]\otimes [(C_{0}(M), _{\varphi}L^{2}(\intN,\mathcal{S}_{\intN}),\slashed{D}_{\intN})],\]
in the group $KK_{0}(C_{0}(M),C_{0}(M))$. Viewing $\mathring{N}_{g}$ as the inverse image $\pi_{g}^{-1}(\mathring{N})\subset M_{g}$ using \eqref{relativefibre} it is straightforward to show that
\[ w:T_{g}^{M}\otimes_{C_{0}(M)}L^{2}(\mathring{N}, \mathcal{S}_{\mathring{N}})\to _{\tau_{g}}L^{2}(\mathring{N}_{g},\mathcal{S}_{\mathring{N}_{g}}),\quad w(\chi\otimes\psi)(n)=\chi(n)\pi_{g}^{*}\psi(n),\]
is a unitary isomorphism intertwining the left $C_{0}(M)$-representations. To prove that $\slashed{D}_{\mathring{N}_{g}}$ represents the Kasparov product we need to check conditions i-iii in \cite[Theorem 13]{Kuc}, of which ii and iii are trivial since the module $T_{g}^{M}$ carries the $0$ operator. Now suppose that $\chi\in C^{1}_{c}(\mathring{M}_{g})$ is such that $\supp \chi\subset U$, with $U$ an open set such  that $\pi_{g}|_{U}$ is injective. Then we can choose $\zeta\in C_{c}^{1}(\mathring{M})$ with $\chi=(\pi_{g}^{*}\zeta)|_{U}$. Then for $\psi\in L^{2}(\mathring{N},\mathcal{S}_{\intN})$ we have  $\Dsla_{\mathring{N}_{g}}\chi\pi_{g}^{*}\psi=\Dsla_{\mathring{N}_{g}}\pi_{g}^{*}(\zeta\psi)|_{U}=\pi_{g}^{*}(\Dsla_{\mathring{N}}\zeta\psi)|_{U}$. Thus we find 
\begin{align*}\Dsla_{\mathring{N}_{g}} \chi \pi_{g}^{*}\psi - \chi\pi_{g}^{*}\Dsla_{\mathring{N}}\psi&=\pi_{g}^{*}(\Dsla_{\mathring{N}} \zeta \psi-\zeta\Dsla_{\mathring{N}} \psi)|_{U}=\pi_{g}^{*}(c(\textnormal{d}_{\mathring{N}}(\zeta|_{\mathring{N}}))\psi)|_{U} 
\end{align*}
where $c$ denotes Clifford multiplication of forms. Since $$\|\pi_{g}^{*}(c(\textnormal{d}_{\mathring{N}}(\zeta|_{\mathring{N}}))\psi)|_{U}\|_{L^{2}(\mathring{N}_{g},\mathcal{S}_{\mathring{N}_{g}})}=\|c(\textnormal{d}_{\mathring{N}}(\zeta|_{\mathring{N}}))\psi)\|_{L^{2}(\mathring{N},\mathcal{S})}\leq \|c(\textnormal{d}_{\mathring{N}}(\zeta|_{\mathring{N}}))\|\|\psi\|_{L^{2}(\mathring{N},\mathcal{S})}$$
it follows that $\psi\mapsto \Dsla_{\mathring{N}_{g}} \chi \pi_{g}^{*}\psi - \chi\pi_{g}^{*}\Dsla_{\mathring{N}}\psi$ extends to a bounded operator. The submodule of $T^{M}_{g}$ generated by elements $\chi\in C_{c}^{1}(\mathring{M}_{g})$ of small support is dense in $T^{M}_{g}$. Hence condition i of \cite[Theorem 13]{Kuc} is satisfied and we are done. \end{proof}

\section{The case of $\psl(\Z)$} \label{section: others}

Let $\Gamma$ be a torsion-free finite index subgroup of $\psl(\Z)$. Then it acts properly discontinuously on the hyperbolic plane $\h_2$ and the quotient $M= \h_2 / \Gamma$ is a finite volume hyperbolic surface with cusps. The boundary of $\h_2$ can be identified with $\P^1(\R)$.

The analogue of Proposition \ref{prop: spectral-big} in this case is the following (note that the cohomological dimension of $\Gamma$ is one). For $i=0,1$, we have
$$K_i(C(\P^1(\R)) \rtimes \Gamma) \simeq H_0(\Gamma,\Z) \oplus H_1(\Gamma,\Z)$$
and 
$$K^i(C(\P^1(\R)) \rtimes \Gamma) \simeq H^0(\Gamma,\Z) \oplus H^1(\Gamma,\Z).$$
This is actually well known, it is a special case of the work of Anantharaman-Delaroche (see \cite{delaroche}) who treated cofinite discrete subgroups of $\psl(\R)$. Note that $H^0(\Gamma,\Z) \simeq \Z$ and $H^1(\Gamma,\Z) \simeq \Z^{2g+c-1}$ where $g$ is the genus of $\Gamma$ and $c \ge 1$ is the number of cusps of $M$. 

In \cite{manin-marcolli}, Manin and Marcolli describe the above isomorphisms in terms of Manin symbols using Pimsner's 6-term exact sequence \cite{pimsner} of which Kasparov's spectral sequence can be viewed as a generalization. 

Much of Section \ref{section: gysin} carries through and we obtain the Hecke equivariant exact hexagon 
\begin{equation}\label{Gysin-modular} 
\xymatrix{ K^{1}(C_{0}(M)) \ar[r] & K^{0}(C(\P^1(\R))\rtimes\Gamma) \ar[r] & K^{0}(C^{*}_{r}(\Gamma) \ar[d] \\ 
K^{1}(C^{*}_{r}(\Gamma)) \ar[u] & K^{1}(C(\P^1(\R))\rtimes \Gamma) \ar[l] & K^{0}(C_{0}(M))  \ar[l] }
\end{equation}
As $M$ is non-compact, this hexagon \eqref{Gysin-modular} breaks apart into two short exact sequences (see \cite{EM})
\begin{equation}\label{Gysineven-modular} 0\rightarrow K^{1}(C_{0}(M))\rightarrow K^{0}(C(\P^1(\R))\rtimes\Gamma)\rightarrow K^{0}(C^{*}_{r}(\Gamma))\rightarrow 0,\end{equation}
\begin{equation}\label{Gysinodd-modular}0\rightarrow K^{0}(C_{0}(M))\rightarrow K^{1}(C(\P^1(\R))\rtimes \Gamma)\rightarrow K^{1}(C^{*}_{r}(\Gamma))\rightarrow 0\end{equation}
as in the Bianchi case.

It is well-known that $\Gamma$ is a free group on $2g+c-1$ generators. It follows, for example, from work of Cuntz \cite{cuntz} and of Lance \cite{lance}, that 
$K^0(C^{*}_{r}(\Gamma)) \simeq \Z$ and $K^1(C^{*}_{r}(\Gamma)) \simeq \Z^{2g+c-1}$. As a result, the sequences 
\eqref{Gysineven-modular} and \eqref{Gysinodd-modular} split and also we get $K^{1}(C_{0}(M)) \simeq \Z^{2g+c-1}$ 
and $K^{0}(C_{0}(M)) \simeq \Z$.

The analogue of Proposition \ref{prop: spectral-small} reads as follows. 
For $i=0,1$, we have
$$K^i(C^{*}_{r}(\Gamma)) \simeq H^i(\Gamma,\Z),$$
$$K_i(C^{*}_{r}(\Gamma)) \simeq H_i(\Gamma,\Z).$$
The map we constructed in Section \ref{section: algebraic} is defined here as well and we get
\begin{theorem} The maps
\begin{center}
\begin{tabular}{rclcccrcl} 
$H_{1}(\Gamma,\Z)$ & $\longrightarrow$ & $K_{1}(C^{*}_{r}(\Gamma))$ && \text{and}  && $H^{1}(\Gamma,\Z)$ & $\longrightarrow$ & $K^{1}(C^{*}_{r}(\Gamma))$ \\
$[\delta]$ & $\mapsto$ & $[u_{\delta}]$ & &&& $[c]$ &  $\mapsto$ & $|c|\cdot [D_{\frac{c}{|c|},\h}]$ 
\end{tabular}
\end{center}
are Hecke equivariant isomorphisms compatible with the pairings of the respective groups.
\end{theorem}

Our results in Section \ref{section: geometric} adapt straightforwardly to the case of $\psl(\Z)$. The  $1$-dimensional analogue of Lemma \ref{Steenrodrepresentable} holds and we get the following 
(using the notation of Section \ref{section: geometric}):
\begin{theorem} There is a Hecke-equivariant isomorphism 
$$H_1(\overline{M},\partial{\overline{M}},\Z) \to K^1(C_{0}(M))$$
sending the homology class $[N, \varphi]$ to the class $[(C_{0}(M), L^{2}(\mathring{N}, E\otimes S), D_{E})]$ 
where $\overline{M}$ denotes the Borel-Serre compatification of $M$.
\end{theorem}

Note that $H_1(\overline{M},\partial{\overline{M}},\Z) \simeq H^1(\Gamma,\Z)$ as Hecke modules.

\section{The extension class as a hypersingular integral operator}  \label{section: hyper}
In section \ref{subsec: extensionclass} we used the harmonic measures $\nu_{x}$ on the boundary $\dH$ to represent the boundary extension \eqref{bdry} as a Kasparov module. In order to compute the $K$-homology boundary map $\partial:K^{0}(C_{0}(M))\to K^{1}(C(\dH)\rtimes\Gamma)$, we now construct an unbounded Kasparov module \cite{BJ} representing the extension class.
\subsection{Harmonic calculus on $T_{1}\h$} As before, $\h$ denotes the Poincar\'{e} ball model of hyperbolic $n+1$ space. To construct an unbounded representative for the boundary extension we discuss hypersingular integral operators defined using the harmonic measures $\nu_{x}$ and a family of metrics $d_{x}$ on $\dH$ which we now describe For all $x,y\in\overline{\h}$ and $g\in G$ we have (cf. \cite[Equation 1.3.2]{Nicholls}):
\begin{equation}\label{fundamental} \|xg- yg\|=|g'(x)|^{\frac{1}{2}}|g'(y)|^{\frac{1}{2}}\|x-y\|,\end{equation}
where $\|\cdot\|$ denotes the Euclidean norm on $\R^{n+1}$. Using the Poisson kernel \eqref{Poissonkernel} the function
\begin{equation} \label{metrics}d_{x}(\xi,\eta):=P(x,\xi)^{1/2}P(x,\eta)^{1/2}\|\xi-\eta\|,\end{equation}
satisfies $d_{xg}(\xi g,\eta g)=d_{x}(\xi,\eta)$ by \eqref{Pg} and so $d_{x}$ is a metric on $\dH$, as this holds for $d_{0}$ and $g$ acts transitively on $\h$ (compare \cite[Lemma 3.4.2]{Nicholls} and \cite[Section 3.3]{quintsurvey}).

 For a pair $(x,\xi)\in T_{1}\h$ we denote by $r_{(x,\xi)}: \R\to \h$ the geodesic ray with $r_{(x,\xi)}(0)=x$ and $\lim_{t\to +\infty}(r_{(x,\xi)}(t)=\xi$. Recall that for $r>0$ and $x,y\in\h$, the \emph{Sullivan shadow} is the set
\[\mathcal{O}_{r}(x,y):=\{\xi\in\partial\h: \inf\{d_{\h}(r_{x,\xi}(t),y):t\in[0,\infty)\}\leq r\}\subset\partial\h.\]
A proof of the following result can be found in \cite{Nicholls, quintsurvey}.
\begin{proposition}[Sullivan's shadow lemma \cite{Sullivan}] \label{Sullshad} For all $\xi,\eta\in \mathcal{O}_{r}(x,y)$ it holds that
$$e^{-r}e^{d(x,y)}\leq d_{x}(\xi,\eta)\leq e^{d(x,y)}.$$ Moreover there exists $r_{0}$ such that for all $r\geq r_{0}$ there exists $C_{r}>0$ for which 
$$C_{r}^{-1}e^{-nd(x,y)}\leq \nu_{x}(\mathcal{O}_{r}(x,y))\leq C_{r}e^{-nd(x,y)}, $$
for all $x,y\in\h$.
\end{proposition}
We start with the following observation, concerning the Riesz potentials commonly studied in metric measure theory (see for instance \cite{Zahle}).
\begin{lemma}\label{unicorn}For all $0<s<n$, the integral $$I_{s} = I_{s}(x,\xi):=\int \frac{1}{d_{x}(\xi,\eta)^{n-s}}d\nu_{x}\eta,$$ is finite and independent of $(x,\xi)$.
\end{lemma}
\begin{proof} To see that the integral is finite for fixed $(x,\xi)$, we only consider the case $0< s <n$, as the case $s\geq n$ is immediate. Choose points $x_{k}$ on the geodesic from $x$ to $\xi$ such that $d(x,x_{k})=k$. We write $\partial\h$ as a disjoint union
\[\partial \h=\bigcup_{k=0}^{\infty} \mathcal{O}_{r}(x,x_{k})\setminus \mathcal{O}_{r}(x,x_{k+1}),\]
and expand the integral and then estimate using Proposition \ref{Sullshad}:
\begin{align*}\int \frac{1}{(d_{x}(\xi,\eta))^{n-s}}d\nu_{x}\eta&=\sum_{k=0}^{\infty}\int_{ \mathcal{O}_{r}(x,x_{k})\setminus \mathcal{O}_{r}(x,x_{k+1})}\frac{1}{d_{x}(\xi,\eta)^{n-s}}d\nu_{x}\eta\leq C_{r} \sum_{k=0}^{\infty}\frac{\nu_{x}(\mathcal{O}_{r}(x,x_{k}))}{(e^{-d(x,x_{k+1})})^{n-s}}\\ &\leq C_{r} \sum_{k=0}^{\infty} e^{(n-s)(k+1)}e^{-nk}= C_{r}e^{n-s} \sum_{k=0}^{\infty} e^{-s k}<\infty.
\end{align*}
Thus the integral $I_{s}\Psi(x,\xi)$ is finite for fixed $(x,\xi)$. 
For $g\in G_{x}$ we have 
\begin{align*}\int \frac{1}{(d_{x}(\xi,\eta))^{n-s}}d\nu_{x}(\eta)&=\int\frac{1}{(d_{x}(\xi g,\eta g))^{n-s}}d\nu_{x}(\eta g)=\int\frac{1}{(d_{x}(\xi g,\eta ))^{n-s}}d\nu_{x}(\eta),
\end{align*}
and since $G_{x}$ acts transitively on $\partial\h$, we see that $I_{s}(x,\xi)$ is constant in $\xi$. Using this fact, we can write for $g\in G$:
\begin{align*}\int \frac{1}{(d_{x}(\xi,\eta))^{n-s}}d\nu_{x}(\eta)&=\int \frac{1}{ d_{xg}(\xi g,\eta g))^{n-s}}d\nu_{xg}(\eta g)=\int \frac{1}{( d_{xg}(\xi,\eta))^{n-s}}d\nu_{xg}(\eta),\end{align*}
and since $G$ acts transitively on $\h$ we see that $I_{s}(x,\xi)$ is independent of $x$ as well. \end{proof}
Now we consider the spherical hypersingular operator $\Delta_{0}:\Lip(\dH)\to L^{2}(\partial\h,\nu_{0})$ and the projection $p_{0}\in \mathbb{B}(L^{2}(\partial\h,\nu_{0}))$ given by 
\begin{equation}\label{initialD}\Delta_0 f(\xi)=\int\frac{f(\xi)-f(\eta)}{d_{0}(\xi,\eta)^{n}}d\nu_{0}\eta,\quad p_{0}f(\xi):=\int f(\eta)d\nu_{0}\eta.\end{equation}
Operators of exponent $n+\ep$ in the denominator have been extensively studied by Samko \cite{Samkoproc, Samkobook}.
\begin{lemma}\label{domainrangeresolvent} The operator $\Delta_{0}$ maps  $\Lip(\partial\h)\subset L^{2}(\partial\h,\nu_{0})$ into bounded functions on $\partial \h$, is essentially self-adjoint on $C^{1}(\dH)\subset \Lip(\dH)$ and has compact resolvent. Moreover $\ker \Delta_{0}=\textnormal{Im } p_{0}$ and $\Delta_{0}+p_{0}$ is strictly positive.
\end{lemma}
\begin{proof} Observe that for $f\in\Lip(\partial\h)$ , by H\"{o}lder's inequality and Lemma \ref{unicorn} we have
\[\left|\int\frac{f(\xi)-f(\eta)}{d_{0}(\xi,\eta)^{n}}d\mu_{0}\eta \right|\leq\int \frac{|f(\xi)-f(\eta)|}{d_{0}(\xi,\eta)^{n}}d\nu_{0}\eta\leq  \|f\|_{\Lip}\int\frac{1}{d_{0}(\xi,\eta)^{n-1}}d\nu_{0}\eta=I_{1}\|f\|_{\Lip},\]
so $\Delta_{0}f$ is a bounded function and thus $\Delta_{0}f\in L^{2}(\partial\h,\nu_{0})$. In particular, any orthonormal family of spherical harmonics $Y_{m,k}\in C^{1}(\dH)$ is in the domain of $\Delta_{0}$. Using the the method of Samko \cite[Lemma 6.25]{Samkobook} we see that the multiplier  of $\Delta_0$ on spherical harmonics is given by
\[Y_{m,k}\mapsto \lambda_{m}Y_{m,k},\quad \lambda_{m}=\int_{-1}^{1}(1+t)^{\frac{n-3}{2}}(1-t)^{-1}(1-P_{m}(t))dt.\]
Here 
$P_{m}(t)=\sum_{k=0}^{m}\left(\frac{-1}{2}\right)^{k}\binom{m}{k}\binom{m+k}{k}(1-t)^{k}$ is  the $m$-th Legendre polynomial. 
In particular we have $\lambda_{0}=0$ since $P_{0}(t)=1$. For $m>0$ we find
\[(1-t)^{-1}(1-P_{m}(t))=\sum_{k=1}^{m}\left(\frac{-1}{2}\right)^{k-1}\binom{m}{k}\binom{m+k}{k}(1-t)^{k-1},\]
so we are concerned with the integrals
\[\int_{-1}^{1}(1-t)^{k-1}(1+t)^{\alpha}dt=\frac{(-2)^{k-1}2^{\alpha}(k-1)!}{(\alpha+1)\cdots (\alpha+k-1)},\quad \alpha=\frac{n-3}{2}.\]
The proof of this equality follows by induction on $k$. Thus we find, for $m\geq \alpha$
\begin{align*}\lambda_{m} 
&=2^{ \alpha} \sum_{k=1}^{m}  \left(\binom{m}{k}\binom{m+k}{k}\frac{(k-1)!}{(\alpha+1)\cdots (\alpha+k-1)}\right)\\ & =2^{\alpha} \sum_{k=1}^{m}  \left(\frac{(m+k)!}{k!(m-k)!}\frac{k}{(\alpha+1)\cdots (\alpha+k-1)}\right)
\geq 2^{ \alpha}\sum_{k=1}^{m} \frac{m!}{k!(m-k)!}=2^{ \alpha}(2^{m}-1) 
 \end{align*}
This proves that $\lambda_{m}>0$ for all $m>0$ and $\lambda_{m}\to\infty$ for $m\to \infty$. Hence $\Delta_{0}$ is essentially self-adjoint on $\Lip(\partial\h)$, and has compact resolvent in $L^{2}(\partial\h,\nu_{0})$. Moreover $\Delta_{0}$ is positive with kernel the constant functions, on which $p_{0}$ projects so $\Delta_{0}+p_{0}$ is strictly positive.
\end{proof}
We wish to extend the operator $\Delta_{0}$ to the module $L^{2}(T_{1}\h,\nu_x)_{C_{0}(\h)}$ in a way compatible with the action of $G=\textnormal{Isom} \h$ on this module. 
\begin{lemma}\label{trivialization} The map $\Psi\mapsto P^{n/2}\Psi,$
where $P$ is the Poisson kernel \eqref{Poissonkernel}, extends to a unitary isomorphism  \begin{equation}\label{trivialize!}v:L^{2}(T_{1}\h,\nu_x)_{C_{0}(\h)}\to C_{0}(\h, L^{2}(\dH,\nu_0))\simeq L^{2}(\partial\h,\nu_{0})\otimes C_{0}(\h),\end{equation}
of right $C_{0}(\h)$-modules.
\end{lemma}
\begin{proof} Since $d\nu_{x}(\xi)=P(x,\xi)^{n}d\nu_{0}(\xi)$ it is straightforward that $v$ is innerproduct preserving. Since it maps $C_{c}(T_{1}\h)$ into itself, it is surjective.
\end{proof}
In $L^{2}(T_{1}\h,\nu_x)_{C_{0}(\h)}$ we consider the operator 
\begin{equation}\label{Diff}\Delta\Psi(x,\xi):=\int\frac{\Psi(x,\xi)-\Psi(x,\eta)}{d_{x}(\xi,\eta)^{n}}d\nu_{x}\eta,\end{equation}
which is initially defined on the dense subspace $C_{c}^{1}(T_1\h)\subset L^{2}(T_{1}\h,\nu_x)_{C_{0}(\h)}$. 
\begin{lemma}\label{Hlem} The function $H:T_1\h\to \R$ given by
\begin{align}
\label{Hdef}H(x,\xi)&=\int\frac{\frac{P(x,\xi)^{n/2}}{P(x,\eta)^{n/2}}-1}{d_{0}(\xi,\eta)^{n}}d\nu_{0}\eta=P(x,\xi)^{n/2}\int \frac{P(x,\eta)^{-n/2}-P(x,\xi)^{-n/2}}{d_{0}(\xi,\eta)^{n}}d\nu_{0}\eta,
\end{align}
defines an element of $C^1(\h, L^\infty(\dH,\nu_0))$ via $x\mapsto H_{x}$, $H_{x} (\xi):=H(x,\xi)$.
\end{lemma}
\begin{proof}
For fixed $x$, $P_{x}:\xi\mapsto P(x,\xi)$ is nonvanishing and Lipschitz in $\xi$. Therefore the function $$H_{x}: \dH\to\dH,\quad H_{x}:\xi\mapsto H(x,\xi)=P_{x}(\xi)^{n/2} \Delta_{0}(P^{-n/2}_{x})(\xi)$$ is bounded by Lemma \ref{domainrangeresolvent} and thus defines an element of $L^{2}(\dH,\nu_0)$.  For $x,y\in \h$, the function $P_{x}^{-n/2}-P_{y}^{-n/2}:\dH\to \R$ is defined on an open neighbourhood of $\dH\subset \R^{n+1}$. Denote by $\grad$ the Euclidean gradient by and  $h(x,y):=\sup_{\xi}\|\grad(P_{x}^{-n/2}(\xi)-P_{y}^{-n/2}(\xi)) \|$, which is a continuous function of $(x,y)$ that vanishes on the diagonal. The mean value theorem gives
\[\|P_{x}(\xi)^{-n/2}-P_{y}(\xi)^{-n/2}-P_{x}(\eta)^{-n/2}+P_{y}(\eta)^{-n/2}\| \leq h(x,y)\|\xi-\eta\|.\]
 Thus, with $I_{1}$ as in Lemma \ref{unicorn} we can estimate
\begin{align*}
\| \Delta_{0}(P^{\frac{-n}{2}}_{x}-P^{\frac{-n}{2}}_{y})\|_{\infty}&\leq\left\|\int\frac{P_{x}(\eta)^{\frac{-n}{2}}-P_{x}(\xi)^{\frac{-n}{2}}-P_{y}(\eta)^{\frac{-n}{2}}+P_{y}(\xi)^{\frac{-n}{2}}}{d_0(\xi,\eta)^{n}}d\nu_0\eta \right\|_{\infty}
\\
&\leq h(x,y)\int_{\dH}\frac{1}{d_{0}(\xi,\eta)^{n-1}}d\nu_0\eta=I_{1}h(x,y).
\end{align*}
Then the estimate
$$\|P_{x}^{n/2}\Delta_{0}P_{x}^{-n/2}-P_{y}^{n/2}\Delta_{0}P_{y}^{-n/2}\|_{\infty}\leq \|P_{x}^{n/2}-P_{y}^{n/2}\|_{\infty}\|\Delta_{0}P_{x}^{-n/2}\|_{\infty}+\|P_{y}^{n/2}\|_{\infty}I_{1}h(x,y),$$
shows that $\lim_{x\to y} \|P_{x}^{n/2}\Delta_{0}P_{x}^{-n/2}-P_{y}^{n/2}\Delta_{0}P_{y}^{-n/2}\|_{\infty}=0$, so $x\mapsto H_{x}$ is a continuous map $\h\to L^{\infty}(\dH,\nu_0)$. The partial derivatives $\partial_{x_{i}}H(x,\xi)$ are continuously differentiable in $\xi$, and so a similar argument shows that the maps $x\mapsto \partial_{x_{i}}H(x,\xi)$ are continuous maps $\h\to L^{\infty}(\h,\nu_{0})$ as well. \end{proof}
For a Banach space $E$, we denote by $C_{c}^{1}(\h, E)$ the space of compactly supported continuously differentiable  functions on $\h$ with values in $E$. We will always consider $\Dom \Delta_{0}$ as a Hilbert space in the graph norm. Using the injection $\Dom \Delta_0\to L^2(\dH,\nu_0)$ we view $C_{c}^{1}(\h, \Dom \Delta_0)$ as a subspace of $C_0(\h, L^2(\dH,\nu_0))$, which both are $C_0(\h)$-bimodules
\begin{proposition}\label{compact!} Let $v$ be as in Lemma \ref{trivialize!} and $\Delta$ as in \eqref{Diff}. The operator $\Delta$ is essentially self-adjoint and regular on $v^*C_{c}^{1}(\h,C^1(\dH))\subset L^{2}(T_{1}\h,\nu_{x})_{C_{0}(\h)}$ and satisfies
$$\Delta:v^*C_{c}^1(\h, \Dom \Delta_{0})\to v^*C_c^1(\h,L^2(\dH,\nu_0)).$$
It has $C_{0}(\h)$ locally compact resolvent in $L^{2}(T_{1}\h,\nu_x)_{C_{0}(\h)}$. In particular, for all $f\in C_{0}(\h)$ we have $f(1+\Delta)^{-1}\in \mathbb{K}(L^{2}(T_{1}\h,\nu_x)_{C_{0}(\h)})$.
\end{proposition}
\begin{proof} Using the map \eqref{trivialize!}, it suffices to show that the operator $vDv^{*}$  with domain $$ C_{c}^1(\h,\Dom \Delta_{0})\subset v(C_{c}^{1}(T_1\h)),$$  is essentially self-adjoint, regular and has $C_{0}(\h)$-compact resolvent in $ C_{0}(\h, L^{2}(\partial\h,\nu_{0}))$. We see that it is given by
\begin{equation}\label{vDv} v\Delta v^{*}\Psi(x,\xi)=\int\frac{\Psi(x,\xi)\frac{P^{n/2}(x,\xi)}{P^{n/2}(x,\eta)}-\Psi(x,\eta)}{d_{0}(\xi,\eta)^{n}}d\nu_{0}\eta=(\Delta_{0}\otimes 1 + H(x,\xi))\Psi(x,\xi),\end{equation}
where $H(x,\xi)\in C^1(\h, L^{\infty}(\dH,\nu_0))$ is the function from \eqref{Hdef}. In particular, $v\Delta v^{*}$ maps $C_{c}^{1}(\h, \Dom \Delta_0)$ into $C_{c}^{1}(\h, L^{2}(\dH,\nu_0))$. It follows that $\Delta$ is a map $$\Delta:v^{*}C_{c}^{1}(\h, \Dom \Delta_0)\to v^*C_{c}^{1}(\h, L^{2}(\dH,\nu_0))\subset L^{2}(T_{1}\h,\nu_x)_{C_{0}(\h)},$$ and thus $\Delta$ is a densely defined symmetric operator. By \cite[Theor\`{e}me 1.18]{Pierrot}, $v\Delta v^{*}$ is self-adjoint and regular if and only if for all $x\in \h$ the localisation $(v\Delta v^{*})_{x}$ is self-adjoint in $L^{2}(\partial\h,\nu_{0})$. But $(v\Delta v^{*})_{x}$ is a bounded perturbation of $\Delta_{0}$ by \eqref{vDv}, and therefore self-adjoint. Thus $v\Delta v^{*}$ is self-adjoint and regular. By the same argument, because $\Delta_{0}$ has compact resolvent, $(v\Delta v^{*})_{x}$ has compact resolvent and thus $v\Delta v^{*}$ has $C_{0}(\h)$-compact resolvent. Since $v$ is unitary, it follows that $\Delta$ is self-adjoint and regular, with $C_{0}(\h)$-compact resolvent.
\end{proof}

\begin{proposition}\label{wellbehaved}The operator $\Delta$ is positive and $G$-equivariant. Moreover, for $p$ the projection of Theorem \ref{extrep}, $\Delta+p$ is strictly positive and  $\Delta p=p\Delta=0$.
\end{proposition}
\begin{proof} By Proposition \ref{compact!} $\Delta_{0}$ is positive and $\Delta_{0}+p_{0}$ is strictly positive. Now consider $g$ with $0=xg$ and for $\Psi\in C_{c}(\h, C^{1}(\dH)),$ which is a core for $\Delta$, write $\Psi^{g}_{x}(\xi)=\Psi(x,\xi g^{-1})$, so $\Psi^{g}_{x}\in  C^1(\dH)$. Then compute
\begin{align*}\langle \Delta\Psi,\Psi\rangle(x)&=\int\int\frac{|\Psi(x,\xi)|^{2}-\overline{\Psi(x,\eta)}\Psi(x,\xi)}{d_{x}(\xi,\eta)^{n}}d\nu_{x}\xi d\nu_{x}\eta\\ & =\int\int\frac{|\Psi(x,\xi)|^{2}-\overline{\Psi(x,\eta)}\Psi(x,\xi)}{d_{0}(\xi g,\eta g)^{n}}d\nu_{0}(\xi g) d\nu_{0}(\eta g)\\
&=\int\int\frac{|\Psi(x,\xi g^{-1})|^{2}-\overline{\Psi(x,\eta g^{-1})}\Psi(x,\xi g^{-1})}{d_{0}(\xi ,\eta )^{n}}d\nu_{0}\xi  d\nu_{0}\eta =\langle \Delta_{0} \Psi^{g}_{x}, \Psi^{g}_{x} \rangle_{0},\end{align*}
and similarly one shows that $\langle p\Psi, \Psi\rangle (x)=\langle p_{0}\Psi ^{g}_{x},\Psi^{g}_{x}\rangle_{0}$. It thus follows that $\Delta$ is positive and $\Delta+p$ is strictly positive. A simple change of variables establishes $G$-equivariance. The equality $\Delta p=0$ follows because $p\Psi(x,\xi)$ is constant in $\xi$, and $p\Delta=0$ thus follows by taking adjoints.
\end{proof}

\subsection{An unbounded Kasparov module for the extension class}
In Section \ref{geltsec} the function $\rho(x)=d_{\h}(0,x)$ was introduced, and on
 $C_{c}(\h, \Dom \Delta_0)$ we consider the multiplication operator $\rho \Psi(x,\xi)=\rho(x)\Psi(x,\xi)$.
  It is straightforward to show that $\Delta$ and $\rho$ commute on this domain and that  $\Delta+\rho$ is essentially self-adjoint and regular.
  \begin{proposition}\label{resolvent} The positive self-adjoint regular operator $\Delta+\rho$ has compact resolvent in the $C^{*}$-module $L^{2}(T_{1}\h,\nu_x)_{C_{0}(\h)}$.
\end{proposition}
\begin{proof} Because both $\Delta$ and $\rho$ are positive regular operators, $(1+\Delta+\rho)^{-1},(1+\Delta)^{-1}$ define adjointable operators and 
the function $(1+\rho)^{-1}$ is an element of $C_{0}(\h)$. Hence if $u_{n}$ is an increasing approximate unit in $C_{0}(\h)$, for $n>m$ we have the operator inequalities
\[0\leq (u_{n}-u_{m})(1+\Delta+\rho)^{-1}(u_{n}-u_{m})\leq (u_{n}-u_{m})(1+\rho)^{-1}(u_{n}-u_{m})\to 0.\]
Thus the sequence $u_{n}(1+\Delta+\rho)^{-1}$ is Cauchy. On the other hand
\[u_{n}(1+\Delta+\rho)^{-1}u_{n}\leq u_{n}(1+\Delta)^{-1}u_{n}\in \mathbb{K}_{C_{0}(\h)}(L^{2}(T_{1}\h,\nu_x)),\]
because $u_{n}(1+\Delta)^{-1}\in \mathbb{K}$ by Proposition \ref{compact!}.
Since $\mathbb{K}$ is an ideal, it is a hereditary subalgebra, and thus the operator $u_{n}(1+\Delta+\rho)^{-1}u_{n}\in\mathbb{K}$  from which it follows that $u_{n}(1+\Delta+\rho)^{-1/2}\in \mathbb{K}$, and since this sequence is Cauchy, its limit is in $\mathbb{K}$ as well. Because the sequence converges pointwise to $(1+\Delta+\rho)^{-1/2}$, it follows that $(1+\Delta+\rho)^{-1}\in\mathbb{K}$, as desired.
\end{proof}
Next we address the commutator properties of $\Delta$ and $\rho$ with functions $f\in \Lip(\dH)$.
\begin{lemma} \label{bdd1}For $f\in \Lip(\partial\h)$, the operator $[\Delta,f]$ extends to a bounded operator. 
\end{lemma}
\begin{proof} This commutator can be computed using the explicit formula for $v\Delta v^{*}$ from Equation \eqref{vDv} to find:
\begin{align*}[v\Delta v^{*},f]\Psi(x,\xi)=\int\frac{(f(\xi)- f(\eta))\Psi(x,\eta)}{d_{0}(\xi,\eta)^{n}} d\mu_{0}(\eta).
\end{align*}
Using H\"{o}lder's inequality, the fact that $f$ is Lipschitz and Lemma \ref{unicorn} we estimate
\begin{align*}|\langle [v\Delta v^{*},&f]\Psi,\Phi \rangle |(x)=\left |\int\int\frac{(f(\xi)- f(\eta))\Psi(x,\eta)}{d_{0}(\xi,\eta)^{n}} d\mu_{0}(\eta)\Phi(x,\xi)d\mu_{0}(\xi)\right| \\
&\leq \|f\|_{\textnormal{Lip}}\int\int\frac{|\Psi(x,\eta)||\Phi(x,\xi)|}{d_{0}(\xi,\eta)^{n-1}} d\mu_{0}(\eta)d\mu_{0}(\xi)\\
&\leq \|f\|_{\textnormal{Lip}}\left(\int\int\frac{|\Psi(x,\eta)|^{2}}{d_{0}(\xi,\eta)^{n-1}} d\mu_{0}(\eta)d\mu_{0}(\xi)\right)^{\frac{1}{2}}\left(\int\int\frac{|\Phi(x,\xi)|^{2}}{d_{0}(\xi,\eta)^{n-1}} d\mu_{0}(\eta)d\mu_{0}(\xi)\right)^{\frac{1}{2}}\\
&\leq\|f\|_{\textnormal{Lip}}I_{1}\|\Phi\|_{2}\|\Psi\|_{2},
\end{align*}
which is independent of $x$. Thus $[v\Delta v^{*},f]$ extends to an adjointable operator.
\end{proof}
Next we recall the function $\rho(x)=d_{\h}(0,x)$ and the projection $p$ from Theorem \ref{extrep}. %
\begin{lemma} \label{bdd2}For $f\in\Lip(\partial\h)$, the operator $[p,f]\rho$ extends to a bounded operator.
\end{lemma}
\begin{proof} The proof consists of pointwise estimates in $x\in\h$:
\begin{align}|\nonumber\langle [p,f&]\Psi,\Phi \rangle (x)|  \leq\int\int |\Psi(x,\eta)\Phi(x,\xi)||f(\xi)-f(\eta)|d\nu_{x}(\xi)d\nu_{x}(\eta) \\
\nonumber&\leq \|f\|_{\Lip}\int\int |\Psi(x,\eta)\Phi(x,\xi)|\|\xi-\eta\|d\nu_{x}(\xi)d\nu_{x}(\eta) \\
\nonumber&\leq  \|f\|_{\Lip}(1-\|x\|^{2})\int\int |\Psi(x,\eta)\Phi(x,\xi)|\frac{\|x-\xi\|+\|x-\eta\|}{1-\|x\|^{2}} d\nu_{x}(\xi)d\nu_{x}(\eta) \\
\label{almostthere}&\leq  \|f\|_{\Lip}(1-\|x\|^{2})\int\int |\Psi(x,\eta)\Phi(x,\xi)| (P(x,\xi)^{-\frac{1}{2}}+P(x,\eta)^{-\frac{1}{2}} )d\nu_{x}(\xi)d\nu_{x}(\eta) 
\end{align}
Thus we estimate
\[\int\int |\Phi(x,\xi)\Psi(x,\eta)|P^{-1/2}(x,\xi)d\nu_{x}\xi d\nu_{x}\eta\leq \|\Phi\|_{2}\|\Psi\|_{2}\left(\int\int P^{-1}(x,\xi)d\nu_{x}\xi d\nu_{x}\eta\right)^{\frac{1}{2}}.\]
Using H\"{o}lder's inequality, the fact that $\nu_{x}$ is a probability measure and $d\nu_{x}=P^{n}d\nu_{0}$ we find
\[\int P^{-1}(x,\xi)d\nu_{x}(\xi)=\int P^{n-1}(x,\xi)d\nu_{0}\xi\leq \left(\int P^{n}(x,\xi)d\nu_{0}\xi\right)^{\frac{n-1}{n}}= 1.\]
Combining this with \eqref{almostthere} we obtain the estimate
\[|[p,f]\Psi,\Phi \rangle (x)| \leq 2\|f\|_{\Lip}\|\Phi\|_{2}\|\Psi\|_{2}(1-\|x\|^{2}).\]
An elementary computation using the explicit distance formula on the hyperbolic ball (see \cite[Section I.6.7]{metricspaces}) shows that $\rho(x)=\log\frac{(1+\|x\|)^{2}}{1-\|x\|^{2}}$. It thus follows that $\rho[p,f]$ is bounded. \end{proof}
This leads us to consider the operator $S:=-\Delta+\rho F_{p}$, as a candidate for the unbounded representative of the Fredholm module constructed in Theorem \ref{extrep}. We arrive at the main result of this section.
\begin{theorem}\label{unbddrep} The triple $(C(\partial\h),L^{2}(T_1\h,\nu_x)_{C_{0}(\h)}, S)$ is an unbounded $G$-equivariant Kasparov module representing the $G$-equivariant extension
\[0\to C_{0}(\h)\to C(\overline{\h})\to C(\partial\h)\to 0.\]
\end{theorem}
\begin{proof} We have $S=-\Delta+\rho F_{p}=F_{p}(\Delta+\rho)=F_{p}(1+\Delta+\rho)-F_{p}$ by Proposition \ref{wellbehaved}. Now $1+\Delta+\rho$ commutes with $F_{p}$ and has compact inverse by Proposition \ref{resolvent}, so $S$ has compact resolvent in $L^{2}(T_{1}\h)_{C_{0}(\h)}$. We note that since
\[[S,f]=[\rho,f]F_{p}-[\Delta,f] +2[p,f]\rho,\]
it follows by Lemmas \ref{bdd1} and \ref{bdd2} that $[S,f]$ extends to a bounded adjointable operator whenever $f\in \Lip(\dH)$. The operators $\Delta$ and $F_{p}$ are $G$-invariant by Theorem \ref{extrep} and Proposition \ref{wellbehaved}, and the function $\rho$ commutes boundedly with $G$. Thus $S$ defines a $G$-equivariant cycle, and by construction, the bounded transform defines the same class as $F_{p}$, so we are done.\end{proof}

\section{Embedded surfaces and $K$-cycles for the boundary crossed product}\label{section: boundary}
In this section we explicitly compute the boundary map $\partial: K^{0}(C_{0}(M))\to K^{1}(C(\partial\h)\rtimes\Gamma)$ for classes $[\Dsla_{\mathring{N}}]\in K^{0}(C_{0}(M))$ attached to a surface $(N, \partial N) \subset (\overline{M}, \partial \overline{M})$. For Bianchi groups, this gives an exhaustive description of the map $\partial$ by Theorem \ref{geometriciso}. 
\subsection{Hyperbolic Dirac operator and Poisson kernel}
We now prove several technical Lemmas concerning the commutation relations of the hypersingular integral operator $\Delta$  \eqref{Diff}, the projection $p$ from Theorem \ref{extrep} and the hyperbolic Dirac operator $\Dsla_{\h}$. In the next section we use these results to compute the Kasparov product of the boundary extension with Dirac operators on embedded surfaces. 

We set some conventions. Following Patterson \cite[Page 294]{Patterson} we view $\h$ as the unit ball in $\R^{n+1}$ as before, with the Riemannian metric $ds^{2}=\frac{dx^{2}}{(1-\|x\|^{2})^{2}}$. Let $\mathcal{S}\to\h$ be the spinor bundle and consider the internal tensor product of $C_0(\h)$-$C^*$-modules 
$$C_{0}(\h,L^{2}(\dH,\nu_0))\otimes_{C_0(\h)} C_{0}(\h,\mathcal{S})\simeq L^{2}(\dH,\nu_0)\otimes C_{0}(\h,\mathcal{S}),$$

and the dense subspace
\begin{equation}\label{wh}W=W_{\h}=C_{c}^{1}(\h,\Dom \Delta_{0})\otimes_{C^{1}_{c}(\h)}^{\alg}C_{c}^{1}(\h,\mathcal{S}).\end{equation}
Let $\Psi\in C_{c}^{1}(\h,\mathcal{S})$ be a compactly supported $C^1$-section and $\x$ a tangent vector field. We denote the Clifford representation associated to the hyperbolic metric on $\h$ by $\psi\mapsto c(\x)\psi$. Let $e_{i}\in \R^{n+1}$ denote the $i$-th standard basis vector. The vector fields $\e_{i}(x)=e_{i}(1-\|x\|^{2})$ define a global orthonormal frame for the  the tangent bundle on $\h$. The hyperbolic Dirac operator $\Dsla_{\h}$ can be computed by elementary methods (see for instance \cite[Theorem 5.3.5]{GilMur}) and  is given by 

\[\Dsla_{\h}:C^{1}_{c}(\h,\mathcal{S})\to C_{c}(\h,\mathcal{S}), \quad \psi\mapsto \sum_{i=0}^{n}(1-\|x\|^{2})c(\e_i)\partial_{i}\psi + L_{i}c(\e_i)\psi.\]
Here the $L_{i}$ are bounded functions on $\h$. It induces an operator $$\Tsla_{\h}:C^{1}_{c}(\h, L^{2}(\dH,\nu_0))\otimes_{C_{c}^{1}(\h)}^{\alg} C_{c}^{1}(\h,\mathcal{S})\to C_{0}(\h,L^{2}(\dH,\nu_0))\otimes_{C_0(\h)}^{\alg} C_{0}(\h,\mathcal{S}),$$
\begin{align*} \Tsla_{\h}(\Psi\otimes \psi) &= \Psi\otimes \Dsla_{\h}\psi + \partial_{i}\Psi\otimes c(\e_{i})\psi\\
&=(1-\|x\|^{2})\left(\sum_{i=0}^{n}\Psi\otimes c(\e_{i})(\partial_{i}\psi)+(\partial_{i}\Psi)\otimes c(\e_{i})\psi\right)+L_{i}\Psi\otimes c(\e_{i})\psi,\end{align*}
For $s\in\R$, the powers of the Poisson kernel define multiplication operators 
$$P^{s}: C_{c}^{1}(\h, L^{2}(\dH,\nu_0))\to C_{c}^{1}(\h, L^{2}(\dH,\nu_0)),\quad \Psi\mapsto P^{s}\Psi.$$ 
\begin{lemma} \label{Poissonderivative} Let $s\in \R$ and $\Psi\otimes\psi\in W_{\h}$. The operators $P^{s}$ and $\Tsla_{\h}$ satisfy 
$$\Tsla_{\h}:W_{\h}\to C_{c}(\h,\Dom \Delta_0)\otimes_{C_c(\h)}^{\alg} C_{c}(\h,\mathcal{S}),$$
$$P^{s}:W_{\h} \to C_{c}^{1}(\h,L^{2}(\dH,\nu_0))\otimes_{C^{1}_{c}(\h)}^{\alg}C_{c}^{1}(\h,\mathcal{S}).$$ There are functions $u_{i}\in C_{b}(T_{1}\h)$ with $\sum_{i=0}^{n}u_{i}^{2}=1$ such that $$[\Tsla_{\h}, P^{s}](\Psi\otimes\psi)=2sP^{s}\sum_{i=0}^{n}c(\e_{i})u_{i}(\Psi\otimes\psi).$$ 
\end{lemma}
\begin{proof} The domain mapping properties are straightforward to check and guarantee that commutator $[\Tsla_{\h},P^{s}]$ is well-defined.  By the derivation property of $\Tsla_{\h}$, we need to compute
\[(1-\|x\|^{2})\partial_{i}\left(P(x,\xi)^{s}\right)=2sP(x,\xi)^{s}\frac{(\xi_{i}-x_{i})(1-\|x\|^{2})-x_{i}\|\xi-x\|^{2} }{\|x-\xi\|^{2} },\]
so we have
\begin{equation}\label{Poissoncomm}
[\Tsla_{\h}, P^{s}](\Psi\otimes\psi)=2sP^{s}\sum_{i}c(\e_{i})u_{i}(\Psi\otimes\psi), \quad u_{i}(x,\xi)=\frac{(\xi_{i}-x_{i})(1-\|x\|^{2})-x_{i}\|\xi-x\|^{2} }{\|x-\xi\|^{2} }
\end{equation}
Now since $$ |\xi_i-x_i|\leq\left(\sum_{k=0}^{n}(x_{k}-\xi_{k})^{2}\right)^{\frac{1}{2}}\leq \|x-\xi\| ,\quad  1-\|x\|^{2}\leq (1+\|x\|)\|x-\xi\| $$ we find $|u_{i}(x,\xi)|\leq 1+ \|x\|-x_{i}$ which is a bounded function. The proof that $\sum u_{i}^{2}=1$ is an elementary computation which we omit.
\end{proof}
The operator $v\Delta v^{*}$ from \eqref{vDv} induces and operator
$$(v\Delta v^{*})\otimes 1: W_{\h}\to C_{c}^{1}(\h, L^{2}(\dH,\nu_0))\otimes_{C_{c}^{1}(\h)}^{\alg} C_{c}^{1}(\h,\mathcal{S}), $$
which we will denote by $v\Delta v^{*}$ as well.
\begin{proposition}\label{comm} The operators $(v\Delta v^{*})\otimes 1$ and $(vpv^{*})\otimes 1$ satisfy
$$(v\Delta v^{*})\otimes 1, (vpv^{*})\otimes 1:W_{\h}\to C^{1}_{c}(\h, L^{2}(\dH,\nu_{0}))\otimes_{C_{c}^{1}(\h)}^{\alg}C_{c}^{1}(\h,\mathcal{S}).$$
For $\Psi\otimes\psi\in W_{\h}$ it holds that 
\begin{align*}[T_\h, vpv^{*}] (\Psi\otimes\psi) &=n\sum_{i}(v(u_i p+ pu_i )v^{*})\Psi\otimes c(\e_{i})\psi,\\
[v\Delta v^{*},\Tsla_{\h}](\Psi\otimes\psi) &=n\sum_{i}g_{i}\Psi\otimes c(\e_{i})\psi, \quad g_{i}(x,\xi):= \left(\int \frac{(u_{i}(x,\xi)-u_{i}(x,\eta))}{d_{x}(\xi,\eta)^{n}}d\nu_{x}(\eta)\right).\end{align*}
Moreover, the functions $g_{i}:T_{1}\h\to \R$ satisfy $\sup_{(x,\xi)}\sum g_{i}(x,\xi)^{2}<\infty$  . 
\end{proposition}
\begin{proof} The operator $vpv^{*}$ can be written as
\[vpv^{*}\Psi(x,\xi)=P(x,\xi)^{n/2}\int \Psi(x,\eta)P^{n/2}(x,\eta)d\nu_0\eta,\]
from which the domain mapping property follows readily. The formula for commutator is a direct application of Equation \ref{Poissoncomm}. We turn to the operator $v\Delta v^{*}$. Recall from Equation \eqref{vDv} that the operator $v\Delta v^{*}$ can be written as
$v\Delta v^{*}=\Delta_{0}\otimes 1 + H,$
with $H$ as in Equation \eqref{Hdef}. By Lemma \ref{Hlem}, $H\in C^{1}(\h, L^{2}(\dH,\nu_0))$ and thus $H$ multiplies $C_{c}^{1}(\h, \Dom \Delta_{0})$ into $C_{c}^{1}(\h, L^{2}(\dH,\nu_0))$. Clearly $\Delta_{0}\otimes 1$ maps  $C_{c}^{1}(\h, \Dom \Delta_{0})$ into $C_{c}^{1}(\h, L^{2}(\dH,\nu_0))$ as well. Therefore 
$$(v\Delta v^{*})\otimes 1:W_{\h}\to C^{1}_{c}(\h, L^{2}(\dH,\nu_{0}))\otimes_{C_{c}^{1}(\h)}^{\alg}C_{c}^{1}(\h,\mathcal{S}),$$
and the commutator $[\Tsla_{\h}, (v\Delta v^{*})\otimes 1]$ is well defined and equals $[\Tsla_{\h}, H\otimes 1]$. By Lemma \ref{Poissonderivative} and the Leibniz rule $$[\Tsla_{\h},P(x,\eta)^{n/2}P(x,\xi)^{-n/2}]=nP(x,\eta)^{n/2}P(x,\xi)^{-n/2}\sum_{i}c(\e_{i})(u_{i}(x,\xi)-u_{i}(x,\eta)).$$
It now follows that on $W_{\h}$
\begin{align*}[\Tsla_{\h}, H\otimes 1]&=n\sum_{i}c(\e_i)\int\frac{P(x,\eta)^{n/2}P(x,\xi)^{-n/2}(u_{i}(x,\xi)-u_{i}(x,\eta))}{d_{0}(\xi,\eta)}d\nu_{0}\eta\\
&=n\sum_{i}c(\e_i)\int\frac{(u_{i}(x,\xi)-u_{i}(x,\eta))}{P(x,\xi)^{n/2}P(x,\xi)^{n/2}d_{0}(\xi,\eta)}P(x,\eta)^{n}d\nu_{0}\eta \\ &= n\sum_{i}c(\e_i)\int\frac{(u_{i}(x,\xi)-u_{i}(x,\eta))}{d_{x}(\xi,\eta)}d\nu_{x}\eta,\end{align*} 
as claimed. To prove that $\sup_{(x,\xi)}\sum g_{i}(x,\xi)^{2}<\infty,$ it suffices to show that
\begin{equation}\label{normestimate}\sup_{(x,\xi)}\sum_{i}\left(\int \frac{(u_{i}(x,\xi)-u_{i}(x,\eta))}{d_{x}(\xi,\eta)^{n}}d\mu_{x}(\eta)\right)^{2} \leq \sup_{(x,\xi)}\sum_{i}\left(\int \frac{|u_{i}(x,\xi)-u_{i}(x,\eta)|}{d_{x}(\xi,\eta)^{n}}d\mu_{x}(\eta)\right)^{2},
\end{equation}is finite. By H\"{o}lder's inequality we have, for $0<s<1$ that 
\begin{equation}\label{Rieszestimate}\left(\int \frac{|u_{i}(x,\xi)-u_{i}(x,\eta)|}{d_{x}(\xi,\eta)^{n}}d\nu_{x}(\eta)\right)^{2}\leq \int \frac{(u_{i}(x,\xi)-u_{i}(x,\eta))^{2}}{d_{x}(\xi,\eta)^{n+s}}d\nu_{x}(\eta)\int \frac{1}{d_{x}(\xi,\eta)^{n-s}}d\nu_{x}(\eta).\end{equation}
A lengthy but elementary calculation shows that 
\begin{equation}\label{expr}\sum_{i=0}^{n}(u_{i}(x,\xi)-u_{i}(x,\eta))^{2}=\|\xi-\eta\|^{2}P(x,\xi)P(x,\eta)=d_{x}(\xi,\eta)^{2}.
\end{equation}
Combining \eqref{normestimate}, \eqref{Rieszestimate} and \eqref{expr}, we find
\begin{align*}\sum_{i}&\left(\int \frac{(u_{i}(x,\xi)-u_{i}(x,\eta))}{d_{x}(\xi,\eta)^{n}}d\nu_{x}(\eta)\right)^{2} \leq I_{s} \sum_{i} \int \frac{(u_{i}(x,\xi)-u_{i}(x,\eta))^{2}}{d_{x}(\xi,\eta)^{n+s}}d\nu_{x}(\eta)\\
&= I_{s} \int \frac{d_{x}(\xi,\eta)^{2}}{d_{x}(\xi,\eta)^{n+s}}d\nu_{x}(\eta)
= I_{s} \int \frac{1}{d_{x}(\xi,\eta)^{n+s-2}}d\nu_{x}(\eta)
= I_{s}I_{2-s},\end{align*}
which proves boundedness by Lemma \ref{unicorn}.\end{proof}
 \subsection{Kasparov products with embedded surfaces} 
 By Theorem \ref{geometriciso}, any element  of the group $K^{0}(C_{0}(M))$ of a Bianchi manifold $M$ can be represented by the self-adjoint Dirac operator on a closed embedded hypersurface $\intN\to M$. Throughout we will use that the spinor bundle on a closed embedded hypersurface is the restriction of the spinor bundle of the ambient manifold (cf. \cite{Baer, Diracembedded}).
 
We will consider the embedded hypersurface $\Sigma:=\pi^{-1}(\intN)\subset \h$ inside the universal cover $\h$ of $M$ and denote by $\nv$ the unit normal vector field of $\Sigma\subset \h$. Let $\mathcal{S}\to\h$ be the spinor bundle of $\h$, which is the pullback of the spinor bundle of $M$ under the covering map $\pi:\h\to M$. The Clifford module structure on $\mathcal{S}|_{\Sigma}$ is given by $$c_{\Sigma}(\x)\psi:=c(\x)c(\nv)\psi,$$ with $\x$ a vector field on $\Sigma$. We denote by $\Dsla_{\Sigma}:C_{c}^{1}(\Sigma,\mathcal{S}|_{\Sigma})\to C_{c}(\Sigma,\mathcal{S}|_{\Sigma})$ the Dirac operator on $\mathcal{S}|_{\Sigma}$ associated to this Clifford module structure. The map 
\begin{equation}\label{sigmagrading} \sigma:\mathcal{S}|_{\Sigma}\to \mathcal{S}|_{\Sigma},\quad  \psi\mapsto ic(\nv)\psi\end{equation} 
is self-adjoint for the Riemannian inner product on $\mathcal{S}$ and squares to $1$. As such $\sigma$  induces a grading on $\mathcal{S}|_{\Sigma}$, giving a decomposition $\mathcal{S}|_{\Sigma}^{+}\oplus \mathcal{S}|_{\Sigma}^{-}$. Moreover, since $\mathcal{S}|_{\h}$ is $G$-equivariant, $\mathcal{S}|_{\Sigma}$ is $\Gamma$-equivariant. Similar relations hold for the spinor bundles of $M$ and $\mathring{N}$. 

Let $\langle \psi,\phi\rangle_{\mathcal{S}}$ denote the inner product on the spinor bundle and $C_{0}(\Sigma,\mathcal{S})$ the associated $C_{0}(\Sigma)$ module of sections. Moreover, we write $L^{2}_{\pi}(\Sigma,\mathcal{S}|_{\Sigma})$ for the $C_{0}(\intN)$-module obtained as the completion of $C_{c}(\Sigma,\mathcal{S})$ in the inner product norm given by
\[\langle\psi,\phi\rangle(n):=\sum_{x\in\pi^{-1}(n)}\langle \psi,\phi\rangle_{\mathcal{S}}(x)\in C_{0}(\intN).\]

\begin{proposition}\label{tensorisos} The right $C_{c}(\mathring{N})$ module map 
\begin{align*}v:C_{c}(T_{1}(\h))\otimes_{C_{c}(M)}^{\alg}C_{c}(\mathring{N},\mathcal{S}_{\mathring{N}})&\to C_{c}(\Sigma, L^{2}(\dH,\nu_0))\otimes_{C_{c}(\Sigma)}^{\alg}C_{c}(\Sigma,\mathcal{S}|_{\Sigma}),\\
\Psi\otimes\psi &\mapsto (P^{n/2}\cdot\Psi)|_{\Sigma}\otimes \pi^{*}\psi \end{align*}
is well-defined and extends to a unitary isomorphism of $C_{0}(\mathring{N})$ modules
 $$v: L^{2}_{\pi}(T_{1}\h,\nu_x)\otimes_{C_{0}(M)} C_0(\intN,\mathcal{S}_{\intN})\to C_{0}(\Sigma, L^{2}(\dH,\nu_0))\otimes_{C_{0}(\Sigma)}L^{2}_{\pi}(\Sigma, \mathcal{S}|_{\Sigma})_{C_{0}(\mathring{N})},$$ 
 and to a unitary isomorphism of Hilbert spaces
 $$v: L^{2}_{\pi}(T_{1}\h,\nu_x)\otimes_{C_{0}(M)} L^{2}(\intN,\mathcal{S}_{\intN})\to C_0(\Sigma, L^{2}(\dH,\nu_0))\otimes_{C_{0}(\Sigma)}L^2(\Sigma, \mathcal{S}|_{\Sigma}).$$ 
Here the latter Hilbert space is the completion of $C_{c}(\Sigma, L^{2}(\dH,\nu_0))\otimes_{C_{c}(\Sigma)}^{\alg}C_{c}(\Sigma,\mathcal{S}|_{\Sigma})$ in the inner product
\begin{equation}\label{sigmainn}
\langle \Psi\otimes \psi,\Phi\otimes\phi\rangle_{\mu}=\int_{\Sigma}\langle  \Psi\otimes\psi, \Phi\otimes\phi\rangle_{C_0(\Sigma)}(x)d\mu(x),
\end{equation}
where $\langle  \Psi\otimes\psi, \Phi\otimes\phi\rangle_{C_{0}(\Sigma)}$ denotes the inner product on $C_{0}(\Sigma,L^{2}(\dH,\nu_0))\otimes_{C_0(\Sigma)}C_{0}(\Sigma,\mathcal{S}).$
The algebra $C(\dH)$ acts by pointwise multiplication and the left  $\Gamma$-representation is given by \begin{equation}\label{gammarep}v( u_{\gamma}\otimes 1) v^{*}(\Psi\otimes\psi)=|\gamma'|^{-\frac{n}{2}}(\gamma\circ\Psi)\otimes (\gamma\circ\psi).\end{equation} 
\end{proposition}
\begin{proof} First note that $v$ is well defined for if $\chi\in C_{c}(\Sigma)$ is a function such that $\chi=1$ on $\supp\Psi|_{\Sigma}$, then $$(P^{n/2}\cdot\Psi)|_{\Sigma}\otimes \pi^{*}\psi=(P^{n/2}\cdot\Psi)|_{\Sigma}\otimes \chi \pi^{*}\psi\in C_{c}(\Sigma, L^{2}(\dH,\nu_0))\otimes_{C_{c}(\Sigma)}C_{c}(\Sigma,\mathcal{S}|_{\Sigma}).$$ The balancing relation is respected by $v$ for if $g\in C_{0}(M)$ and $\chi$ is as above then
\begin{align*}v(\Psi\cdot g\otimes \psi)&=v(\Psi\chi\pi^{*}(g)\otimes \psi)=P^{n/2}\cdot \Psi\chi \pi^{*}(g)\otimes \pi^{*}(\psi)\\ &=P^{n/2}\cdot\Psi\otimes \chi \pi^{*}(g)|_{\Sigma}\pi^{*}\psi = P^{n/2}\cdot\Psi\otimes \chi  \pi^{*}(g|_{N}\psi)=v(\Psi\otimes g\cdot \psi).\end{align*}
Compatibility of the inner products follows by
\begin{align*}\langle v(\Psi\otimes \psi),& v(\Phi\otimes \phi) \rangle(n)=\langle \pi^{*}\psi, \langle P^{n/2}\Psi,P^{n/2}\Phi\rangle_{C_{0}(\Sigma,L^{2})} \pi^{*}\phi\rangle(n)\\&
=\sum_{x\in\pi^{-1}(n)}\langle \pi^{*}\psi, \langle P^{n/2}\Psi,P^{n/2}\Phi\rangle_{C_{0}(\Sigma,L^{2})} \pi^{*}\phi\rangle_{\mathcal{S}_{\Sigma}}(x)\\ &=\sum_{x\in\pi^{-1}(n)}\langle \psi(n), \langle \Psi,\Phi\rangle_{L^{2}(T_1\h,\nu_{x})} (x) \phi(n)\rangle_{\mathcal{S}_{\intN}}\\
&=\langle \psi(n), \langle \Psi,\Phi\rangle_{L^{2}_{\pi}(T_1\h,\nu_{x})_{C_{0}(M)}} (n) \phi(n)\rangle_{\mathcal{S}_{\intN}}=\langle \Psi\otimes \psi, \Phi\otimes \phi \rangle(n),
\end{align*}
so it remains to show that $v$ has dense range. Choose a pre-compact open cover $\{U_{i}\}$ of $\Sigma$ with the property that $U_{i}\gamma \cap U_{i}=\emptyset$ whenever $\gamma\neq e$ and let $\chi_{i}^{2}$ be a partition of unity subordinate to $\{U_{i}\}$. Then for each $\chi_{i}$, $\psi\in C_{c}(\Sigma,\mathcal{S}|_{\Sigma})$ and $\Psi\in C_{c}(\Sigma, L^{2}(\dH,\nu_0))$ there is a section $\psi_{i}\in C_{c}(\mathring{N},\mathcal{S})$ such that $\chi_{i}\psi=\pi^{*}\psi_{i}|_{U_{i}}$, and a function $\Psi_{i}\in C_{c}(T_1\h)$ such that $\Psi\chi_{i}=P^{n/2}\Psi_{i}$. Now choose functions $f_{i}\in C_{c}(\Sigma)$ with $f_{i}=1$ on $U_{i}$, so that
\[\Psi\otimes \psi=\sum \Psi\chi_{i}f_{i}\otimes \chi_{i}\psi=\sum \Psi\chi_{i}\otimes f_{i}\pi^{*}\psi_{i}=\sum \Psi\chi_{i}\otimes \pi^{*}\psi_{i}=\sum v(\Psi_{i}\otimes\phi_{i}),\]
which shows that $v$ is surjective and thus extends to unitary isomorphism of the $C^{*}$-completions. To see that $v$ extends to the Hilbert space completions, we need only observe that $\pi:\Sigma\to \mathring{N}$ is a local isometry, so $L^{2}(\Sigma,\mathcal{S}|_{\Sigma},\mu)\simeq C_{0}(\Sigma,\mathcal{S}|_{\Sigma})\otimes_{C_{0}(\mathring{N})}L^{2}(\mathring{N},\mu)$. The statements about the algebra representations follow by straightforward calculation.\end{proof}
Similar to \eqref{wh} we consider the subspace $W_{\Sigma}:=C_{c}^{1}(\Sigma,\Dom \Delta_0)\otimes_{C^{1}_{c}(\Sigma)}^{\alg}C_{c}^{1}(\Sigma,\mathcal{S}|_{\Sigma})$. The restriction maps
$$C_{c}^{1}(\h)\to C_{c}^{1}(\Sigma),\quad C_{c}^{1}(\h,\mathcal{S})\to C_{c}^{1}(\Sigma,\mathcal{S}|_{\Sigma}),\quad C_{c}^{1}(\h,\Dom \Delta_{0})\to C_{c}^{1}(\Sigma,\Dom \Delta_{0}), $$
are all surjective and we conclude that restriction gives a surjection $W_{\h}\to W_{\Sigma}$. 

The closed embedded surface $\Sigma\subset \h$ admits unit normal vector field $\nv$, which we can extend locally to a vector field on $\h$. For $x\in\Sigma$ let $\{{\bf x}_{i}\}_{i=0}^{n}$ be a local orthonormal frame at $x$ for the tangent bundle of $\h$, with $\x_{n}=\nv$. 
We define $\Tsla_{\Sigma}:W_{\Sigma}\to C_{0}(\Sigma, L^{2})\otimes_{C_{0}(\Sigma)} C_{0}(\Sigma,\mathcal{S}|_{\Sigma})$ by
$$ \Tsla_{\Sigma}(\Psi\otimes \psi) (x):=(1-\|x\|^{2})\left(\sum_{k=0}^{n-1}\partial_{\x_{k}}\Psi(x)\otimes c_{\Sigma}(\x_{k})\psi(x)\right)+\Psi(x)\otimes (\Dsla_{\Sigma}\psi)(x).$$
\begin{lemma}\label{Tprops}The operator $ \Tsla_{\Sigma}$ is essentially self-adjoint on $W_{\Sigma}\subset C_{0}(\Sigma, L^2(\partial\h,\nu_{0}))\otimes_{C_{0}(\Sigma)} L^{2}(\Sigma,\mathcal{S})$. Moreover $T_{\Sigma}$ commutes with functions $f\in C(\dH)$ and the $\Gamma$-representation \eqref{gammarep}. \end{lemma}
\begin{proof} The operator $\Dsla_{\Sigma}:C_{c}^{1}(\Sigma,\mathcal{S})\to C_{c}(\Sigma,\mathcal{S})$ is essentially self-adjoint because $\Sigma$ is a complete manifold.
The subspace $L^{2}(\dH,\nu_{0})\otimes^{\alg}C_{c}(\Sigma)\subset C_{0}(\Sigma, L^2(\partial\h,\nu_{0}))$ is dense and
$$L^{2}(\dH,\nu_{0})\otimes^{\alg} C_{c}(\Sigma)\otimes_{C_{c}(\Sigma)} C_{c}(\Sigma,\mathcal{S})\to L^{2}(\dH,\nu_{0})\otimes C_{c}(\Sigma,\mathcal{S}), \quad f\otimes g\otimes\psi\mapsto f\otimes g\psi,$$
extends to a unitary isomorphism 
\begin{equation}\label{alpha} \alpha :C_{0}(\Sigma, L^2(\partial\h,\nu_{0}))\otimes_{C_{0}(\Sigma)} L^{2}(\Sigma,\mathcal{S})\to L^{2}(\dH,\nu_{0})\otimes L^{2}(\Sigma,\mathcal{S}).\end{equation}
For $f\in L^{2}(\dH,\nu_{0})$, $g\in C^{1}_{c}(\Sigma)$ and $\psi\in C^{1}_{c}(\Sigma,\mathcal{S})$ we have $f\otimes g\otimes\psi \in W_{\Sigma}$ and
$$(\alpha\circ \Tsla_{\Sigma})(f\otimes g\otimes\psi)=f\otimes \Dsla_{\Sigma}(g\psi),\quad \alpha T_{\Sigma}\alpha^{-1}=1\otimes\Dsla_{\Sigma}. $$
While the elements $g\psi$ span the core $C_{c}^{1}(\Sigma,\mathcal{S})$ on which $\Dsla_{\Sigma}$ is essentially selfadjoint it follows that $\alpha\circ (\Tsla_{\Sigma}\pm i)$ has dense range, and hence so does $\Tsla_{\Sigma}\pm i$. Lastly, it is clear that $T_{\Sigma}$ commutes with functions $f\in C(\dH)$. For the $\Gamma$-representation \eqref{gammarep}, it is enough to observe that under the map $\alpha$ above we have
\[\alpha[ T_{\Sigma}, vu_{\gamma}v^{*}]\alpha^{-1} f\otimes\psi =[1\otimes \Dsla_{\Sigma},|\gamma'|^{-\frac{n}{2}}u_{\gamma}] f\otimes \psi =0,\]
because $\mathcal{S}$ is $\Gamma$-equivariant and $\gamma'$ does not depend on $x$.\end{proof}
Let $\psi\in C_{c}^{1}(\h,\mathcal{S})$ and $x\in\Sigma$. By \cite[Proposition 2.2]{Baer} the hyperbolic Dirac operator $\Dsla_{\h}$ and the surface Dirac operator $\Dsla_{\Sigma}$ are related by the formulae
\[c(\nv)\Dsla_{\h}\psi(x)=\Dsla_{\Sigma}\psi(x)-K\psi(x)+\nabla_{\nv}^{\h}\psi(x),\quad (\Dsla_{\h}c(\nv)-c(\nv)\Dsla_{\h})\psi(x)=2\Dsla_{\Sigma}\psi(x).\]
Here $\psi$ is a section of the spinor bundle on $\h$ defined in a neighbourhood of $\Sigma$,  $\nabla^{\h}$ is the spin connection in the spinor bundle of $\h$ and $K$ is the \emph{mean curvature} of the surface $\Sigma$. For $\Psi\otimes \psi\in W_{\h}$ and $x\in\Sigma$ we find the simple formula  
\begin{equation}\label{surfacefactor} (T_{\h}c(\nv)-c(\nv)T_{\h})(\Psi\otimes\psi)(x)=2\Tsla_{\Sigma}(\Psi\otimes\psi)(x).\end{equation}
Choose $\lambda_{k}^{i}\in C_c(\h)$ such that $\e_{i}=\sum_{k}\lambda_{k}^{i}\x_{k}$ in a neighbourhood of $x$. The following Lemma allows us to exploit the commutator computations of the previous section.
\begin{lemma}\label{framelemma} Let $R:W_{\h} \to C_{c}^{1}(\h,L^{2}(\dH,\nu_0))\otimes_{C^{1}_{c}(\h)}^{\alg}C_{c}^{1}(\h,\mathcal{S})$ be an operator such that $[T_{\h},R]=\sum_{i=0}^{n} c(\e_i)R_{i}$, where 
$$R_i: W_{\h} \to C_{c}^{1}(\h,L^{2}(\dH,\nu_0))\otimes_{C^{1}_{c}(\h)}^{\alg}C_{c}^{1}(\h,\mathcal{S})$$
are operators that commute with the Clifford action. Then for $\Psi\otimes\psi\in W_{\h}$
\begin{equation}\label{TSigmacomm}\left([T_{\Sigma},R]\right)(\Psi\otimes\psi)(x)=\sum_{k=0}^{n-1}\sum_{i=0}^{n}\lambda_{k}^{i}R_{i}c_{\Sigma}(\x_{k})(\Psi\otimes \psi)(x) ,\end{equation}
locally at $x$. If the $R_{i}$ define bounded operators on $C_0(\Sigma, L^{2}(\dH,\nu_0))\otimes_{C_{0}(\Sigma)}L^2(\Sigma, \mathcal{S}|_{\Sigma})$ then $[T_{\Sigma},R]$ extends to a bounded operator on $C_0(\Sigma, L^{2}(\dH,\nu_0))\otimes_{C_{0}(\Sigma)}L^2(\Sigma, \mathcal{S}|_{\Sigma})$.
\end{lemma}
\begin{proof} Equation \eqref{TSigmacomm} is obtained by linear algebra using Equation \eqref{surfacefactor} and the local relations $c(\x_{k})^{2}=1$ for $0\leq k \leq n$ and $c(\x_{i})c(\x_{j})=-c(\x_{j})c(\x_{i})$ for $i\neq j$. To see that $[T_{\Sigma}, R]$ defines a bounded operator whenever the $R_{i}$ are bounded we first compute the $C_{0}(\Sigma)$-valued inner product $\langle\cdot,\cdot\rangle_{C_{0}(\Sigma)}$ from Equation \ref{sigmainn}, and do a pointwise estimate:
\begin{align*}\langle [\Tsla_{C_0(\Sigma)}, R]w, [\Tsla_{\Sigma}, R]w\rangle_{C_0(\Sigma)}(x)
&=\sum_{i=0}^{n}\left(\sum_{k=1}^{n-1}\lambda_{k}^{i}(x)^{2}\right)\langle R_{i}w,R_{i}w\rangle_{C_0(\Sigma)}(x) \\&\leq \sum_{i=0}^{n}\langle R_{i}w,R_{i}w\rangle_{C_0(\Sigma)}(x). 
\end{align*}
Therefore, integration against the measure $\mu$, for the inner product \eqref{sigmainn} we find
\begin{align*}\langle [\Tsla_{\Sigma}, R]w, [\Tsla_{\Sigma}, R]w\rangle_{\mu}\leq \sum_{i=0}^{n} \langle R_{i}w,R_{i}w\rangle_{\mu}\leq \left(\sum_{i=0}^{n} \|R_{i}\|^{2}\right)\langle w,w\rangle_{\mu}, \end{align*}
which proves boundedness of $[T_{\Sigma},R]$ on $C_0(\Sigma, L^{2}(\dH,\nu_0))\otimes_{C_{0}(\Sigma)}L^2(\Sigma, \mathcal{S}|_{\Sigma})$.
\end{proof}
 We write $\sigma:=ic(\nv)$, which satisfies $\sigma^{*}=\sigma$ and $\sigma^{2}=1$ and commutes with the operator  $S$ constructed in Theorem \ref{unbddrep}.
\begin{proposition}\label{reducetofull} 
The operator $v\sigma S v^{*}+ T_{\Sigma}:W_{\Sigma}\to C_0(\Sigma, L^{2}(\dH,\nu_0))\otimes_{C_{0}(\Sigma)}L^2(\Sigma, \mathcal{S}|_{\Sigma})$ is essentially self-adjoint.

\end{proposition}
\begin{proof} Write $t:=\Dsla_{\Sigma}$ and $s:=v\sigma S v^{*}$. On $W_{\Sigma}\subset \Dom s\cap \Dom t$ we can write
\begin{align*}st+ts=T_{\Sigma}v\sigma Sv^{*}+v\sigma Sv^{*}T_{\Sigma}&=\sigma(T_{\Sigma}vSv^*-vSv^*T_{\Sigma})\\ &= \sigma[T_{\Sigma},v\Delta v^{*}]+\sigma\rho[T_{\Sigma},vpv^{*}] + \sigma[T_{\Sigma},\rho]vpv^{*}.\end{align*}
It is straightforward to check that formula \eqref{vDv} holds for the isometry $v$ from Proposition \ref{tensorisos}. By Proposition \ref{comm} and Lemma \ref{framelemma}, the operators $[T_{\Sigma},v\Delta v^{*}], [T_{\Sigma},vpv^{*}]$ and $[T_{\Sigma},\rho]$ extend to bounded operators.
The unbounded multiplication operator $\rho$ commutes with the other operators involved. Now since
\[S^{2}=(-\Delta-\rho +2p\rho)^{2}=\Delta^2+2\Delta\rho+\rho^{2}\geq \rho^{2},\]
it follows that $\rho^{2}(1+S^{2})^{-1}\leq \rho^{2}(1+\rho^{2})^{-1}\leq 1$ and hence $\rho(vSv^{*}\pm i)^{-1}$ extends to a bounded operator. The operators $(s\pm i)^{-1}$ preserve the core $W_{\Sigma}$ and we have shown that $(st+ts)(s\pm i)^{-1}$ extends to a bounded operator. Thus $s$ and $t$ satisfy \cite[Definition A.1]{MR} (see also \cite{KaLe}) and by \cite[Theorem A.4]{MR} the sum operator $s+t$ is self-adjoint on $\Dom s\cap\Dom t$.
\end{proof}
For $\chi\in C_{c}(\h)$ and $f\in L^{2}(\dH,\nu_0)$ we denote by $f\cdot \chi $ the function $f\cdot \chi(x,\xi):=\chi(x)f(\xi)$. Using that $\chi$ has compact support, it is straightforward to check that $f\cdot\chi\in L^{2}(T_{1}\h,\nu_x)_{C_{0}(M)}$.
\begin{lemma}\label{connection} For any function $\chi\in C_{c}^{1}(\h)$ for which $\pi:\supp \chi\to M$ is injective and $f\in L^{2}(\dH,\nu_{0})$, the operator
$$\psi\mapsto v^{*}\Tsla_{\Sigma} v (f\cdot \chi \otimes \psi)-f\cdot \chi\otimes \Dsla_{\intN}\psi,$$
extends to a bounded operator $L^{2}(\intN,\mathcal{S})\to L^{2}_{\pi}(T_{1}\h,\nu_{x})\otimes _{C_{0}(M)} L^{2}(\intN,\mathcal{S})$.
\end{lemma}
\begin{proof} As in the proof of Proposition \ref{geoequiv}, there exists $\zeta\in C_{c}^{1}(M)$  such that $\chi=(\pi^{*}\zeta)|_{U}$. Then
\begin{align*} \Tsla_{\Sigma} v &(f\cdot \chi \otimes \psi)-v(f\cdot \chi\otimes \Dsla_{\intN}\psi)  = \Tsla_{\Sigma} (P^{n/2} f\cdot \chi \otimes \pi^{*}\psi )  -P^{n/2} f\cdot \chi \otimes \pi^{*}\Dsla_{\intN}\psi \\
&=[\Tsla_{\Sigma},P^{n/2}] f\cdot \chi\otimes \pi^{*} \psi + P^{n/2}f\cdot (\Tsla_{\Sigma}\chi \otimes \pi^{*}\psi -\chi \otimes \pi^{*}\Dsla_{\intN}\psi) 
\end{align*}

and we consider both summands separately. Writing $\lx_{k}$ for the vector fields  on $\intN$, satisfying $c(\x_{k})(x)=\pi^{*}c(\lx_{k})(x)$, Lemmas \ref{Poissoncomm} and \ref{framelemma} give the local expression
\begin{align*}[\Tsla_{\Sigma},P^{n/2}] f\cdot \chi\otimes\pi^{*}\psi (x)&= nP^{n/2}\sum_{k=0}^{n-1}\sum_{i}\lambda_{k}^{i}u_{i}f\cdot \chi\otimes c_{\Sigma}(\x_{k})\pi^{*}\psi(x)\\
&=nv\left(\sum_{k=0}^{n-1}\sum_{i}\lambda_{k}^{i}u_{i}f\cdot \chi\otimes c_{N}(\lx_{k})\psi\right)(x).
\end{align*}
This is shown to be a bounded operator as in the proof of Lemma \ref{framelemma}. For the second summand, 
\begin{align*}P^{n/2}f\cdot (\Tsla_{\Sigma}\pi^{*}\zeta \otimes \pi^{*}\psi -\pi^{*}\zeta \otimes \pi^{*}\Dsla_{N}\psi)|_{U}
&=\sum_{k=0}^{n-1}P^{n/2}f\cdot \pi^{*}(\partial_{\lx_{k}}\zeta)|_{U} \otimes (\pi^{*}c_{N}(\lx_{k})\psi )|_{U} \\
 &= v\left( \sum_{k=0}^{n-1} f\cdot \pi^{*}(\partial_{\lx_{k}}\zeta)|_{U} \otimes c_{N}(\lx_{k})\psi \right),\end{align*}
which defines a bounded operator as in Proposition \ref{geoequiv}. \end{proof}
\begin{theorem}\label{product} The triple $(C(\partial\h)\rtimes \Gamma, L^{2}_{\pi}(T_{1}\h,\nu_{x})\otimes_{C_{0}(M)}L^{2}(\intN,\mathcal{S}_{\intN}), \sigma S+v^{*}T_{\Sigma}v)$ is an unbounded Fredholm module representing the class $\partial[\Dsla_{\mathring{N}}]\in K^{1}(C(\partial \h)\rtimes\Gamma)$. 
\end{theorem}
\begin{proof} The operator $\sigma S+v^{*}T_{\Sigma}v$ is self-adjoint by Proposition \ref{reducetofull}. By Theorem \ref{unbddrep} and Proposition \ref{Tprops} the operator $\sigma S+ v^{*}T_{\Sigma}v$ has bounded commutators with functions $f\in C(\dH)$ and with group elements $u_{\gamma}$. By Lemma \ref{connection} and combining arguments in \cite[Lemma 4.3]{MR}, \cite[Theorem 6.7]{KaLe2} and \cite[Theorem 4.4]{MR}, it follows that the triple $(C(\partial\h)\rtimes \Gamma, L^{2}_{\pi}(T_{1}\h,\nu_{x})\otimes_{C_{0}(M)}L^{2}(\intN,\mathcal{S}_{\intN}), \sigma S+v^{*}T_{\Sigma}v)$ is a $K$-cycle representing the product $[\partial]\otimes[\Dsla_{\intN}]=\partial[\Dsla_{\mathring{N}}]$.
\end{proof}
It should be noted that under the isomorphism \eqref{alpha}, 
$$\alpha :C_{0}(\Sigma, L^2(\partial\h,\nu_{0}))\otimes_{C_{0}(\Sigma)} L^{2}(\Sigma,\mathcal{S})\to L^{2}(\dH,\nu_{0})\otimes L^{2}(\Sigma,\mathcal{S}),$$ the unbounded Fredholm module in Theorem \ref{product} admits a simple description. It can be represented on the module $L^{2}(\dH,\nu_0)\otimes L^{2}(\Sigma,\mathcal{S})$ using the $\Gamma$-representation
\[u_{\gamma} (f\otimes \psi )(\xi,x)=|\gamma'(\xi)|f(\xi\gamma)\otimes\psi(x\gamma),\]
and the operator $\sigma(\Delta_0\otimes 1+H + \rho F_{p})+1\otimes \Dsla_{\Sigma}$. However, proving that this operator is self-adjoint with compact resolvent requires the analysis presented above.


\end{document}